%% file: hp_entropy_sn.tex
\RequirePackage{fix-cm}
\documentclass{svjour3}                     
\smartqed  
\usepackage{graphicx}
\usepackage{amsmath}
\usepackage{amsfonts}
\usepackage{amssymb}
\usepackage{graphicx} 
\usepackage{xspace} 
\usepackage{bm} 
\usepackage[english]{babel}
\usepackage{cancel}
\usepackage{floatrow}
\usepackage{caption}
\usepackage{subcaption}
\usepackage{soul}

\usepackage{hyperref}
\hypersetup{
 colorlinks=true,
 citecolor=blue,
 linkcolor=blue,
 urlcolor=blue}
\usepackage{tikz}
\usetikzlibrary{matrix,backgrounds,shapes}
\usetikzlibrary{positioning}
\usetikzlibrary{fit}
\usetikzlibrary{plotmarks}
\usetikzlibrary{decorations.pathreplacing}
\usepackage{pgfplots}
\usetikzlibrary{shapes}
\usetikzlibrary{positioning,arrows}
\usetikzlibrary{fadings,shapes.arrows,shadows}
\usetikzlibrary{arrows.meta}
\definecolor{darkorange}{rgb}{1.0, 0.55, 0.0}
\definecolor{Royalblue}{rgb}{0.254,0.41,0.88}
\definecolor{royalblue}{rgb}{0.254,0.41,0.88}
\usepackage{xcolor}
\usepackage[titletoc,title]{appendix}
\input{def_v2}
\begin{document}

\title{Entropy Stable $h/p$-Nonconforming Discretization with the Summation-by-Parts Property for the Compressible Euler and Navier--Stokes Equations
}


\author{David C.\ Del Rey Fern\'andez  \and
        Mark H.\ Carpenter \and 
        Lisandro Dalcin \and 
        Stefano Zampini \and 
        Matteo Parsani
}


\institute{David C. Del Rey Fern\'andez \at
              National Institute of Aerospace and Computational AeroSciences Branch, NASA Langley Re-
search Center, Hampton, VA, United States \\
              \email{dcdelrey@gmail.com}           
           \and
           Mark H.\ Carpenter  \at
          NASA Langley Research Center, Hampton, VA, United States
          \email{mark.h.carpenter@nasa.gov} \\
          \and 
        Lisandro Dalcin, 
        Stefano Zampini, 
        Matteo Parsani \at 
        King Abdullah University of Science and Technology (KAUST), Computer Electrical and
Mathematical Science and Engineering Division (CEMSE), Extreme Computing Research Center (ECRC), Thuwal, Saudi Arabia 
\email{dalcinl@gmail.com, stefano.zampini@kaust.edu.sa, matteo.parsani@kaust.edu.sa}
}

\date{Received: date / Accepted: date}

\maketitle

\begin{abstract}
In this paper, the entropy conservative/stable algorithms presented by Del Rey Fern\'andez and coauthors \cite{Fernandez2019_p_euler,Fernandez2018_TM,Fernandez2019_p_ns} 
for the compressible Euler and Navier--Stokes equations on nonconforming $p$-refined/coarsened curvilinear grids 
is extended to $h/p$ refinement/coarsening. The main difficulty in developing nonconforming algorithms is 
the construction of appropriate coupling procedures across nonconforming interfaces. Here,
a computationally simple and efficient approach based upon using decoupled interpolation operators is utilized. 
The resulting scheme is entropy conservative/stable and element-wise 
conservative. 
Numerical simulations of the isentropic vortex and viscous shock propagation
confirm the entropy conservation/stability and accuracy properties of the method
  (achieving $\sim p+1$ convergence) which are  comparable to those of the original conforming scheme 
\cite{Carpenter2014,Parsani2016}. 
Simulations of the Taylor--Green
vortex at $Re=1,600$ and turbulent flow past a sphere at $Re=2,000$ show the robustness
and stability properties of the overall spatial discretization for unstructured grids.
Finally, to demonstrate the entropy conservation property of a 
fully-discrete explicit entropy stable algorithm with $h/p$ refinement/coarsening,
we present the time evolution of the entropy function obtained by simulating
the propagation of the isentropic vortex using a relaxation Runge--Kutta scheme.

\keywords{Nonconforming interfaces 
    \and  $h/p$ adaptation 
     \and Nonlinear entropy stability 
     \and Summation-by-parts 
     \and Simultaneous-approximation-terms 
     \and High-order accurate discretizations 
     \and Curved elements 
     \and Unstructured grid
}
\end{abstract}

\section{Introduction}
\label{intro}
This paper is the final installment in a set aimed at developing arbitrarily high-order, entropy stable, 
$h/p$-nonconforming schemes on curvilinear coordinates for the compressible Euler and 
Navier--Stokes equations~\cite{Friedrich2018,Fernandez2019_p_euler,Fernandez2018_TM,Fernandez2019_p_ns}. The efficient use of exascale concurrency on next generation hardware 
motivates the search for algorithms that are accurate and robust. Moreover, essential 
to efficiency is the ability to optimally use degrees of freedom through  $h$-, $p$-, and $r$-refinement/coarsening 
and communication hiding through dense compute kernels. High order methods are natural candidates for next generation hardware 
because they are accurate and their ratio of communications to local computations is usually
small. However, they have historically been limited 
by robustness issues, which is even more important as problem size and physics complexity increases.

When numerically solving partial differential equations (PDEs) it is 
imperative to find a bound on the growth rate of the 
solution, otherwise the possibility exists that the solution could grow arbitrarily fast. 
This upper bound can be established by ensuring that a numerical method is stable. 
For linear variable coefficient problems in arbitrary dimensions, a general 
and systematic approach to ensure stability is the energy method because it can be applied to the 
continuous as well as the semi-discrete model. The energy method becomes extremely powerful
when it is used in combination with the summation-by-parts (SBP) framework~\cite{Fernandez2014,Svard2014}
since it allows for the construction of provably stable schemes of any order.
SBP operators can be viewed as strong or weak form differentiation matrices that mimic integration-by-parts (IBP) and are endowed with a telescoping property, critical for provable stability. 
Via this property, SBP schemes (augmented with appropriate interface coupling procedures, 
\eg, simultaneous approximation terms (SATs)~\cite{Carpenter1994,Carpenter1999,Nordstrom1999,Nordstrom2001b,Carpenter2010,svard_entropy_stable_euler_wall_2014,Parsani2015b,Parsani2015}) reproduce, in a one-to-one manner, continuous stability proofs. 
Therefore they provide a road map for the development of provably stable 
semi-discrete or fully discrete algorithms (see, for instance, \cite{ranocha2019relaxation,Friedrich2019}).

For nonlinear problems a general and systematic approach for establishing stability has yet to be found. 
Nevertheless, for a certain class of PDEs 
progress has been made. For conservation laws, Tadmor~\cite{Tadmor1987entropy} constructed entropy conservative/stable low-order finite volume schemes 
that achieve entropy conservation by using two-point flux functions. When contracted with entropy variables, 
these schemes telescope the entropy flux. Entropy stability is then achieved by adding appropriate dissipation.
For a review of these ideas see, for instance, Tadmor~\cite{Tadmor2003}.

Tadmor's approach was extended to finite domains and arbitrary high order finite difference WENO schemes in the work of 
Fisher and coauthors who combined the SBP framework with Tadmor's two-point flux functions, 
resulting in entropy stable semi-discrete schemes~\cite{Fisher2012phd,Fisher2013,Fisher2013b}. 
This approach inherits all of the mechanics of linear SBP schemes for the imposition of boundary conditions and 
inter-element coupling and gives a systematic methodology for discretizing problems on complex 
domains; see, for instance, \cite{Carpenter2014,Parsani2015b,Parsani2015,Winters2015,Parsani2016,Gassner2016b,Yamaleev2017,Winters2017,Derigs2017,Wintermeyer2017,Crean2018,Chan2018b,Friedrich2019,Fernandez2019_staggered} and the reference
therein. An alternative method applicable to the compressible Euler equations~\cite{olsson1994,Yee2000,Sandham2002,Bjorn2018},
uses specially chosen entropy functions that result in a homogeneity property on the compressible Euler fluxes. Via this property, 
the Euler fluxes are split such that when contracted with the entropy variables, stability estimates 
result that are analogous in form to energy estimates obtained for linear PDEs. 

The objective of this paper is the extension of entropy stable $p$-nonconforming algorithm presented in 
Del Rey Fern\'andez~\etal~\cite{Fernandez2019_p_euler,Fernandez2018_TM,Fernandez2019_p_ns} for the compressible Euler and Navier--Stokes 
equations in curvilinear coordinates to arbitrary $h/p$-refinement/coarsening.  
The novel contributions of this paper are summarized as follows:
\begin{itemize}
\item A general and simple entropy conservative/stable nonconforming algorithm is proposed in curvilinear coordinates for 
the compressible Euler and Navier--Stokes equations that
\begin{itemize}
\item Enables a simple extension of the algorithm in~\cite{Fernandez2019_p_euler,Fernandez2018_TM,Fernandez2019_p_ns} that 
uses the same type of interface SAT and therefore allows code re-utilization
\item Results in the solution of the discrete geometric conservation laws (GCL) that is local to each element
\item Applies the metric approximation approach of Crean~\etal\cite{Crean2018} to arbitrary $h/p$-nonconforming elements
\item Ensures free-stream preservation by satisfying the discrete GCL conditions 
\item Is element-wise conservative
\end{itemize}
\item Numerical evidence is provided to demonstrate that the scheme retains the stability and accuracy 
properties of the conforming base scheme \cite{Carpenter2014,Parsani2016}
\end{itemize}

The paper is organized as follows. The notation is summarized in Section~\ref{sec:notation}. 
In Section~\ref{sec:hplin} the $h/p$ nonconforming algorithm is detailed in the context of the 
linear convection-diffusion equation. The required nonlinear mechanic necessary to 
extend the nonconforming algorithm to the compressible Navier--Stokes equations is 
described in the simple context of the Burgers' equation in Section~\ref{sec:Burgers}. 
Section~\ref{sec:NS} details the extension of the nonconforming algorithm to the  
compressible Navier--Stokes equations. The addition of interface dissipation that 
retains the provable properties of the base algorithm is discussed in Section~\ref{sec:dissipation}. 
Numerical experiments are detailed in Section~\ref{sec:num} while conclusions are drawn 
in Section~\ref{sec:conclusions}.
\section{Notation and definitions}\label{sec:notation}
The notation used herein is identical to that in~\cite{Fernandez2019_p_euler,Fernandez2018_TM,Fernandez2019_p_ns}; readers familiar with the notation 
can skip to Section~\ref{sec:hplin}. PDEs are discretized on cubes having Cartesian computational coordinates denoted by 
the triple $(\xil{1},\xil{2},\xil{3})$, where the physical coordinates are denoted by the triple 
$(\xm{1},\xm{2},\xm{3})$. Vectors are represented by lowercase bold font, for example $\bm{u}$, 
while matrices are represented using sans-serif font, for example, $\mat{B}$. Continuous 
functions on a space-time domain are denoted by capital letters in script font.  For example, 
\begin{equation*}
\fnc{U}\left(\xil{1},\xil{2},\xil{3},t\right)\in L^{2}\left(\left[\alphal{1},\betal{1}\right]\times
\left[\alphal{2},\betal{2}\right]\times\left[\alphal{3},\betal{3}\right]\times\left[0,T\right]\right)
\end{equation*}
represents a square integrable function, where $t$ is the temporal coordinate. The restriction of such 
functions onto a set of mesh nodes is denoted by lower case bold font. For example, the restriction of 
$\fnc{U}$ onto a grid of $\Nl{1}\times\Nl{2}\times\Nl{3}$ nodes is given by the vector
\begin{equation*}
\bm{u} = \left[\fnc{U}\left(\bxili{}{1},t\right),\dots,\fnc{U}\left(\bxili{}{N},t\right)\right]\Tr,
\end{equation*}
where, $N$ is the total number of nodes ($N\equiv\Nl{1}\Nl{2}\Nl{3}$) square brackets ($[]$) are used 
to delineate vectors and matrices as well as ranges for variables (the context will make clear which meaning is being used). Moreover, $\bm{\xi}$ is a vector of vectors 
constructed from the three vectors $\bxil{1}$, $\bxil{2}$, and $\bxil{3}$, which are 
vectors of size $\Nl{1}$, $\Nl{2}$, and $\Nl{3}$ and contain the coordinates of the mesh in 
the three computational directions, respectively. Finally, $\bxil{}$ is constructed as 
\begin{equation*}
\bxil{}(3(i-1)+1:3i)\equiv  \bxili{}{i}
\equiv\left[\bxil{1}(i),\bxil{2}(i),\bxil{3}(i)\right]\Tr,
\end{equation*}
where the notation $\bm{u}(i)$ means the $i\Th$ entry of the vector $\bm{u}$ and $\bm{u}(i:j)$ is the subvector 
constructed from $\bm{u}$ using the $i\Th$ through $j\Th$ entries (\ie, Matlab notation is used).

 Oftentimes, monomials are discussed and the following notation is used:
\begin{equation*}
\bxil{l}^{j} \equiv \left[\left(\bxil{l}(1)\right)^{j},\dots,\left(\bxil{l}(\Nl{l})\right)^{j}\right]\Tr,
\end{equation*}
and the 
convention that $\bxil{l}^{j}=\bm{0}$ for $j<0$ is used.

Herein, one-dimensional SBP operators are used to discretize derivatives. 
The definition of a one-dimensional SBP operator in the $\xil{l}$ direction, $l=1,2,3$, 
is~\cite{DCDRF2014,Fernandez2014,Svard2014}
\begin{definition}\label{SBP}
\textbf{Summation-by-parts operator for the first derivative}: A matrix operator with constant coefficients, 
$\DxiloneD{l}\in\mathbb{R}^{\Nl{l}\times\Nl{l}}$, is an SBP operator of degree $p$ approximating the derivative 
$\frac{\partial}{\partial \xil{l}}$ on the domain $\xil{l}\in\left[\alphal{l},\betal{l}\right]$ with nodal 
distribution $\bxil{l}$ having $\Nl{l}$ nodes, if 
\begin{enumerate}
\item $\DxiloneD{l}\bxil{l}^{j}=j\bxil{l}^{j-1}$, $j=0,1,\dots,p$;
\item $\DxiloneD{l}\equiv\left(\PxiloneD{l}\right)^{-1}\QxiloneD{l}$, where the norm matrix, 
$\PxiloneD{l}$, is symmetric positive definite;
\item $\QxiloneD{l}\equiv\left(\SxiloneD{l}+\frac{1}{2}\ExiloneD{l}\right)$, 
$\SxiloneD{l}=-\left(\SxiloneD{l}\right)\Tr$, $\ExiloneD{l}=\left(\ExiloneD{l}\right)\Tr$,  
$\ExiloneD{l} = \diag\left(-1,0,\dots,0,1\right)=\eNl{l}\eNl{l}\Tr-\eonel{l}\eonel{l}\Tr$, 
$\eonel{l}\equiv\left[1,0,\dots,0\right]\Tr$, and $\eNl{l}\equiv\left[0,0,\dots,1\right]\Tr$. 
\end{enumerate}
Thus, a degree $p$ SBP operator is 
one that differentiates exactly monomials up to degree $p$.
\end{definition}

In this work, one-dimensional SBP operators are extended to multiple dimensions 
using tensor products ($\otimes$).  The tensor product between the matrices $\mat{A}$ and $\mat{B}$ 
is given as $\mat{A}\otimes\mat{B}$. When referencing individual entries in a matrix the notation $\mat{A}(i,j)$ 
is used, which means 
the $i\Th$ $j\Th$ entry in the matrix $\mat{A}$.

The focus in this paper is exclusively on diagonal-norm SBP operators. Moreover, the same 
one-dimensional SBP operator are used in each direction, each operating on $N$ nodes. 
Specifically, diagonal-norm SBP operators constructed on the Legendre--Gauss--Lobatto (LGL) 
nodes are used, \ie, a discontinuous Galerkin collocated spectral element approach is utilized.

The physical domain $\Omega\subset\mathbb{R}^{3}$, 
with boundary $\Gamma\equiv\partial\Omega$ is partitioned into $K$ non-overlapping 
hexahedral elements. The domain of the $\kappa^{\text{th}}$ element is denoted by $\Ok$ and has 
boundary $\pOk$. Numerically, PDEs are solved in computational 
coordinates, 
where each $\Ok$ is locally transformed to $\Ohatk$, with boundary $\Ghat\equiv\pOhatk$, under the 
following assumption:
\begin{assume}\label{assume:curv}
Each element in physical space is transformed using 
a local and invertible curvilinear coordinate transformation that is compatible at 
shared interfaces, meaning that points in computational space on either side of a 
shared interface  mapped to the same physical location and therefore map back 
to the analogous location in computational space; this is the standard assumption 
that the curvilinear coordinate transformation is water tight.
\end{assume}
\section{An $h/p$-nonconforming algorithm: Linear convection-diffusion equation}\label{sec:hplin}
In this paper, the focus is on curvilinearly mapped elements with interfaces that 1) are conforming 
but have nonconforming nodal distributions, such as would arise in $p$-refinement, 2) 
elements that have nonconforming faces, such as would arise in $h$-refinement, and 3) arbitrary 
combinations of 1) and 2). The development of entropy stable $h/p$-refinement algorithm for the 
compressible Euler equations on Cartesian grids is detailed in~\cite{Friedrich2018}. The 
extension to curvilinear coordinates and $p$-refinement for the compressible Euler and Navier--Stokes equations is detailed in 
the series of papers~\cite{Fernandez2019_p_euler,Fernandez2019_p_ns,Fernandez2018_TM}, where an interface coupling technique is 
introduced that maintains, accuracy, discrete entropy conservation/stability, element-wise conservation and requires only local 
solves to approximate metric terms. Herein, the algorithm in~\cite{Fernandez2019_p_euler,Fernandez2019_p_ns,Fernandez2018_TM} 
is extended to allow for arbitrary $h/p$-refinement on unstructured grids for the compressible Euler and 
Navier--Stokes equations.
\subsection{Continuous and semi-discrete analysis}

A number of key technical difficulties that arise in developing a stable and conservative nonconforming discretization 
for the compressible Navier--Stokes equations are already present in the more simple context of the 
linear convection-diffusion equation. As a result, the proposed interface coupling procedure for both the inviscid and 
viscous terms are first presented for this simple linear scalar equation. In Cartesian coordinates, the 
linear convection-diffusion equation reads  
\begin{equation}\label{eq:cartconvectiondiffusion}
  \frac{\partial\U}{\partial t}+\sum\limits_{m=1}^{3}\frac{\partial\left(\am{m}\U\right)}{\partial\xm{m}}=
  \sum\limits_{m=1}^{3}\frac{\partial^{2}(\Bm{m}\U)}{\partial\xm{m}^{2}},
\end{equation}
where $(\am{m}\U)$ and $\frac{\partial(\Bm{m}\U)}{\partial\xm{m}}$ are the inviscid and viscous fluxes,
respectively. The symbols $\am{m}$ correspond to the constant components of the convection speed
whereas $\Bm{m}$ are the positive and constant diffusion coefficients. 
The stability of~\eqref{eq:cartconvectiondiffusion} can be determined via the energy method, which proceeds by 
multiplying~\eqref{eq:cartconvectiondiffusion} by the solution, ($\U$), 
and after using the product rule yields
\begin{equation}\label{eq:cartconvectionenergydiffusion1}
  \frac{1}{2}\frac{\partial\U^{2}}{\partial t}+
  \frac{1}{2}\sum\limits_{m=1}^{3}\frac{\partial\left(a_{m}\U^{2}\right)}{\partial\xm{m}}=
  \sum\limits_{m=1}^{3}\left\{\frac{\partial}{\partial\xm{m}}\left(\U\frac{\partial(\Bm{m}\U)}{\partial\xm{m}}\right)
  -\left(\frac{\partial(\Bm{m}\U)}{\partial\xm{m}}\right)^{2}\right\}.
\end{equation}
Integrating over the domain, $\Omega$, using integration by parts, and the Leibniz rule gives
\begin{equation}\label{eq:cartconvectionenergydiffusion2}
    \begin{split}
  &\frac{\mr{d}}{\mr{d}t}\int_{\Omega}\frac{\U^{2}}{2}\mr{d}\Omega=\\
  &\frac{1}{2}\sum\limits_{m=1}^{3}\left(\oint_{\Gamma}\left\{
    -\left(\am{m}\U^{2}\right)
    +2\U\frac{(\Bm{m}\U)}{\partial\xm{m}}
  \right\}\nxm{m}
    \mr{d}\Gamma-2\int_{\Omega}\left(\frac{\partial(\Bm{m}\U)}{\partial\xm{m}}\right)^{2}\mr{d}\Omega\right),
\end{split}
\end{equation}
where $\nxm{m}$ is the $m\Th$ component of the outward facing unit normal. What \Eq~\eqref{eq:cartconvectionenergydiffusion2} 
demonstrates is that the time rate of change of the norm of the solution, 
$\|\U\|^{2}\equiv\int_{\Omega}\U^{2}\mr{d}\Omega$, depends on surface flux integrals and a viscous dissipation term. 
This implies that, in combination with appropriate boundary conditions, 
\Eq~\eqref{eq:cartconvectionenergydiffusion2} results in a bound on the solution 
in terms of the data of the problem, and therefore a proof of stability. 
The SBP framework used
in this paper  
is mimetic of the above energy stability
analysis in a one-to-one fashion and results in similar stability statements for the semi-discrete equations.  

Derivatives are approximated using differentiation matrices that 
are defined in computational space. To do so, \Eq~\eqref{eq:cartconvectiondiffusion} is transformed using the 
curvilinear coordinate transformation $\xm{m}=\xm{m}\left(\xil{1},\xil{2},\xil{3}\right)$. Thus, on the 
$\kappa\Th$ element, the $\xm{m}$ derivatives are expanded using the chain rule as
\begin{equation*}
  \frac{\partial}{\partial\xm{m}}=
  \sum\limits_{l=1}^{3}\frac{\partial\xil{l}}{\partial\xm{m}}\frac{\partial}{\partial\xil{l}},\quad
  \frac{\partial^{2}}{\partial\xm{m}^{2}}=
  \sum\limits_{l,a=1}^{3}\frac{\partial\xil{l}}{\partial\xm{m}}
  \frac{\partial}{\partial\xil{l}}\left(
  \frac{\partial\xil{a}}{\partial \xm{m}}\frac{\partial}{\partial\xil{a}}  
  \right).
\end{equation*} 
Multiplying by the metric Jacobian, ($\Jk$), \Eq~\eqref{eq:cartconvectiondiffusion} becomes 
 \begin{equation}\label{eq:convectiondiffusionchain}
  \Jk\frac{\partial\U}{\partial t}
 +\sum\limits_{l,m=1}^{3}\Jk\frac{\partial\xil{l}}{\partial\xm{m}}
  \frac{\partial \left(a_{m}\U\right)}{\partial\xil{l}}=\sum\limits_{l,a,m=1}^{3}
  \Jk\frac{\partial\xil{l}}{\partial\xm{m}}\frac{\partial}{\partial\xil{l}}
  \left(\frac{\partial\xil{a}}{\partial \xm{m}}
  \frac{\partial(\Bm{m}\U)}{\partial\xil{a}}
  \right).
\end{equation}

Herein, Eq.~\eqref{eq:convectiondiffusionchain} is referenced as the chain rule form of Eq.~\eqref{eq:cartconvectiondiffusion}. 
Bringing the metric terms, $\Jdxildxm{l}{m}$, inside the derivative and using the product rule gives
 \begin{equation}\label{eq:convectiondiffusionstrong1}
  \begin{split}
  \Jk\frac{\partial\U}{\partial t}+\sum\limits_{l,m=1}^{3}
  \frac{\partial}{\partial\xil{l}}\left(\Jdxildxm{l}{m}a_{m}\U\right)
-&\sum\limits_{l,m=1}^{3}a_{m}\U\frac{\partial}{\partial\xil{l}}\left(\Jdxildxm{l}{m}\right)
 =\\
 \sum\limits_{l,a,m=1}^{3}
  \frac{\partial}{\partial\xil{l}}
  \left(\Jk\frac{\partial\xil{l}}{\partial\xm{m}}\frac{\partial\xil{a}}{\partial \xm{m}}
  \frac{\partial(\Bm{m}\U)}{\partial\xil{a}}\right)
-&\sum\limits_{l,a,m=1}^{3}
\frac{\partial\xil{a}}{\partial \xm{m}}
  \frac{\partial(\Bm{m}\U)}{\partial\xil{a}}
  \frac{\partial}{\partial\xil{l}}\left( \Jk\frac{\partial\xil{l}}{\partial\xm{m}}\right).
  \end{split}
\end{equation}
The last terms on the left- and right-hand sides of~\eqref{eq:convectiondiffusionstrong1} are zero via the 
GCL relations
\begin{equation}\label{eq:GCL}
    \sum\limits_{l=1}^{3}\frac{\partial}{\partial\xil{l}}\left(\Jdxildxm{l}{m}\right)=0,\quad m=1,2,3,
\end{equation}
leading to the strong conservation form of the convection-diffusion equation in curvilinear coordinates:
 \begin{equation}\label{eq:convectiondiffusionstrong}
  \Jk\frac{\partial\U}{\partial t}+\sum\limits_{l,m=1}^{3}
  \frac{\partial}{\partial\xil{l}}\left(\Jdxildxm{l}{m}a_{m}\U\right)= \sum\limits_{l,a,m=1}^{3}
  \frac{\partial}{\partial\xil{l}}
  \left(\Jk\frac{\partial\xil{l}}{\partial\xm{m}}\frac{\partial\xil{a}}{\partial \xm{m}}
  \frac{\partial(\Bm{m}\U)}{\partial\xil{a}}\right).
\end{equation}

The $h$-refinement procedure proceeds by subdividing the computational domain of parent elements; where children elements 
inherit the curvilinear coordinate transformation of the parent element. It is therefore convenient to introduce 
two computational coordinates:
\begin{itemize}
\item $\xil{l}$, which is the mapping from the child element to physical space, \ie, the computational coordinate system of the $\kappa\Th$ element,
\item $\xilhat{l}$, which is the mapping from the parent element to the physical space. 
\end{itemize}
The mapping from children to parent elements is rectilinear. Thus, 
assuming that the child element has a computational domain of $[-1,1]^{3}$, this transformation 
for the $\kappa\Th$ element is given by
\begin{equation}\label{eq:lintransformation}
\xil{l}=\frac{2}{\Deltalk{l}{\kappa}}\xilhat{l}-\frac{\left(\xilHkhat{l}{\kappa}+\xilLkhat{l}{\kappa}\right)}{\Deltalk{l}{\kappa}},\quad \Deltalk{l}{\kappa}\equiv\xilHkhat{l}{\kappa}-\xilLkhat{l}{\kappa}, \quad l=1,2,3,
\end{equation}
where $\xilHkhat{l}{\kappa}$ and $\xilLkhat{l}{\kappa}$ are the largest and smallest extent of the $\xil{l}$ coordinate in the coordinate system of the parent element ($\xilhat{l}$). Using Eq.~\eqref{eq:lintransformation} 
the Jacobian and metrics are recast in terms of the Jacobian, $\Jkhat$, and metrics terms, $\Jdxildxmhat{l}{m}$, 
of the parent element. This step results in 
\begin{equation}\label{eq:newmetrics}
\frac{\partial\xil{l}}{\partial \xm{m}}=\frac{2}{\Deltalk{l}{\kappa}}\frac{\partial\xilhat{l}}{\partial\xm{m}},\qquad
\Jk=\frac{\Deltalk{1}{\kappa}\Deltalk{2}{\kappa}\Deltalk{3}{\kappa}}{8}\Jkhat,\qquad
\Jdxildxm{l}{m}=\frac{\Deltalk{1}{\kappa}\Deltalk{2}{\kappa}\Deltalk{3}{\kappa}}{4\Deltalk{l}{\kappa}}\Jdxildxmhat{l}{m}.
\end{equation}

Inserting \Eq~\eqref{eq:newmetrics} into \Eq~\eqref{eq:convectiondiffusionchain} and multiplying by $8/\left(\Deltalk{1}{\kappa}\Deltalk{2}{\kappa}\Deltalk{3}{\kappa}\right)$ gives 
 \begin{equation}\label{eq:convectiondiffusionchainnew}
  \Jkhat\frac{\partial\U}{\partial t}
 +\sum\limits_{l,m=1}^{3}\frac{2}{\Deltalk{l}{\kappa}}\Jdxildxmhat{l}{m}
  \frac{\partial \left(a_{m}\U\right)}{\partial\xil{l}}=\sum\limits_{l,a,m=1}^{3}
  \frac{4}{\Deltalk{l}{\kappa}\Deltalk{a}{\kappa}}\Jdxildxmhat{l}{m}
  \left(\frac{\partial\xilhat{a}}{\partial \xm{m}}
  \frac{\partial(\Bm{m}\U)}{\partial\xil{a}}
  \right).
\end{equation}
Similarly, \Eq~\eqref{eq:convectiondiffusionstrong1} is transformed to 
 \begin{equation}\label{eq:convectiondiffusionstrong1new}
  \begin{split}
  &\Jkhat\frac{\partial\U}{\partial t}+\sum\limits_{l,m=1}^{3}
  \frac{2}{\Deltalk{l}{\kappa}}\frac{\partial}{\partial\xil{l}}\left(\Jdxildxmhat{l}{m}a_{m}\U\right)
-\sum\limits_{l,m=1}^{3}\frac{2}{\Deltalk{l}{\kappa}}a_{m}\U\frac{\partial}{\partial\xil{l}}\left(\Jdxildxmhat{l}{m}\right)
 =\\
 &\sum\limits_{l,a,m=1}^{3}
  \frac{4}{\Deltalk{l}{\kappa}\Deltalk{a}{\kappa}}\frac{\partial}{\partial\xil{l}}
  \left(\Jdxildxmhat{l}{m}\frac{\partial\xilhat{a}}{\partial\xm{m}}
  \frac{\partial(\Bm{m}\U)}{\partial\xil{a}}\right)\\
-&\sum\limits_{l,a,m=1}^{3}
 \frac{4}{\Deltalk{l}{\kappa}\Deltalk{a}{\kappa}}\frac{\partial\xilhat{a}}{\partial \xm{m}}
  \frac{\partial(\Bm{m}\U)}{\partial\xil{a}}
  \frac{\partial}{\partial\xil{l}}\left( \Jdxildxmhat{l}{m}\right),
  \end{split}
\end{equation}
and \Eq~\eqref{eq:convectiondiffusionstrong} is transformed to
 \begin{equation}\label{eq:convectiondiffusionstrongnew}
  \begin{split}
  &\Jkhat\frac{\partial\U}{\partial t}+\sum\limits_{l,m=1}^{3}
  \frac{2}{\Deltalk{l}{\kappa}}\frac{\partial}{\partial\xil{l}}\left(\Jdxildxmhat{l}{m}a_{m}\U\right)= \\&\sum\limits_{l,a,m=1}^{3}
  \frac{4}{\Deltalk{l}{\kappa}\Deltalk{a}{\kappa}}\frac{\partial}{\partial\xil{l}}
  \left(\Jdxildxmhat{l}{m}\frac{\partial\xilhat{a}}{\partial \xm{m}}
  \frac{\partial(\Bm{m}\U)}{\partial\xil{a}}\right).
  \end{split}
\end{equation}
Directly discretizing \Eq~\eqref{eq:convectiondiffusionstrongnew} leads to semi-discrete schemes that are not guaranteed to 
be stable. Instead, a well known approach is to use a canonical splitting of the inviscid terms which is 
constructed by using one half of the inviscid terms in~\eqref{eq:convectiondiffusionchainnew} 
and one half of the inviscid terms in~\eqref{eq:convectiondiffusionstrong1new} 
(see, for instance, \cite{Carpenter2015}). On the other hand,  
the viscous terms are treated in strong conservation form. This process results in
\begin{equation}\label{eq:convectiondiffusionsplit}
  \begin{split}
  &\Jkhat\frac{\partial\U}{\partial t}+\frac{1}{2}\sum\limits_{l,m=1}^{3}\frac{2}{\Deltalk{l}{\kappa}}\left\{
    \frac{\partial}{\partial\xil{l}}\left(\Jdxildxmhat{l}{m}a_{m}\U\right)+
     \Jdxildxmhat{l}{m}\frac{\partial}{\partial\xil{l}}\left(a_{m}\U\right)
    \right\}\\&-\frac{1}{2}\sum\limits_{l,m=1}^{3}\left\{
    a_{m}\U\frac{\partial}{\partial\xil{l}}\left(\Jdxildxm{l}{m}\right)\right\}=
\sum\limits_{l,a,m=1}^{3}\frac{4}{\Deltalk{l}{\kappa}\Deltalk{a}{\kappa}}
  \frac{\partial}{\partial\xil{l}}
  \left(\Jdxildxmhat{l}{m}\frac{\partial\xilhat{a}}{\partial \xm{m}}
  \frac{\partial(\Bm{m}\U)}{\partial\xil{a}}\right),
  \end{split}
\end{equation}
where the last set of terms on the left-hand side are zero by the GCL conditions~\eqref{eq:GCL}. 
Now, consider discretizing Eq.~\eqref{eq:convectiondiffusionsplit} by using 
the following differentiation matrices:
\begin{equation*}
\Dxil{1}\equiv\mat{D}^{(1D)}\otimes\Imat{N}\otimes\Imat{N},\quad
\Dxil{2}\equiv\Imat{N}\otimes\mat{D}^{(1D)}\otimes\Imat{N},\quad
\Dxil{3}\equiv\Imat{N}\otimes\Imat{N}\otimes\mat{D}^{(1D)},
\end{equation*} 
where $\Imat{N}$ is an $N\times N$ identity matrix. The diagonal 
matrices containing the metric Jacobian and metric terms along their diagonals, respectively, are defined as follows:
\begin{equation*}
  \begin{split}
  &\matJkhat{\kappa}\equiv\diag\left(\Jkhat\left(\bmxi{1}\right),\dots,\Jkhat\left(\bmxi{\Nl{\kappa}}\right)\right),\\
  &\matAlmkhat{l}{m}{\kappa}\equiv\diag\left(\Jdxildxmhat{l}{m}(\bmxi{1}),\dots,
  \Jdxildxmhat{l}{m}(\bmxi{\Nl{\kappa}})\right),
  \end{split}
\end{equation*}
where $\Nl{\kappa}\equiv N^{3}$ is the total number of nodes in the element $\kappa$.
With these matrices, the discretization of~\eqref{eq:convectiondiffusionsplit} on the $\kappa\Th$ element is given as
\begin{equation}\label{eq:convectionsplitdisc}
  \begin{split}
  &\matJkhat{\kappa}\frac{\mr{d}\uk}{\mr{d} t}+\frac{1}{2}\sum\limits_{l,m=1}^{3}
  \frac{2}{\Deltalk{l}{\kappa}}a_{m}\left\{\Dxil{l}\matAlmkhat{l}{m}{\kappa}+\matAlmkhat{l}{m}{\kappa}\Dxil{l}\right\}\uk
  \\&-\frac{1}{2}\sum\limits_{l,m=1}^{3}\left\{
    \frac{2}{\Deltalk{l}{\kappa}}a_{m}\diag\left(\uk\right)\Dxil{l}\matAlmkhat{l}{m}{\kappa}\ones{\kappa}\right\}=\\
    &\sum\limits_{l,m,a=1}^{3}\frac{4}{\Deltalk{l}{\kappa}\Deltalk{a}{\kappa}}\Bm{m}\Dxil{l}^{\kappa}\matJkhat{\kappa}^{-1}\matAlmkhat{l}{m}{\kappa}\matAlmkhat{a}{m}{\kappa}\Dxil{a}^{\kappa}\uk,
  \end{split}
\end{equation}
where $\ones{\kappa}$ is a vector of ones of the size of the number of nodes on 
the $\kappa\Th$ element and the SATs have been dropped as they are not important for the current analysis. 
In the same way as in the continuous case, the semi-discrete equations have an associated set of discrete GCL conditions
\begin{equation}\label{eq:discGCLconvection}
\sum\limits_{l=1}^{3}
    \frac{2}{\Deltalk{l}{\kappa}}\Dxil{l}\matAlmkhat{l}{m}{\kappa}\ones{\kappa}=\bm{0}, \quad m = 1,2,3,
\end{equation}
that if satisfied, lead to the following telescoping and therefore provably stable semi-discrete form:
\begin{equation}\label{eq:convectionsplitdisctele}
  \begin{split}
  &\matJkhat{\kappa}\frac{\mr{d}\uk}{\mr{d} t}+\frac{1}{2}\sum\limits_{l,m=1}^{3}
  \frac{2}{\Deltalk{l}{\kappa}}a_{m}\left\{\Dxil{l}\matAlmkhat{l}{m}{\kappa}+\matAlmkhat{l}{m}{\kappa}\Dxil{l}\right\}\uk=\\
    &\sum\limits_{l,m,a=1}^{3}\frac{4}{\Deltalk{l}{\kappa}\Deltalk{a}{\kappa}}\Bm{m}\Dxil{l}^{\kappa}\matJkhat{\kappa}^{-1}\matAlmkhat{l}{m}{\kappa}\matAlmkhat{a}{m}{\kappa}\Dxil{a}^{\kappa}\uk.
  \end{split}
\end{equation}
How to construct metrics that satisfy the discrete GCL conditions \eqref{eq:discGCLconvection} will be detailed later in the paper and is 
one of the major contributions of this work. In the next subsection, attention is focused on the construction of 
appropriate interface coupling procedures that retain the stability (telescoping) properties of 
scheme~\eqref{eq:convectionsplitdisctele} across $h/p$ nonconforming elements (the pure $p$ nonconforming 
case is detailed in~\cite{Fernandez2019_p_ns,Fernandez2019_p_euler,Fernandez2018_TM}).
\subsection{The nonconforming interface}\label{sec:cdnon}
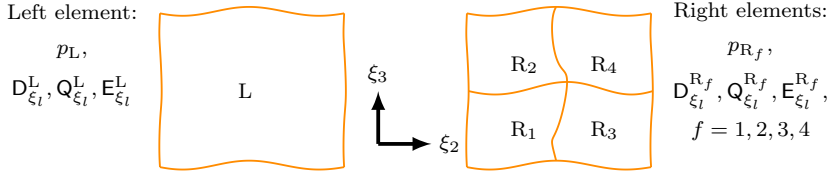
\begin{figure}[!t]
\centering
\input{non_conforming_face.tex}
\caption{Generic surface used to describe the construction of various quantities as well as to simplify 
the analysis of the semi-discrete schemes.}
\label{fig:genint}
\end{figure}
To simplify both the analysis as well as the presentation of the semi-discrete scheme discussed in 
this paper, it is convenient to focus the attention only on a shared interface between two elements, one of which is $h/p$-refined. Without 
loss of generality, five elements are considered that have aligned computational coordinates and 
adjoin along a vertical interface; see Fig.~\ref{fig:genint}. The focus is on nonconformities that arise from both $h$ refinement/coarsening 
as well as local approximations with differing polynomial degrees, as would result from $p$-refinement/coarsening. 
Thus, the generic example, Fig.~\ref{fig:genint}, that will be discussed 
considers a left element having polynomial degree $\pL$ 
and a set of four right elements having polynomial degrees $\pRf{f}$, $f=1,2,3,4$, possibly all having differing degree and 
not equal to $\pL$ (\ie, they originate from a conforming right element that has been $h$-refined and then $p$-refined/coarsened). Therefore, the 
contributions from the left element are identified with subscript or superscript $\mathrm{L}$ and 
similarly for the right elements with subscripts or superscripts $\Rf{f}$; see Fig.~\ref{fig:genint}. 

The analysis proceeds by developing macro matrix differentiation operators over the five elements, 
i.e., composed of elements $\rmL$ and $\Rf{f}$, $f=1,2,3,4$, and then 
determining the required modifications/restrictions so that the resulting operators have the SBP property.
A naive construction, in the computational coordinates of the parent elements, would be the following 
operators:
\begin{equation}\label{eq:naive}
  \begin{split}
    &\DtildehatH{l}\equiv\left[
  \begin{array}{cccc}
    \Dxilhat{l}^{\rmL}&\\
    &\Dxilhat{l}^{\Rf{1}}\\
&&\ddots\\
&&&\Dxilhat{l}^{\Rf{4}}
  \end{array}  
  \right], 
  \quad l=1,2,3,
\end{split}   
\end{equation}
where the element-wise components of the SBP operators, for example $\Dxilhat{1}^{\rmL}$, are constructed as 
\begin{equation*}
\Dxilhat{1}^{\rmL}\equiv\frac{2}{\Deltalk{1}{\rmL}}\Dxil{1},\quad\Dxil{1}\equiv\DoneD_{\rmL}\otimes\Imat{\rmL}\otimes\Imat{\rmL},
\end{equation*}
\begin{equation*}
\M_{\rmL}\equiv\frac{\Deltalk{1}{\rmL}\Deltalk{2}{\rmL}\Deltalk{3}{\rmL}}{8}\PoneD_{\rmL}\otimes\PoneD_{\rmL}\otimes\PoneD_{\rmL},
\end{equation*}
\begin{equation*}
\Qxilhat{1}^{\rmL}\equiv\frac{\Deltalk{2}{\rmL}\Deltalk{3}{\rmL}}{4}\QoneD_{\rmL}\otimes\PoneD_{\rmL}\otimes\PoneD_{\rmL},
\end{equation*}
\begin{equation*}
\Exilhat{1}^{\rmL}\equiv\frac{\Deltalk{2}{\rmL}\Deltalk{3}{\rmL}}{4}\left\{\eNl{}^{\rmL}\left(\eNl{}^{\rmL}\right)\otimes\PoneD_{\rmL}\otimes\PoneD_{\rmL}-
\eonel{}^{\rmL}\left(\eonel{}^{\rmL}\right)\otimes\PoneD_{\rmL}\otimes\PoneD_{\rmL}\right\},
\end{equation*}
where $\Imat{\rmL}$ is an identity matrix of size $\Nl{\rmL}^{1/3}\times\Nl{\rmL}^{1/3}$ and 
$\Nl{\rmL}$ is the total number of nodes in element $\rmL$.

While the $\DtildehatH{2}$ and $\DtildehatH{3}$ macro element operators are by construction SBP 
operators (\ie, they telescope to the boundaries), the $\DtildehatH{1}$ is not an SBP operator. 
Moreover, the $\DtildehatH{1}$ operator does not have any coupling between the elements. 
By using appropriate interface coupling, the $\DtildehatH{1}$ operator can be modified so 
that the result is a macro element SBP operator. To accomplish this, interpolation operators are needed 
that interpolate information from elements $\Rf{f}$ to element $\rmL$ and vice versa.  For simplicity, 
the interpolation operators use only tensor product surface information from the adjoining interface surface.

With this background, general matrix difference operators between the five elements are constructed as 
\begin{equation}
  \Dtildehatl{l}=\Mtilde^{-1}\Qtildehatl{l}=\Mtilde^{-1}\left(\Stildehatl{l}+\frac{1}{2}\Etildehatl{l}\right).
\end{equation}
Focusing on the direction orthogonal to the interface ($\xil{1}$) the relevant matrices are given by 
\begin{equation}\label{eq:marcoD} 
  \begin{split}
  &\Mtilde\equiv\diag\left(\M_{\rmL},\M_{\Rf{1}},\M_{\Rf{2}},\M_{\Rf{3}},\M_{\Rf{4}}\right),\\
  &\Stildel{1}\equiv
  \left[
  \begin{array}{cccccc}
    \Sxilhat{1}^{\rmL}&\mat{\tilde{S}}_{12}&\mat{\tilde{S}}_{13}&\mat{\tilde{S}}_{14}&\mat{\tilde{S}}_{15}\\
    \mat{\tilde{S}}_{21}&\Sxilhat{1}^{\Rf{1}}\\
    \mat{\tilde{S}}_{31}&&\Sxilhat{1}^{\Rf{2}}\\
    \mat{\tilde{S}}_{41}&&&\Sxilhat{1}^{\Rf{3}}\\
    \mat{\tilde{S}}_{51}&&&&\Sxilhat{1}^{\Rf{4}}\\
  \end{array}
  \right],\\
  &\mat{\tilde{S}}_{12}\equiv\frac{\Deltalk{2}{\rmL}\Deltalk{3}{\rmL}}{4}
  \eNl{}^{\rmL}\left(\eonel{}^{\Rf{1}}\right)\Tr\PoneD_{\rmL}\IRftoLoneD{1}{2}\otimes
  \PoneD_{\rmL}\IRftoLoneD{1}{3},\\
  &\mat{\tilde{S}}_{21}\equiv-\frac{\Deltalk{2}{\Rf{1}}\Deltalk{3}{\Rf{1}}}{4}
  \eonel{}^{\Rf{1}}\left(\eNl{}^{\rmL}\right)\Tr\PoneD_{\Rf{1}}\ILtoRfoneD{1}{2}\otimes
  \PoneD_{\Rf{1}}\ILtoRfoneD{1}{3},\\
  &\mat{\tilde{S}}_{13}\equiv\frac{\Deltalk{2}{\rmL}\Deltalk{3}{\rmL}}{4}\eNl{}^{\rmL}
  \left(\eonel{}^{\Rf{2}}\right)\Tr\PoneD_{\rmL}\IRftoLoneD{2}{2}\otimes
  \PoneD_{\rmL}\IRftoLoneD{2}{3},\\
  &\mat{\tilde{S}}_{31}\equiv-\frac{\Deltalk{2}{\Rf{2}}\Deltalk{3}{\Rf{2}}}{4}\eonel{}^{\Rf{2}}
  \left(\eNl{}^{\rmL}\right)\Tr\PoneD_{\Rf{2}}\ILtoRfoneD{2}{2}\otimes
  \PoneD_{\Rf{2}}\ILtoRfoneD{2}{3},\\
  &\mat{\tilde{S}}_{14}\equiv\frac{\Deltalk{2}{\rmL}\Deltalk{3}{\rmL}}{4}\eNl{}^{\rmL}
  \left(\eonel{}^{\Rf{3}}\right)\Tr\PoneD_{\rmL}\IRftoLoneD{3}{2}\otimes
  \PoneD_{\rmL}\IRftoLoneD{3}{3},\\
  &\mat{\tilde{S}}_{41}\equiv-\frac{\Deltalk{2}{\Rf{3}}\Deltalk{3}{\Rf{3}}}{4}\eonel{}^{\Rf{3}}
  \left(\eNl{}^{\rmL}\right)\Tr\PoneD_{\Rf{3}}\ILtoRfoneD{3}{2}\otimes
  \PoneD_{\Rf{3}}\ILtoRfoneD{3}{3},\\
  &\mat{\tilde{S}}_{15}\equiv\frac{\Deltalk{2}{\rmL}\Deltalk{3}{\rmL}}{4}\eNl{}^{\rmL}
  \left(\eonel{}^{\Rf{4}}\right)\Tr\PoneD_{\rmL}\IRftoLoneD{4}{2}\otimes
  \PoneD_{\rmL}\IRftoLoneD{5}{3},\\
  &\mat{\tilde{S}}_{51}\equiv-\frac{\Deltalk{2}{\rmL}\Deltalk{3}{\Rf{4}}}{4}\eonel{}^{\Rf{4}}
  \left(\eNl{}^{\rmL}\right)\Tr\PoneD_{\Rf{4}}\ILtoRfoneD{5}{2}\otimes
  \PoneD_{\Rf{5}}\ILtoRfoneD{5}{3},\\
  &\Etildel{1}\equiv
  \left[
  \begin{array}{cccc}
 \tilde{\mat{E}}_{11}\\
    &\tilde{\mat{E}}_{22}\\
    &&\ddots\\
    &&&\tilde{\mat{E}}_{55}
  \end{array}
  \right],\\ 
  &\tilde{\mat{E}}_{11}\equiv    -\frac{\Deltalk{2}{\rmL}\Deltalk{3}{\rmL}}{4}\eonel{}^{\rmL}\left(\eonel{}^{\rmL}\right)\Tr\otimes\PoneD_{\rmL}\otimes\PoneD_{\rmL},\\
  &\tilde{\mat{E}}_{22} \equiv \frac{\Deltalk{2}{\Rf{1}}\Deltalk{3}{\Rf{1}}}{4}\eNl{}^{\Rf{1}}\left(\eNl{}^{\Rf{1}}\right)\Tr\otimes\PoneD_{\Rf{1}}\otimes\PoneD_{\Rf{1}},\\
  &\tilde{\mat{E}}_{55} \equiv\frac{\Deltalk{2}{\Rf{4}}\Deltalk{3}{\Rf{4}}}{4}\eNl{}^{\Rf{4}}\left(\eNl{}^{\Rf{4}}\right)\Tr\otimes\PoneD_{\Rf{4}}\otimes\PoneD_{\Rf{4}}, \\
\end{split}
\end{equation}
and $\IRftoLoneD{f}{l}$ and $\ILtoRfoneD{f}{l}$ are one-dimensional interpolation operators, in the $l$ direction, from the $\Rf{f}$ element
to the $\rmL$ element and vice versa. 

A necessary constraint that the SBP formalism places on $\Dtildel{1}$ is skew-symmetry of the 
$\Stildel{1}$ matrices. The block-diagonal matrices in $\Stildel{1}$ are 
already skew-symmetric but the off diagonal blocks are not. 
Thus, it is necessary to satisfy the following
conditions:
\begin{equation*}
  \tilde{\mat{S}}_{12}=-\tilde{\mat{S}}_{21}\Tr,\quad
  \tilde{\mat{S}}_{13}=-\tilde{\mat{S}}_{31}\Tr,\quad
  \tilde{\mat{S}}_{14}=-\tilde{\mat{S}}_{41}\Tr.
\end{equation*}
This implies that the interpolation operators are related to each other as follows:
\begin{equation*}
  \IRftoLoneD{f}{l}=\frac{\Deltalk{l}{\Rf{f}}}{\Deltalk{l}{\rmL}}\left(\PoneD_{\rmL}\right)^{-1}\left(\ILtoRfoneD{f}{l}\right)\Tr\PoneD_{\Rf{1}},
  \quad l= 1,2,\quad f=1,2,3,4.
\end{equation*}
This property is denoted as the SBP preserving property because it leads to 
a macro element differentiation matrix that is an SBP operator. 
The interpolation operators from the left element to the right elements is constructed using an $L_{2}$ 
projection approach such that
\begin{equation*}
\begin{split}
&\ILtoRfoneD{f}{l}\equiv\left(\mat{M}_{\mathrm{R}_{f}}^{\xilhat{l}}\right)^{-1}\mat{M}_{\mathrm{Lto}\mathrm{R}_{f}}^{\xilhat{l}},\\
&\mat{M}_{\mathrm{R}_{f}}^{\xilhat{l}}(i,j) \equiv \int_{\xilLkhat{l}{\mathrm{R}_{f}}}^{\xilHkhat{l}{\mathrm{R}_{f}}}
\fnc{L}_{\xilhat{l},\mathrm{R}_{f}}^{(i)}\fnc{L}_{\xilhat{l},\mathrm{R}_{f}}^{(j)}\mr{d}\xilhat{l},\quad i,j = 1,\dots,\left(\pRf{f}+1\right),\\
&\mat{M}_{\mathrm{Lto}\mathrm{R}_{f}}^{\xilhat{l}}(i,j)\equiv \int_{\xilLkhat{l}{\mathrm{R}_{f}}}^{\xilHkhat{l}{\mathrm{R}_{f}}}
\fnc{L}_{\xilhat{l},\mathrm{R}_{f}}^{(i)}\fnc{L}_{\xilhat{l},\mathrm{L}}^{(j)}\mr{d}\xilhat{l},\quad i= 1,\dots,\left(\pRf{f}+1\right)\quad
j= 1,\dots,\left(\pL+1\right),
\end{split}
\end{equation*}
where $\fnc{L}_{\xilhat{l},\mathrm{R}_{f}}^{(i)}$ and $\fnc{L}_{\xilhat{l},\mathrm{L}}^{(j)}$ are the $i^{\mathrm{th}}$ and $j^{\mathrm{th}}$ 
Lagrange basis functions constructed from the nodes of element $\Rf{f}$ and $\mathrm{L}$ in the parent elements coordinates, respectively.
Now theorems on the accuracy of the interpolation operators are presented
\begin{thrm}\label{thrm:accILtoRf2}
The interpolation operator $\ILtoRfoneD{f}{l}$ is of degree $\min\left(\Nl{\rmL}^{1/3}-1,\Nl{\Rf{f}}^{1/3}-1\right)$.
\end{thrm}
\begin{proof}
The proof is standard and is not included for brevity. It follows by expanding out the matrices and taking advantage of the interpolating property of the 
Lagrangian basis functions. 
\end{proof}

The interpolation operators $\IRftoLoneD{f}{l}$ individually are not polynomial exact, but rather, their 
combination is. 
\begin{thrm}\label{thrm:accIRftoL2}
The combined interpolation from the right elements to the left element is of degree 
$\min\left(\Nl{\rmL}-2,\Nl{\Rf{1}}-2,\dots,\Nl{\Rf{4}}-2\right)$, if the norms are suboptimal, \ie, degree $2p-1$, 
otherwise 
it is of degree $\min\left(\Nl{\rmL}-1,\Nl{\Rf{1}}-1,\dots,\Nl{\Rf{4}}-1\right)$. In the five-element example used herein, 
the combined interpolation operator, acting on some function $\fnc{U}$, is
\begin{equation*}
\IRftoL{1}\bm{u}_{\Rf{1}}+\IRftoL{2}\bm{u}_{\Rf{2}}+\IRftoL{3}\bm{u}_{\Rf{3}}+\IRftoL{4}\bm{u}_{\Rf{4}},
\end{equation*} 
where $\bm{u}_{\Rf{f}}$ is the vector containing the evaluation of the function $\fnc{U}$ at the nodes of the abutting surface 
of the $\Rf{f}$ element and 
\begin{equation*}
  \IRftoL{f}\equiv\IRftoLoneD{f}{2}\otimes\IRftoLoneD{f}{3},\quad f=1,2,3,4.
\end{equation*}
\end{thrm}
\begin{proof}
This proof follows in the same way as proven elsewhere, for example, see~\cite{Friedrich2018}.
\end{proof}

The semi-discrete skew-symmetric split operator given in \Eq~\eqref{eq:convectiondiffusionsplit}, 
discretized using the macro element operators $\Dtildel{l}$, and metric terms, $\Jtilde$, $\matAlmk{l}{m}{}$,
leads to the following scheme:
\begin{equation}\label{eq:macro}
  \begin{split}
  &\Jtilde\frac{\mr{d}\utilde}{\mr{d}t}+
  \frac{1}{2}\sum\limits_{l,m=1}^{3}a_{m}\left(\Dtildel{l}\matAlmkhat{l}{m}{}+\matAlmkhat{l}{m}{}\Dtildel{l}\right)\utilde\\
  &-\frac{1}{2}\sum\limits_{l,m=1}^{3}a_{m}\diag\left(\utilde\right)\Dtildel{l}\matAlmkhat{l}{m}{}\tildeone
  =\sum\limits_{l,a=1}^{3}\Dtildehatl{l}\matCtildela{l}{a}\Dtildehatl{a},
  \end{split}
\end{equation} 
where
\begin{equation}\label{eq:marcmetrics} 
  \begin{split}
  &\utilde\equiv\left[\uL\Tr,\uRf{1}\Tr,\uRf{2}\Tr,\uRf{3}\Tr,\uRf{4}\Tr\right]\Tr,
  \Jtilde\equiv\diag\left[
  \begin{array}{cccc}  
  \matJkhat{\rmL}\\
  &\matJkhat{\Rf{1}}\\
  &&\ddots\\
  &&&\matJkhat{\Rf{4}}
  \end{array}
  \right],\\
  &\matAlmkhat{l}{m}{}\equiv
  \left[
  \begin{array}{cccc}
  \matAlmkhat{l}{m}{\rmL}\\
  &\matAlmkhat{l}{m}{\Rf{1}}\\
  &&\ddots\\ 
  &&&\matAlmkhat{l}{m}{\Rf{4}}
\end{array}  
  \right],\\
&\matCtildela{l}{a}\equiv\sum\limits_{m=1}^{3}b_{m}\matAlmkhat{l}{m}{}=
  \left[
  \begin{array}{cccc}
  \matChatla{l}{a}_{\rmL}\\
  &\matChatla{l}{m}_{\Rf{1}}\\
  &&\ddots\\ 
  &&&\matChatla{l}{m}_{\Rf{4}}
\end{array}  
  \right].
\end{split}
\end{equation}
As was the case in \Eq~\eqref{eq:convectionsplitdisc}, a necessary condition for stability is that the metric
terms satisfy the following discrete GCL conditions:
\begin{equation}\label{eq:discGCLmacro}
  \sum\limits_{l=1}^{3}\Dtildel{l}\matAlmkhat{l}{m}{}\tildeone=\bm{0}.
\end{equation}
Unfortunately, since $\Dtildel{1}$ is not a tensor product operator
and therefore in general does not commute with the other derivative matrix operators, 
discrete metrics constructed using the analytic formalism of Vinokur and Yee~\cite{Vinokur2002a} 
or Thomas and Lombard~\cite{Thomas1979} will not in general satisfy the discrete GCL condition 
required in \Eq~\eqref{eq:discGCLmacro}. This means that instead, the metric terms 
have to be constructed so that they directly satisfy the GCL constraints.

\begin{remark}
The metric terms are assigned colors; e.g.,  
the time-term Jacobian: $\Jtilde$ or the volume metric terms: $\matAlmkhat{l}{m}{}$.  
Metric terms with common colors form a set that must be computed consistently.  
For example, the time-term Jacobian and the volume metric Jacobian may not be computed in the 
same way.  Another
important set are the surface metrics are introduced in the next subsection.
\end{remark}
\subsection{Isolating the metric terms}
The discrete GCL system~\eqref{eq:discGCLmacro} is highly under-determined and couples the approximation 
of the metric terms in all five elements. In general, the resulting GCL conditions for 
arbitrary $h/p$-refinement would couple large sets of elements making the solution of~\eqref{eq:discGCLmacro} 
difficult if not impossible. Note that the GCL conditions originate from the spatial discretization 
of the skew-symmetric splitting of the convective terms. Thus, if the approximation for those terms can be 
appropriately modified then a set of element-local discrete GCL conditions can be constructed 
making the problem tractable again. This is precisely the approach taken in~\cite{Fernandez2019_p_euler,Fernandez2019_p_ns,Fernandez2018_TM} 
in the context of $p$-refinement/coarsening, and it is the same procedure used herein. 

Examining the volume terms for the approximation of the skew-symmetric splitting 
highlights how to decouple the discrete GCL conditions:
\begin{equation}\label{eq:combined}
  \begin{split}
&\Mtilde\left(\Dtildel{1}\matAlmkhat{1}{m}{}+\matAlmkhat{1}{m}{}\Dtildel{1}\right)=\\
 &\left[
  \begin{array}{ccccc}
   \mat{A}_{11}&
\mat{A}_{12}&\mat{A}_{13}&\mat{A}_{14}&\mat{A}_{15}\\
-\mat{A}_{12}\Tr
      &\mat{A}_{22}\\
      -\mat{A}_{13}\Tr&&\mat{A}_{33}\\ 
       -\mat{A}_{14}\Tr&&&\mat{A}_{44}\\ 
       -\mat{A}_{15}\Tr&&&&\mat{A}_{55} 
  \end{array}
  \right]
  +\left(\Etildel{1}\matAlmkhat{1}{m}{}+\matAlmkhat{1}{m}{}\Etildel{1}\right),\\\\
  &\mat{A}_{11}\equiv \left\{\Sxilhat{1}^{\rmL}\matAlmkhat{1}{m}{\rmL}+\matAlmkhat{1}{m}{\rmL}\Sxilhat{1}^{\rmL}\right\},\\
  &\mat{A}_{1f}=  \frac{\Deltalk{2}{\rmL}\Deltalk{3}{\rmL}}{4}\left\{\begin{array}{l}
    \colorbox{yellow}{\matAlmkhat{1}{m}{\rmL}}\left(\eNl{}^{\rmL}\left(\eonel{}^{\Rf{f}}\right)\Tr\otimes\PoneD_{\rmL}\IRftoLoneD{f}{2}\otimes\PoneD_{\rmL}\IRftoLoneD{f}{3}\right)\\
    +\left(\eNl{}^{\rmL}\left(\eonel{}^{\Rf{f}}\right)\Tr\otimes\PoneD_{\rmL}\IRftoLoneD{f}{2}\otimes\PoneD_{\rmL}\IRftoLoneD{f}{3}\right)\colorbox{red}{\matAlmkhat{1}{m}{\Rf{f}}}\end{array}\right\},\\
    &\mat{A}_{ff}\equiv\left\{\Sxilhat{1}^{\Rf{f}}\matAlmkhat{1}{m}{\Rf{f}}+\matAlmkhat{1}{m}{\Rf{f}}\Sxilhat{1}^{\Rf{f}}\right\},
    \quad f=1,2,3,4.
  \end{split}
\end{equation}
The highlighted terms are responsible for the weak coupling in the discrete GCL constraints. Note that 
these can be replaced with any design order quantities. The approach taken here to decouple 
the discrete GCL conditions is to zero the terms associated with the 
surface metrics on the element $\rmL$, \ie, the terms $\matAlmkhat{1}{m}{\rmL}$ and to specify 
the terms on the $\Rf{f}$ elements, \ie, $\matAlmkhat{1}{m}{\Rf{f}}$. 

\begin{remark}
  In contrast to the $p$-adaptation case~\cite{Fernandez2019_p_ns,Fernandez2019_p_euler,Fernandez2018_TM}, we do not use surface metric terms 
  from both sides of the element. This is because using surface metric terms from the $\rmL$ element results in a coupled  system of equations for the GCL conditions \eqref{eq:discGCLmacro}. 
\end{remark}

The action of the interface coupling is illustrated in Fig.~\ref{fig:footprint}. 
\begin{figure}
     \centering
     \begin{subfigure}[b]{0.45\textwidth}
         \centering
         \includegraphics[width=\textwidth]{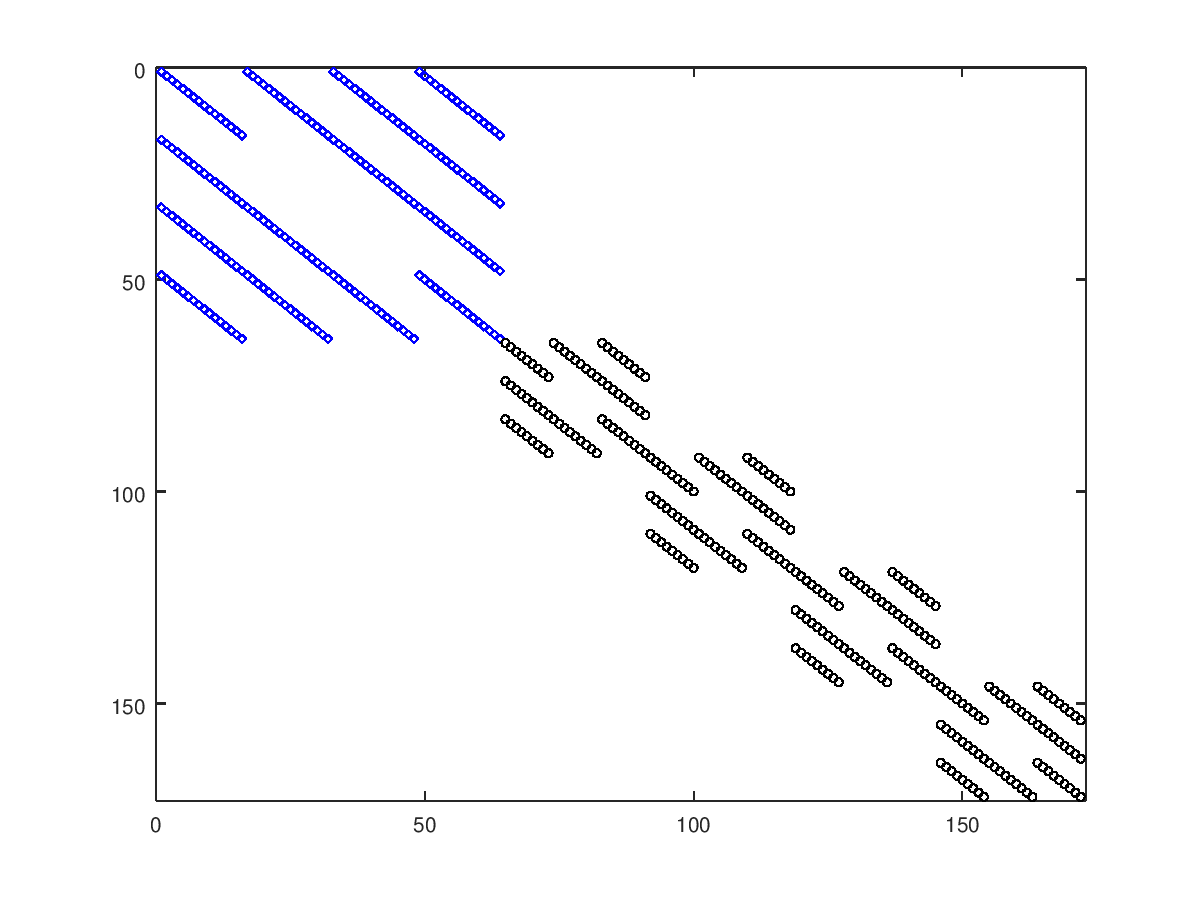}
         \caption{Original macro (non SBP) $\Dtildehatl{1}$ operator without coupling.}
         \label{fig:original}
     \end{subfigure}
     \hfill
     \begin{subfigure}[b]{0.45\textwidth}
         \centering
         \includegraphics[width=\textwidth]{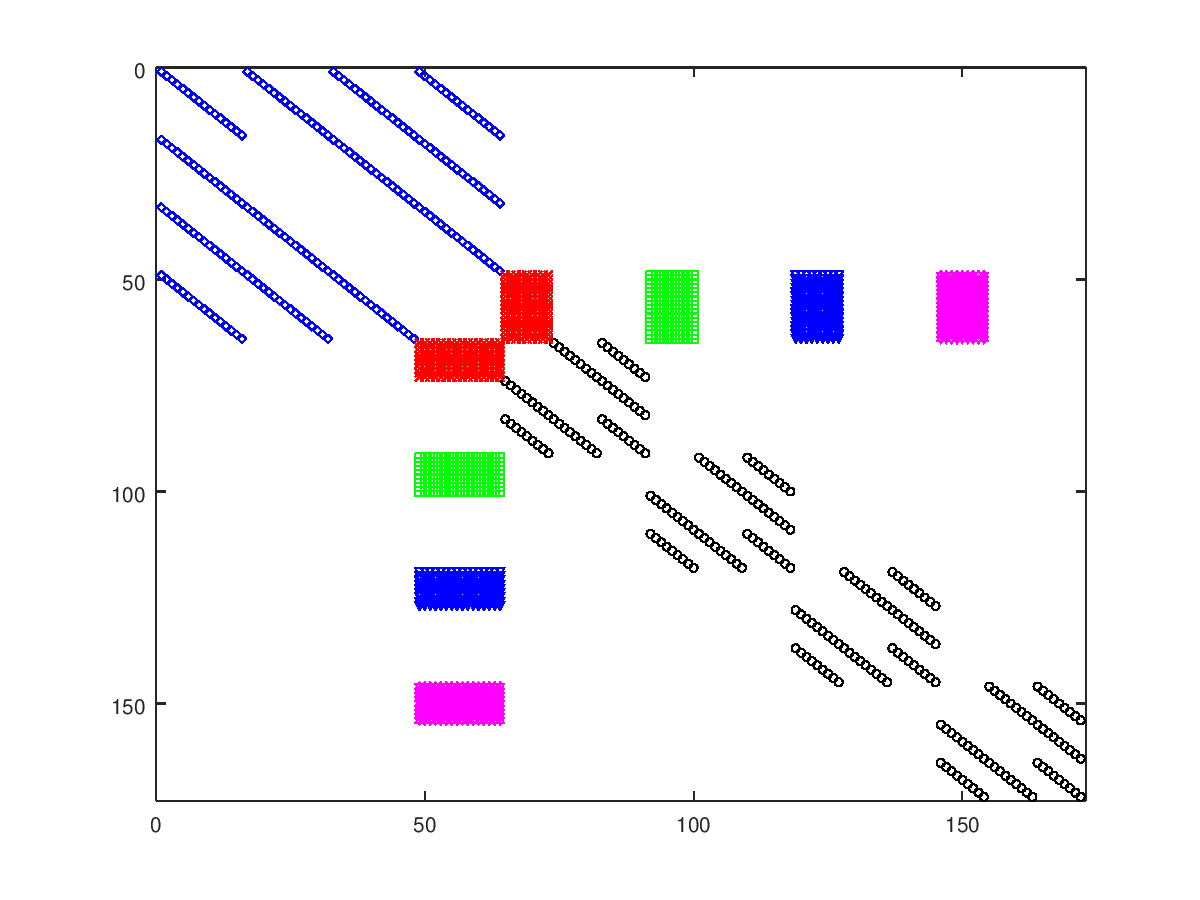}
         \caption{Modified SBP macro element operator $\Dtildehatl{1}$.}
         \label{fig:new}
     \end{subfigure}
\begin{subfigure}[b]{0.45\textwidth}
         \centering
         \includegraphics[width=\textwidth]{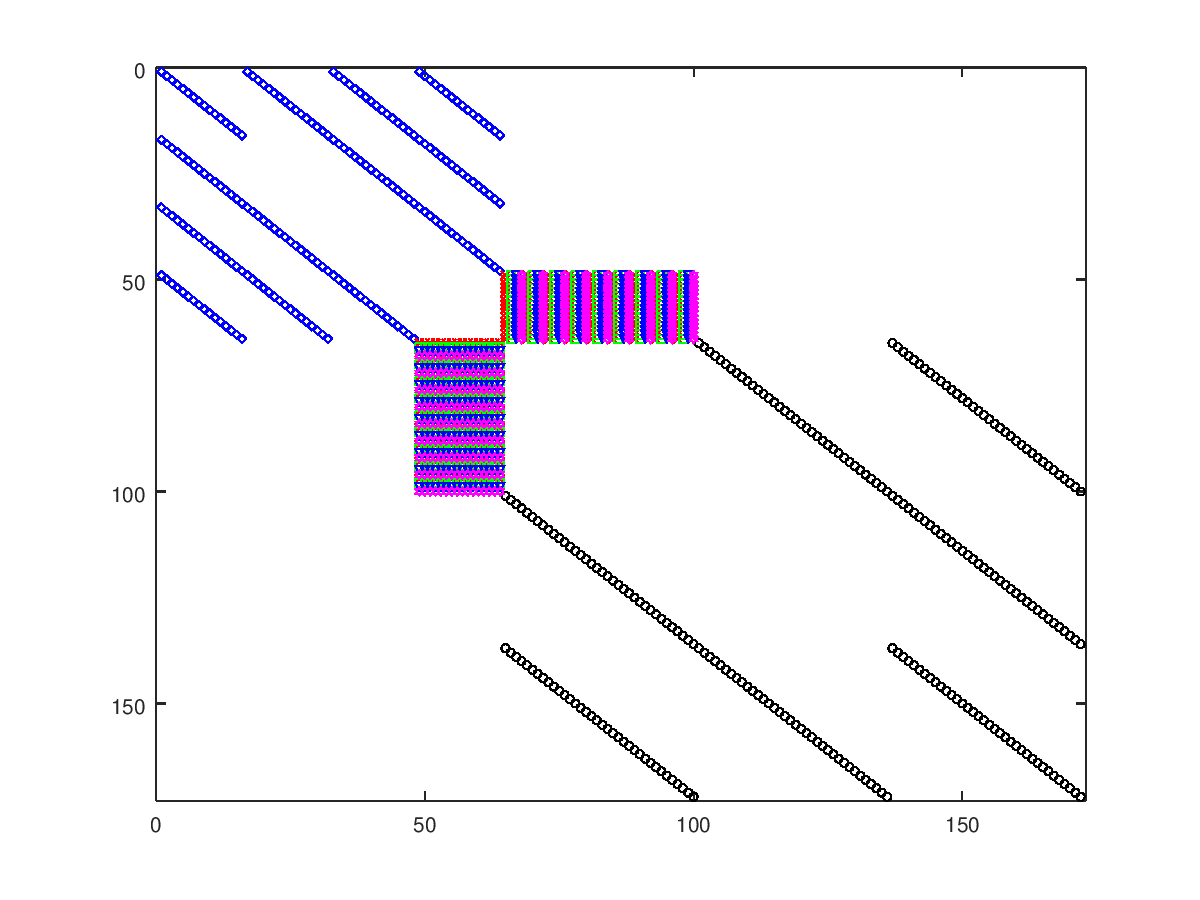}
         \caption{SBP macro element $\Dtildehatl{1}$ organized in lexicographical order.}
         \label{fig:newlex}
     \end{subfigure}
        \caption{Non-zero pattern of the various macro $\Dtildehatl{1}$ operators.}
        \label{fig:footprint}
\end{figure}
Using the above approach, the discrete GCL conditions \eqref{eq:discGCLmacro} become (where contributions from the boundary SATs 
have been ignored)
\begin{equation}\label{eq:discGCLL}
  \begin{split}
  &\sum\limits_{l=1}^{3}\Dxilhat{l}^{\rmL}\matAlmkhat{l}{m}{\rmL}\ones{\rmL}=\\
  &\frac{\Deltalk{2}{\Rf{f}}\Deltalk{3}{\Rf{f}}}{4} \left(\M^{\rmL}\right)^{-1}
 \left\{\RL\Tr\left(\PoneD_{\rmL}\otimes\PoneD_{\rmL}\right)\RL\matAlmkhat{1}{m}{\rmL}\right\}\ones{\rmL}\\
 &-\frac{\Deltalk{2}{\Rf{f}}\Deltalk{3}{\Rf{f}}}{4} \left(\M^{\rmL}\right)^{-1}\sum\limits_{f=1}^{4}
\left\{
\RL\Tr
  \left(\PoneD_{\rmL}\IRftoLoneD{f}{2}\otimes\PoneD_{\rmL}\IRftoLoneD{f}{3}\right)\matAlmkhatgamma{1}{m}{\Rf{f}}{\Ghat}{red}
  \RRf{f}\right\}\ones{\Rf{f}},
  \end{split}
\end{equation}
\begin{equation}\label{eq:discGCLRf}
  \begin{split}
   &\sum\limits_{l=1}^{3}\Dxilhat{l}^{\Rf{f}}\matAlmkhat{l}{m}{\Rf{f}}\ones{\Rf{f}}=\\
  &-\frac{\Deltalk{2}{\Rf{f}}\Deltalk{3}{\Rf{f}}}{4}\left(\M^{\Rf{f}}\right)^{-1}
  \left\{\RRf{f}\Tr
  \left(\PoneD_{\Rf{f}}\otimes
  \PoneD_{\Rf{f}}\right)\RRf{f}\matAlmkhat{1}{m}{\Rf{f}}\right\}\ones{\Rf{f}}\\
 &+ \frac{\Deltalk{2}{\Rf{f}}\Deltalk{3}{\Rf{f}}}{4}\left(\M^{\Rf{f}}\right)^{-1}\left\{
\RRf{f}\Tr
  \matAlmkhatgamma{1}{m}{\Rf{f}}{\Ghat}{red}
 \left(\PoneD_{\Rf{f}}\ILtoRfoneD{f}{2}\otimes\PoneD_{\Rf{f}}\ILtoRfoneD{f}{3}\right)
  \RL
\right\}\ones{\rmL},
  \end{split}
\end{equation}
where
\begin{equation*}
  \begin{split}
    &\RL\equiv\left(\eNl{}^{\rmL}\right)\Tr\otimes\Imat{\rmL}\otimes\Imat{\rmL},\quad
    \RRf{f}\equiv\left(\eonel{}^{\Rf{f}}\right)\Tr\otimes\Imat{\Rf{f}}\otimes\Imat{\Rf{f}}.
  \end{split}
\end{equation*}
The matrices $\matAlmkhatgamma{1}{m}{\Rf{f}}{\Ghat}{red}$ are of size 
$\Nl{\Rf{f}}^{2/3}\times\Nl{\Rf{f}}^{2/3}$ and their
diagonal elements are 
approximations to the 
metrics on the surface nodes of element $\Rf{f}$ at the shared interface. 
In order to decouple the five systems of equations in~\eqref{eq:discGCLL} and 
\eqref{eq:discGCLRf}
the terms in $\matAlmkhatgamma{1}{m}{\Rf{f}}{\Ghat}{red}$ need to be specified, for example, 
using the analytic metrics, which is the approach taken in this paper. For later use, we introduce notation for the macro element $\Dtildelm{l}{m}$ which is 
the macro element operator constructed as described above for the metric terms $\Jdxildxmhat{l}{m}$.

The next section reviews how to construct the metrics so that the 
discrete GCL conditions \eqref{eq:discGCLL} and~\eqref{eq:discGCLRf} are satisfied.
\subsection{Metric solution mechanics}
This section details the approximation of the metric terms so that entropy stability and free-stream 
preservation are maintained. There are two sets of metrics that need to be approximated, the volume 
metrics and the surface metrics. What needs to be satisfied are the discrete GCL 
equations~\eqref{eq:discGCLL} and~\eqref{eq:discGCLRf}, which are recast below in a form that 
is more convenient for developing a solution procedure. Thus,  
multiplying the discrete GCL constraints by $-1$, 
using the SBP property $\mat{Q}=-\mat{Q}\Tr+\mat{E}$, and simplifying the expressions gives  
\begin{equation}\label{eq:discGCLL2}
  \begin{split}
  &\sum\limits_{l=1}^{3}\left(\Qxilhat{l}^{\rmL}\right)\Tr\matAlmkhat{l}{m}{\rmL}\ones{\rmL}=\\
 &\frac{\Deltalk{2}{\Rf{f}}\Deltalk{3}{\Rf{f}}}{4}\sum\limits_{f=1}^{4}
\left\{
\RL\Tr
  \left(\PoneD_{\rmL}\IRftoLoneD{f}{2}\otimes\PoneD_{\rmL}\IRftoLoneD{f}{3}\right)\matAlmkhatgamma{1}{m}{\Rf{f}}{\Ghat}{red}
  \RRf{f}\right\}\ones{\Rf{f}},
  \end{split}
\end{equation}
\begin{equation}\label{eq:discGCLRf2}
  \begin{split}
   &\sum\limits_{l=1}^{3}\Qxilhat{l}^{\Rf{f}}\matAlmkhat{l}{m}{\Rf{f}}\ones{\Rf{f}}=\\
  &-\frac{\Deltalk{2}{\Rf{f}}\Deltalk{3}{\Rf{f}}}{4}\left(\M^{\Rf{f}}\right)^{-1}\left\{
\RRf{f}\Tr
  \matAlmkhatgamma{1}{m}{\Rf{f}}{\Ghat}{red}
 \left(\PoneD_{\Rf{f}}\ILtoRfoneD{f}{2}\otimes\PoneD_{\Rf{f}}\ILtoRfoneD{f}{3}\right)
  \RL
\right\}\ones{\rmL},
  \end{split}
\end{equation}
where $\Imat{\Rf{f}}$ is an identity matrix of size $\Nl{\Rf{f}}^{1/3}\times\Nl{\Rf{f}}$ and 
$\Nl{\Rf{f}}$ is the total number of nodes in element $\Rf{f}$.

Note that the contributions from the $\Exilhat{l}$ from the left-hand side 
(\ie, coming from the step $\mat{Q}=-\mat{Q}\Tr+\mat{E}$) related to the boundaries of the macro element 
are ignored. This contributions interact with the boundary SATs in the same way as the interface does. 

The metric terms in \Eq~\eqref{eq:discGCLL2} and \Eq~\eqref{eq:discGCLRf2} are set by solving a 
strictly convex quadratic optimization problem, based on the algorithm proposed in Crean \etal~\cite{Crean2018} 
(see also~\cite{Fernandez2019_p_euler,Fernandez2018_TM}). Here the procedure is exemplified in terms 
of the discrete GCL system on the $\rmL$ element:
\begin{equation}\label{eq:optvol}
\begin{split}
&\min\limits_{\amk{m}{\rmL}}\frac{1}{2}\left(\amk{m}{\rmL}-\amk{m,\,\mathrm{target}}{\rmL}\right)\Tr
\left(\amk{m}{\rmL}-\amk{m,\,\mathrm{target}}{\rmL}\right),\\&\text{subject to }
\Mkopt{\rmL}\amk{m}{\rmL}=\cmk{m}{\rmL},\quad m=1,2,3,
\end{split}
\end{equation}
where the vectors $\amk{m}{\rmL}$ and $\amk{m,\,\mathrm{target}}{\rmL}$ are the optimized and 
target volume metric terms, respectively. Herein, the analytic metric terms are the
target volume metrics. Furthermore, 
\begin{equation*}
\left(\amk{m}{\rmL}\right)\Tr\equiv
\ones{\rmL}\Tr\left[
\matAlmkhat{1}{m}{\rmL},\matAlmkhat{2}{m}{\rmL},\matAlmkhat{3}{m}{\rmL}
\right],
\end{equation*} 
\begin{equation*}
\Mkopt{\rmL}\equiv\left[
  \left(\Qxilhat{1}^{\rmL}\right)\Tr,
  \left(\Qxilhat{2}^{\rmL}\right)\Tr,
  \left(\Qxilhat{3}^{\rmL}\right)\Tr
\right],
\end{equation*}
and
\begin{equation*}
\cmk{m}{\rmL}\equiv\frac{\Deltalk{2}{\Rf{f}}\Deltalk{3}{\Rf{f}}}{4}\sum\limits_{f=1}^{4}
\left\{
\RL\Tr
  \left(\PoneD_{\rmL}\IRftoLoneD{f}{2}\otimes\PoneD_{\rmL}\IRftoLoneD{f}{3}\right)\matAlmkhatgamma{1}{m}{\Rf{f}}{\Ghat}{red}
  \RRf{f}\right\}\ones{\Rf{f}},
\end{equation*}
with $\amk{m}{\rmL}$ of size $3\Nl{\rmL}\times 1$, $\Mkopt{\rmL}$ of size $\Nl{\rmL}\times3\Nl{\rmL}$, and 
$\cmk{m}{\rmL}$ of size $\Nl{\rmL}\times 1$, where $\Nl{\rmL}$ is the total number of nodes in element 
$\rmL$. The optimal solution, in the Cartesian $2$-norm, is given by (see Proposition $1$ in 
Crean~\etal~\cite{Crean2018})
\begin{equation}\label{eq:optvol}
\amk{m}{\rmL}=\amk{m,\,\mathrm{target}}{\rmL}-\left(\Mkopt{\rmL}\right)^{\dagger}
\left(\Mkopt{\rmL}\amk{m,\,\mathrm{target}}{\rmL}-\cmk{m}{\rmL}\right),
\end{equation}
where $\left(\Mkopt{\rmL}\right)^{\dagger}$ is the Moore--Penrose pseudo inverse of 
$\Mkopt{\rmL}$. This pseudo inverse is computed using a singular value decomposition of $\Mkopt{\rmL}$
\begin{equation*}
\Mkopt{\rmL}=\Uk{\rmL}\Sigmak{\rmL}\left(\Vk{\rmL}\right)\Tr,\qquad
\left(\Mkopt{\rmL}\right)^{\dagger}=\Vk{\rmL}\left(\Sigmak{\rmL}\right)^{\dagger}
\left(\Uk{\rmL}\right)\Tr.
\end{equation*}
The unitary matrix $\Uk{\rmL}$ is of size $\Nl{\rmL}\times\Nl{\rmL}$, $\Sigmak{\rmL}$ is a diagonal matrix 
of size $\Nl{\rmL}\times\Nl{\rmL}$ containing the singular values of $\Mkopt{\rmL}$, and 
$\left(\Vk{\rmL}\right)\Tr$ is of size $\Nl{\rmL}\times3\Nl{\rmL}$ with orthonormal rows. 
The optimal solution 
$\amk{m}{\rmL}$ given by~\eqref{eq:optvol} satisfies the discrete GCL relations~\eqref{eq:discGCLL2} if the 
following constraint is satisfied:
\begin{equation}\label{eq:ckconstraint}
\ones{\rmL}\Tr\cmk{m}{\rmL}=0.
\end{equation}
The constraint~\eqref{eq:ckconstraint} is a discrete approximation to the integral of the GCL 
equations over the domain $\OhatL$, \ie
\begin{equation}\label{eq:cmkexact}
\ones{\rmL}\Tr\cmk{m}{\rmL}\approx
\int_{\OhatL}\sum\limits_{l=1}^{3}\frac{\partial}{\partial\xilhat{l}}\left(\Jdxildxmhat{l}{m}\right)\mr{d}\Ohat=
\oint_{\GammahatL}\sum\limits_{l=1}^{3}\Jdxildxmhat{l}{m}\nxilhat{l}\mr{d}\Gammahat=0.
\end{equation}
In fact, our approach is to specify the surface metric terms $\matAlmkhatgamma{l}{m}{\Rf{f}}{\Ghat}{red}$ such that 
$\ones{\rmL}\Tr\cmk{m}{\rmL}$ is exactly equal to the surface integral term on the RHS 
of~\eqref{eq:cmkexact}.

The constraint~\eqref{eq:ckconstraint} arises because $\Mkopt{\rmL}$ has one zero singular value 
associated with the constant singular vector. This implies that in order for~\eqref{eq:discGCLL2} to 
have an exact solution, $\cmk{m}{\rmL}$ must be orthogonal to the constant vector 
(see~\cite{Fernandez2019_p_euler} for a complete discussion). The next 
theorem is one of the main results of this work and gives the conditions on the analytic metric terms so that the constraint~\eqref{eq:ckconstraint} is 
satisfied.
\begin{thrm}
If the analytic metric terms used to populate $\matAlmkhatgamma{l}{m}{\Rf{f}}{\Ghat}{red}$ are at most the degree of 
the weakest cubature rule involved in the nonconforming interface, then the 
constraints~\eqref{eq:ckconstraint} are satisfied.
\end{thrm}
\begin{proof}
  The proof follows from the accuracy of the interpolation operators and the associated 
  cubature rules interacting at the nonconforming interface. 
\end{proof}
Thus far, the concentration has been on nonconforming faces. For nonconforming elements (\ie, 
elements that have at least one nonconforming face), on conforming faces the surface metric terms 
that appear in the discrete GCL~\eqref{eq:discGCLL2} and~\eqref{eq:discGCLRf2}  are taken as the surface metric terms of the 
adjoining face. The metric terms of the adjoining face are approximated using a standard approach, 
such as that of Vinokur and Yee~\cite{Vinokur2002a} or Thomas and Lombard~\cite{Thomas1979}, and 
\Theorem $\,2$ of Ref.~\cite{Fernandez2019_p_euler} guarantees that metric terms computed in this way 
satisfy the constraint~\eqref{eq:ckconstraint}.
\section{Nonlinearly stable schemes: Viscous Burgers' equation}\label{sec:Burgers}
The general $h/p$-nonconforming machinery presented in the previous section will be applied to 
the compressible Navier--Stokes equations in Section~\ref{sec:NS}. However, in order for the resulting
discretization to have the telescoping property, and therefore nonlinear stability, 
necessitates special nonlinear approximations that lead to this property. In this section, the required 
Hadamard derivative formulation is exemplified using the simple viscous Burgers' equation. 

The viscous Burgers' equation and its canonically split form are 
\begin{equation}\label{eq:Burgerssplit}
  \frac{\partial\U}{\partial t}+\frac{\partial}{\partial \xm{1}}\left(\frac{\U^{2}}{2}\right) = \frac{\partial^{2}\U}{\partial\xm{1}^{2}} \quad;\quad
  \frac{\partial\U}{\partial t}+\frac{1}{3}\frac{\partial}{\partial \xm{1}}\left(\U^{2}\right)+\frac{\U}{3}\frac{\partial\U}{\partial \xm{1}} = \frac{\partial^{2}\U}{\partial\xm{1}^{2}},
\end{equation}
where, as in the convection-diffusion equation, the splitting is on the inviscid terms.
Applying an energy analysis to the split form of~\eqref{eq:Burgerssplit} gives (for details see, for example, \cite{Carpenter2015}) 
\begin{equation}\label{eq:energyBurgers}
  \frac{1}{2}\frac{\mr{d}\|\U\|^{2}}{\mr{d}t}+\oint_{\Gamma}\frac{\U^{3}}{3}\nxm{1}\mr{d}\Gamma=\oint_{\Gamma}\U\frac{\partial\U}{\partial\xm{1}}\nxm{1}\mr{d}\Gamma-\int_{\Omega}\left(\frac{\partial\U}{\partial\xm{1}}\right)^{2}\mr{d}\Omega,\,\|\U\|^{2}\equiv\int_{\Omega}\U^{2}\mr{d}\Omega.
\end{equation}

The semi-discrete proof of stability that will be constructed shortly follows the continuous proof in a discrete sense 
such that when contracted by $\bm{u}\Tr\M$, \ie, the discrete analogue of multiplying by the solution 
and integrating in space, the result is the sum of spatial terms that telescope to the boundaries.

Ignoring the imposition of boundary conditions (\ie, SATs) and concentrating on a single element, 
then the discretization of~\eqref{eq:Burgerssplit} with SBP operators is given as 
\begin{equation}\label{eq:Burgerssplitdisc}
  \frac{\mr{d}\bm{u}}{\mr{d} t}+ \frac{1}{3}\Dxm{1}\diag\left(\bm{u}\right)\bm{u}+\frac{1}{3}\diag\left(\bm{u}\right)\Dxm{1}\bm{u}= \Dxm{1}\bm{\Theta},\quad
  \bm{\Theta}\equiv\Dxm{1}\bm{u}.
\end{equation}
Multiplying~\eqref{eq:Burgerssplitdisc} by $\bm{u}\Tr\M$ results in 
\begin{equation}\label{eq:Burgerssplitdiscenergy}
  \frac{1}{2}\frac{\mr{d}\bm{u}\Tr\M \bm{u}}{\mr{d} t}+\frac{1}{3}\left(\bm{u}^{3}(N)-\bm{u}^{3}(1)\right)= 
  \bm{u}\Tr\Exm{1}\Dxm{1}\bm{u}-\bm{u}\Tr\Dxm{1}\Tr\M\Dxm{1}\bm{u},
\end{equation}
where each term mimics the corresponding term in \eqref{eq:energyBurgers}. 
Furthermore, Eq.~\eqref{eq:Burgerssplitdiscenergy} has the telescoping property, 
\ie, the remaining terms are at the boundaries.

Notice that the key to obtaining a telescoping semi-discrete form is the 
careful discretization of the inviscid terms (in this case using a canonical split form), whereas 
the viscous terms were directly discretized in strong conservation form. 

The discrete inviscid terms in~\eqref{eq:Burgerssplitdisc} can be recast 
using the Hadamard derivative formalism. The equivalence between the 
split form and he Hadamard derivative operators is given as follows 
\begin{equation}\label{eq:Burgershadamard}
   2\Dxm{1}\circ\matFxm{1}{\bm{u}}{\bm{u}}\ones{1}
  \quad \leftrightarrow \quad
  \frac{1}{3}\Dxm{1}\diag\left(\bm{u}\right)\bm{u}+\frac{1}{3}\diag\left(\bm{u}\right)\Dxm{1}\bm{u} \: .
\end{equation}
%
Two components are use to construct the Hadamard derivative operator: first, an SBP derivative operator, and 
second a two-point flux function related to inviscid flux vector being discretely differentiated.
The Hadamard derivative operator combines these two components such that two-point fluxes are 
constructed between the center point and all other points of dependency within
the SBP stencil.  The SBP telescoping property~\cite{Fisher2013} results from precise 
local cancellation of spatial terms and can then be extended directly to nonlinear operators.  

In the case of the Burgers' equation, the two-point flux function that results in an equivalence 
between the split form and the Hadamard derivative operator is~\cite{Tadmor2003,Carpenter2015} 
\begin{equation*}
\fxmsc{m}{\bmui{i}}{\bmui{j}}\equiv 
 \frac{\left\{\left(\bmui{i}\right)^{2}+\bmui{i}\bmui{j}+\left(\bmui{j}\right)^{2}\right\}}{6},
\end{equation*}
where $\bmui{i}$ and $\bmui{j}$ are the $i\Th$ and $j\Th$ components of $\bm{u}$.
For the purpose of demonstration, a simple SBP operator constructed on the LGL nodes $\left(-1,0,1\right)$
is used:
\begin{equation*}
    \mat{D}_{\xm{1}}=
    \left[ \begin {array}{ccc} -\frac{3}{2}&2&-\frac{1}{2}\\ \noalign{\medskip}-\frac{1}{2}&0&\frac{1}{2}
\\ \noalign{\medskip}\frac{1}{2}&-2&\frac{3}{2}\end {array} \right]. 
\end{equation*}
The two argument Hadamard matrix flux, \matFxm{m}{\bm{u}}{\bm{u}} is given as
\begin{equation*}
  \begin{split}
    &\matFxm{m}{\bm{u}}{\bm{u}}=\\
    &\left[
        \begin{array}{ccc}
            \frac{\left(\bmui{1}\right)^{2}}{2}&
            \frac{\left(\bmui{1}\right)^{2}+\bmui{1}\bmui{2}+\left(\bmui{2}\right)^{2}}{6}&
            \frac{\left(\bmui{1}\right)^{2}+\bmui{1}\bmui{3}+\left(\bmui{3}\right)^{2}}{6}\\\\
\frac{\left(\bmui{2}\right)^{2}+\bmui{2}\bmui{1}+\left(\bmui{1}\right)^{2}}{6}&
\frac{\left(\bmui{2}\right)^{2}}{2}&
\frac{\left(\bmui{2}\right)^{2}+\bmui{2}\bmui{3}+\left(\bmui{3}\right)^{2}}{6}\\\\
\frac{\left(\bmui{3}\right)^{2}+\bmui{3}\bmui{1}+\left(\bmui{1}\right)^{2}}{6}&
\frac{\left(\bmui{3}\right)^{2}+\bmui{3}\bmui{2}+\left(\bmui{2}\right)^{2}}{6}&
\frac{\left(\bmui{3}\right)^{2}}{2}
        \end{array}
    \right].
      \end{split}
\end{equation*}
Thus, 
\begin{equation*}
  \begin{split}
    &\mat{D}_{\xm{1}}\circ\matFxm{m}{\bm{u}}{\bm{u}}\bm{1}=\\
    &\left[
      \arraycolsep=0.5pt
        \begin{array}{ccc}
           {\scriptscriptstyle-\frac{3}{2} \frac{\left(\bmui{1}\right)^{2}}{2}}&
            {\scriptscriptstyle2\frac{\left(\bmui{1}\right)^{2}+\bmui{1}\bmui{2}+\left(\bmui{2}\right)^{2}}{6}}&
            {\scriptscriptstyle-\frac{1}{2}\frac{\left(\bmui{1}\right)^{2}+\bmui{1}\bmui{3}+\left(\bmui{3}\right)^{2}}{6}}\\\\
{\scriptscriptstyle-\frac{1}{2}\frac{\left(\bmui{2}\right)^{2}+\bmui{2}\bmui{1}+\left(\bmui{1}\right)^{2}}{6}}&
{\scriptscriptstyle0}&
{\scriptscriptstyle\frac{1}{2}\frac{\left(\bmui{2}\right)^{2}+\bmui{2}\bmui{3}+\left(\bmui{3}\right)^{2}}{6}}\\\\
{\scriptscriptstyle\frac{1}{2}\frac{\left(\bmui{3}\right)^{2}+\bmui{3}\bmui{1}+\left(\bmui{1}\right)^{2}}{6}}&
{\scriptscriptstyle-2\frac{\left(\bmui{3}\right)^{2}+\bmui{3}\bmui{2}+\left(\bmui{2}\right)^{2}}{6}}&
{\scriptscriptstyle\frac{3}{2}\frac{\left(\bmui{3}\right)^{2}}{2}}
        \end{array}
    \right]
    \left[
        \begin{array}{c}
            {\scriptscriptstyle1}\\\\
            {\scriptscriptstyle1}\\\\
            {\scriptscriptstyle1}
        \end{array}
        \right].
      \end{split}
\end{equation*}
The equivalence between the two approaches can be determined via inspection.

The general notation necessary for discretizing the inviscid fluxes of the compressible 
Navier--Stokes equations is now detailed. Consider the discretization of the derivative of a flux vector $\Fxm{m}$ in the
$\xm{m}$ Cartesian direction. As for the Burgers' equation, the key components are 
an SBP matrix difference operator, $\Dxm{m}$, and a two argument matrix flux function,
 $\matFxm{m}{\uk}{\ur}$, which is constructed from diagonal matrices and is defined block-wise as
\begin{equation*}
    \begin{split}
  &\left(\matFxm{m}{\uk}{\ur}\right)\left(e(i-1)+1:ei,e(j-1)+1:ej\right) \equiv  \diag\left(
  \fxmsc{m}{\uki{(i)}}{\uri{(j)}}    
  \right),\\\\
  &\uki{(i)}\equiv\uk\left(e(i-1)+1:ei\right),\;
  \uri{(j)}\equiv\ur\left(e(j-1)+1:ej\right),\\\\
  &i = 1\dots,\Nl{\kappa}^{3},\; j = 1,\dots,\Nl{r}^{3},
    \end{split}
\end{equation*}
where $e$ is the number of equations in the system of PDEs. In the context of the 
compressible Navier--Stokes equations $e=5$ and the two argument matrix flux function is of size 
$\left(e\,\Nl{\kappa}^{3}\right)\times\left(e\,\Nl{r}^{3}\right)$, where $e\Nl{\kappa}^{3}$ and $e\Nl{r}^{3}$ are 
the total number of entries in the vectors $\uk$ and $\ur$ corresponding the solution 
variables in elements $\kappa$ and $r$, respectively. Therefore, $\uki{(i)}$ is the vector of the $e$ solution 
variables evaluated at the $i\Th$ node.
The vectors $\fxmsc{m}{\uki{(i)}}{\uri{(j)}}$ 
are constructed from two-point flux functions that are symmetric in their arguments, 
$\left(\uki{(i)},\uri{(j)}\right)$, and consistent, \ie, 
\begin{equation*}
 \fxmsc{m}{\uki{(i)}}{\uri{(j)}}= \fxmsc{m}{\uri{(j)}}{\uki{(i)}},
\quad\fxmsc{m}{\uki{(i)}}{\uki{(i)}} = \Fxm{m}\left(\uki{(i)}\right),
\end{equation*}
where $\Fxm{m}$ is the inviscid flux vector in the $\xm{m}$ Cartesian direction. 
With the notation defined, the Hadamard differentiation operator for the inviscid flux,
$\frac{\partial\Fxm{m}}{\partial\xm{m}}$, is constructed as 
\begin{equation*}
    2\Dxm{m}\circ\matFxm{m}{\qk{\kappa}}{\qk{\kappa}}\ones{\kappa}\approx
    \frac{\partial\Fxm{m}}{\partial\xm{m}}\left(\bm{x}^{\kappa}\right),
\end{equation*}
where $\bm{x}^{\kappa}$ is the vector of vectors containing the nodal coordinates. 
The resulting approximation has equivalent order properties as constructing an approximation 
to the derivative of the flux vector directly using an SBP operator $\Dxm{m}$ 
(see Theorem  $1$ in Crean~\etal~\cite{Crean2018}). 
\section{Application to the compressible Navier--Stokes equations}\label{sec:NS}
Herein, the nonconforming algorithm presented in Section~\ref{sec:hplin} is combined with the 
mechanics presented in Section~\ref{sec:Burgers} to construct an entropy conservative discretization 
of the compressible Navier--Stokes equations for arbitrary $h/p$-nonconforming meshes. 
First, the continuous equations and entropy analysis are reviewed in Section~\ref{sec:reviewNSentropy}. 
 Second, in Section~\ref{sec:semiNS}, the semi-discrete algorithm is presented and analyzed.

\subsection{Review of the continuous entropy analysis}\label{sec:reviewNSentropy}
The entropy stable algorithm is constructed by discretizing the skew-symmetric form 
of the compressible Navier--Stokes equations, with the viscous flux recast in terms
of entropy variables. This form of the equations is given as
\begin{equation}\label{eq:NSCCS}
\begin{split}
&\Jk\frac{\partial\Qk}{\partial t}+\frac{1}{2}\sum\limits_{l,m=1}^{3}\left(
\frac{\partial }{\partial \xil{l}}\left(\Jdxildxm{l}{m}\FxmI{m}\right)
+\Jdxildxm{l}{m}\frac{\partial \FxmI{m}}{\partial \xil{l}}\right)\\
&-\frac{1}{2}\sum\limits_{l,m=1}^{3}\FxmI{m}\frac{\partial}{\partial\xil{l}}\left(\Jdxildxm{l}{m}\right)
=\sum\limits_{l,a=1}^{3}
\frac{\partial}{\partial\xil{l}}\left(\Chatij{l}{a}\frac{\partial\bfnc{W}}{\partial \xil{a}}\right),
\end{split}
\end{equation}
where the last set of terms on the left-hand side are zero by the GCL relations~\eqref{eq:GCL}. Furthermore,
\begin{equation}\label{eq:Chatij}
\Chatij{l}{a}=\Jdxildxm{l}{m}\sum\limits_{m,j=1}^{3}\Cij{m}{j}\frac{\partial\xil{a}}{\partial x_{j}},
\end{equation}
$\Q$ is the vector of conserved variables, and $\FxmI{m}$ is the inviscid flux vector in the 
$\xm{m}$ direction. The vector of conserved variables is given by 
\begin{equation*}
\Q = \left[\rho,\rho\Um{1},\rho\Um{2},\rho\Um{3},\rho\E\right]\Tr,
\end{equation*}
where $\rho$ denotes the density, $\bm{\fnc{U}} = \left[\Um{1},\Um{2},\Um{3}\right]\Tr$ is the velocity 
vector, and $\E$ is the specific total energy. The inviscid fluxes are given as
\begin{equation*}
  \begin{split}
\FxmI{m} = &\left[\rho\Um{m},\rho\Um{m}\Um{1}+\delta_{m,1}\fnc{P},\rho\Um{m}\Um{2}+\delta_{m,2}\fnc{P},\right.\\
&\left.\rho\Um{m}\Um{3}+\delta_{m,3}\fnc{P},\rho\Um{m}\fnc{H}\right]\Tr,
  \end{split}
\end{equation*}
where $\fnc{P}$ is the pressure, $\fnc{H}$ is the specific total enthalpy and $\delta_{i,j}$ is the 
Kronecker delta.

The necessary constituent relations are
\begin{equation*}
\fnc{H} = c_{\fnc{P}}\fnc{T}+\frac{1}{2}\bm{\fnc{U}}\Tr\bm{\fnc{U}},\quad \fnc{P} = \rho R \fnc{T},\quad R = \frac{R_{u}}{M_{w}},
\end{equation*}
where $\fnc{T}$ is the temperature, $R_{u}$ is the universal gas constant, $M_{w}$ is the molecular weight of the gas, 
and $c_{\fnc{P}}$ is the specific heat capacity at constant pressure. Finally, the specific thermodynamic entropy is given as 
\begin{equation*}
s=\frac{R}{\gamma-1}\log\left(\frac{\fnc{T}}{\fnc{T}_{\infty}}\right)-R\log\left(\frac{\rho}{\rho_{\infty}}\right),\quad \gamma=\frac{c_{p}}{c_{p}-R},
\end{equation*}
where $\fnc{T}_{\infty}$ and $\rho_{\infty}$ are the reference temperature and density, respectively. 

The viscous fluxes, $\FxmV{m}$, have been recast in terms of the entropy variables, 
$\bfnc{W}\equiv\partial\fnc{S}/\partial\bfnc{Q}$, where $\fnc{S}$ is the entropy function 
$\fnc{S}\equiv-\rho s$: 
\begin{equation}\label{eq:Fxment}
\FxmV{m}=\sum\limits_{j=1}^{3}\Cij{m}{j}\frac{\partial\bfnc{W}}{\partial x_{j}}.
\end{equation}
The viscous fluxes written in components are given as
\begin{equation}\label{eq:Fv}
\FxmV{m}=\left[0,\tau_{1,m},\tau_{2,m},\tau_{3,m},
\sum\limits_{i=1}^{3}\tau_{i,m}\fnc{U}_{i}-\kappa\frac{\partial \fnc{T}}{\partial\xm{m}}\right]\Tr,
\end{equation} 
and the viscous stresses are defined as
\begin{equation}\label{eq:tau}
\tau_{i,j} = \mu\left(\frac{\partial\fnc{U}_{i}}{\partial x_{j}}+\frac{\partial\fnc{U}_{j}}{\partial x_{i}}
-\delta_{i,j}\frac{2}{3}\sum\limits_{n=1}^{3}\frac{\partial\fnc{U}_{n}}{\partial x_{n}}\right),
\end{equation}
where $\mu(T)$ is the dynamic viscosity and $\kappa(T)$ is the thermal conductivity (not to be confused with the choice of 
parameter for element numbering). 

The compressible Navier--Stokes equations have a convex extension, 
that when integrated over the physical domain, $\Omega$, depends only on the boundary data and negative semi-definite dissipation terms. This convex extension depends on an entropy function, $\fnc{S}$, 
and it is used to prove the stability in the $L^{2}$ norm.
Here, a brief review of the 
entropy stability analysis is given. A detailed presentation 
is available, for instance, in \cite{dafermos-book-2010,Svard2015,Carpenter2015}. 

Under the assumption that that the entropy function 
$\fnc{S}$ is convex, which is guaranteed if $\rho,\fnc{T}>0$, then the vector of entropy variables, 
$\bfnc{W}$, simultaneously contracts all of the inviscid flux as follows:
\begin{equation}\label{eq:wcontract}
  \bfnc{W}\Tr\frac{\partial\FxmI{m}}{\partial \xil{l}}=
  \frac{\partial\fnc{S}}{\partial\bfnc{Q}}
  \frac{\partial\FxmI{m}}{\partial \bfnc{Q}}
  \frac{\partial\bfnc{Q}}{\partial\xil{l}}=
  \frac{\partial\fxm{m}}{\partial\bfnc{Q}}\frac{\partial\FxmI{m}}{\partial \bfnc{Q}},\quad
  l,\,m=1,\,2,\,3,
\end{equation}
where $\fxm{m}$ is the entropy flux in the $\xm{m}$ direction. 

The entropy stability analysis proceeds by first multiplying (contracting) \Eq~\eqref{eq:NSCCS} 
by the transpose of the entropy variables, $\bfnc{W}\Tr$, 
\begin{equation}\label{eq:NSCCSW}
\begin{split}
&\overbrace{\Jk\bfnc{W}\Tr\frac{\partial\Qk}{\partial t}}^{I}
+\frac{1}{2}\sum\limits_{l,m=1}^{3}\left(
\overbrace{\bfnc{W}\Tr\frac{\partial }{\partial \xil{l}}\left(\Jdxildxm{l}{m}\FxmI{m}\right)}^{II}
+\overbrace{\Jdxildxm{l}{m}\bfnc{W}\Tr\frac{\partial \FxmI{m}}{\partial \xil{l}}}^{III}\right)
=\\&\sum\limits_{l,a=1}^{3}
\overbrace{\bfnc{W}\Tr\frac{\partial}{\partial\xil{l}}
\left(\Chatij{l}{a}\frac{\partial\bfnc{W}}{\partial \xil{a}}\right)}^{IV}.
\end{split}
\end{equation}
With the help of \Eq~\eqref{eq:wcontract} and the product rule, the terms $I-IV$ are now simplified:
\begin{equation}\label{eq:I}
I\equiv\Jk\bfnc{W}\Tr\frac{\partial\Qk}{\partial t} = 
\Jk\frac{\partial\fnc{S}}{\partial\bfnc{Q}}\frac{\partial\Qk}{\partial t}=
\Jk\frac{\partial\Sk}{\partial t},
\end{equation}
\begin{equation}\label{eq:II}
  \begin{split}
  II\equiv\bfnc{W}\Tr\frac{\partial }{\partial \xil{l}}\left(\Jdxildxm{l}{m}\FxmI{m}\right)
  &=\Jdxildxm{l}{m}\bfnc{W}\Tr\frac{\partial\FxmI{m} }{\partial \xil{l}}
+\bfnc{W}\Tr\FxmI{m}\frac{\partial }{\partial \xil{l}}\left(\Jdxildxm{l}{m}\right)\\ 
&=\Jdxildxm{l}{m}\frac{\partial\fxm{m} }{\partial \xil{l}}
+\bfnc{W}\Tr\FxmI{m}\frac{\partial }{\partial \xil{l}}\left(\Jdxildxm{l}{m}\right),
  \end{split}
\end{equation}
\begin{equation}\label{eq:III}
  III\equiv\Jdxildxm{l}{m}\bfnc{W}\Tr\frac{\partial \FxmI{m}}{\partial \xil{l}}=
  \Jdxildxm{l}{m}\frac{\partial \fxm{m}}{\partial \xil{l}},
  \end{equation}
  \begin{equation}\label{eq:IV}
    IV\equiv\bfnc{W}\Tr\frac{\partial}{\partial\xil{l}}
\left(\Chatij{l}{a}\frac{\partial\bfnc{W}}{\partial \xil{a}}\right)=
\frac{\partial}{\partial\xil{l}}
\left(\bfnc{W}\Tr\Chatij{l}{a}\frac{\partial\bfnc{W}}{\partial \xil{a}}\right)
-\frac{\partial\bfnc{W}\Tr}{\partial\xil{l}}
\Chatij{l}{a}\frac{\partial\bfnc{W}}{\partial \xil{a}}.
  \end{equation}
  Substituting \Eq~\eqref{eq:I} through~\eqref{eq:IV} into~\eqref{eq:NSCCSW} results in
\begin{equation}\label{eq:NSCCSW2}
\begin{split}
&\Jk\frac{\partial\Sk}{\partial t}+\sum\limits_{l,m=1}^{3}
 \Jdxildxm{l}{m}\frac{\partial\fxm{m} }{\partial \xil{l}}
+\frac{\bfnc{W}\Tr}{2}\sum\limits_{m=1}^{3}\FxmI{m}\cancelto{0\text{ via GCL\eqref{eq:GCL}}}{\sum\limits_{l=1}^{3}\frac{\partial }{\partial \xil{l}}\left(\Jdxildxm{l}{m}\right)}
=\\&\sum\limits_{l,a=1}^{3}
\left\{
\frac{\partial}{\partial\xil{l}}
\left(\bfnc{W}\Tr\Chatij{l}{a}\frac{\partial\bfnc{W}}{\partial \xil{a}}\right)
-\frac{\partial\bfnc{W}\Tr}{\partial\xil{l}}
\Chatij{l}{a}\frac{\partial\bfnc{W}}{\partial \xil{a}}
\right\}.
\end{split}
\end{equation}
Bringing the metric terms within the derivative on the term $\Jdxildxm{l}{m}\frac{\partial\fxm{m} }{\partial \xil{l}}$ 
and using the product rule results in 
\begin{equation}\label{eq:NSCCSW3}
\begin{split}
&\Jk\frac{\partial\Sk}{\partial t}+\sum\limits_{l,m=1}^{3}
 \frac{\partial}{\partial \xil{l}}\left(\Jdxildxm{l}{m}\fxm{m}\right)
 -\sum\limits_{m=1}^{3}\fxm{m}\cancelto{0\text{ via GCL~\eqref{eq:GCL}}}
 {\sum\limits_{m=1}^{3}\frac{\partial}{\partial \xil{l}}\left(\Jdxildxm{l}{m}\right)}
=\\&\sum\limits_{l,a=1}^{3}
\left\{
\frac{\partial}{\partial\xil{l}}
\left(\bfnc{W}\Tr\Chatij{l}{a}\frac{\partial\bfnc{W}}{\partial \xil{a}}\right)
-\frac{\partial\bfnc{W}\Tr}{\partial\xil{l}}
\Chatij{l}{a}\frac{\partial\bfnc{W}}{\partial \xil{a}}
\right\}.
\end{split}
\end{equation}
Rearranging \Eq~\eqref{eq:NSCCSW3} and expanding the dissipation term yields 
\begin{equation}\label{eq:NSCCSW4}
\begin{split}
&\Jk\frac{\partial\Sk}{\partial t}=\sum\limits_{l=1}^{3}
 \frac{\partial}{\partial \xil{l}}\left(
  -\sum\limits_{m=1}^{3} \Jdxildxm{l}{m}\fxm{m}
 +\sum\limits_{a=1}^{3}\left(\bfnc{W}\Tr
 \Chatij{l}{a}\frac{\partial\bfnc{W}}{\partial \xil{a}}\right) 
  \right)\\ 
&-\left[
  \begin{array}{c}
  \frac{\partial\bfnc{W}}{\partial \xil{1}}\\
  \frac{\partial\bfnc{W}}{\partial \xil{2}}\\
\frac{\partial\bfnc{W}}{\partial \xil{3}}
\end{array}
\right]\Tr
\underbrace{
\left[
\begin{array}{ccc}
  \Chatla{1}{1}&\Chatla{1}{2}&\Chatla{1}{3}\\
  \Chatla{1}{2}\Tr&\Chatla{2}{2}&\Chatla{2}{3}\\
  \Chatla{1}{3}\Tr&\Chatla{2}{3}\Tr&\Chatla{3}{3}
\end{array}  
\right]}_{\equiv\hat{\mat{C}}}
\underbrace{
\left[
  \begin{array}{c}
  \frac{\partial\bfnc{W}}{\partial \xil{1}}\\
  \frac{\partial\bfnc{W}}{\partial \xil{2}}\\
\frac{\partial\bfnc{W}}{\partial \xil{3}}
\end{array}
\right]}_{\equiv\hat{\bfnc{W}}},
\end{split}
\end{equation}
where the matrix $\hat{\mat{C}}$ is symmetric semi-definite (see~\cite{Fisher2012phd} for details).

Integrating \Eq~\eqref{eq:NSCCSW4} in space and using integration by parts gives
\begin{equation}\label{eq:NSCCSW5}
\begin{split}
\int_{\Ohatk}\Jk\frac{\partial\Sk}{\partial t}\mr{d}\Ohat\leq&\sum\limits_{l=1}^{3}
 \oint_{\Ghatk}\left(
  -\sum\limits_{m=1}^{3} \Jdxildxm{l}{m}\fxm{m}
 +\sum\limits_{a=1}^{3}\left(\bfnc{W}\Tr
 \Chatij{l}{a}\frac{\partial\bfnc{W}}{\partial \xil{a}}\right)
  \right)\nxil{l}\mr{d}\Ghat.
\end{split}
\end{equation}

An $L^{2}$ bound on the solution is derived from inequality~\eqref{eq:NSCCSW5} 
by integrating in time and assuming 1) nonlinearly well-posed boundary and initial 
conditions, and 2) positivity of temperature and density. Then, the result 
can be turned into a bound on the solution in terms of the data of the problem~\cite{dafermos-book-2010,Svard2015}. 
\subsection{An $h/p$-nonconforming algorithm}\label{sec:semiNS}
 The skew-symmetrically split form of the compressible Navier--Stokes equations~\eqref{eq:NSCCS}
 is discretized by combining the macro element SBP operator in Section~\ref{sec:cdnon} 
 with the nonlinear mechanics presented in Section~\ref{sec:Burgers}. Thus, the discretization 
of~\eqref{eq:NSCCS} over the macro element is given as
\begin{equation}\label{eq:NSCCS1disc}
\begin{split}
&\Jtilde\frac{\mr{d}\qtilde}{\partial t}+
\sum\limits_{l,m=1}^{3}\Dtildelm{l}{m}
\circ\matFxm{m}{\qtilde}{\qtilde}\tildeone\\
&-\frac{1}{2}\sum\limits_{l,m=1}^{3}\diag\left(\bm{f}_{\xm{m}}^{I}\right)\Dtildehatl{l}\matAlmkhat{l}{m}{}\tildeone
=\sum\limits_{l,a=1}^{3}\Dtildehatl{l}\matChatla{l}{a}\Dtildehatl{a}\wtilde,\\ 
&\qtilde\equiv\left[\qk{\rmL}\Tr,\qk{\Rf{1}}\Tr,\dots,\qk{\Rf{4}}\Tr\right]\Tr,\;
\wtilde\equiv\left[\wk{\rmL}\Tr,\wk{\Rf{1}}\Tr,\dots,\wk{\Rf{4}}\Tr\right]\Tr,
\end{split}
\end{equation}
where $\bm{f}_{\xm{m}}^{I}$ is a vector of vectors constructed by evaluating $\FxmI{m}$ at the 
mesh nodes. 
Note that the factor of $\frac{1}{2}$ on the skew-symmetric inviscid volume terms has been absorbed 
as a result of using the nonlinear operator, \eg, $2\Dxil{l}\circ\matFxm{m}{\qk{\kappa}}{\qk{\kappa}}\ones{\kappa}\approx\frac{\partial\Fxm{m}}{\partial\xil{l}}(\bm{\xi}^{\kappa})$.
Furthermore, the flux function matrix, $\matFxm{m}{\qtilde}{\qtilde}$, is constructed using a two-point flux function, $\fxmsc{m}{\tildeqi{(i)}}{\tildeqi{(i)}}$, that satisfies 
the Tadmor's shuffle condition~\cite{Tadmor2003}
\begin{equation}\label{eq:shuffle}
\left(\tildewi{(i)}-\tildewi{(j)}\right)\Tr\fxmsc{m}{\tildeqi{(i)}}{\tildeqi{(i)}}=\tilde{\bm{\psi}}_{\xm{m}}^{(i)}-\tilde{\bm{\psi}}_{\xm{m}}^{(j)}.
\end{equation}
The $\Dtildelm{l}{m}$ operators are constructed 
from the scalar conservation law 
operators developed in Section~\eqref{sec:hplin} by tensoring them with
an identity 
matrix, $\Imat{5}$, to accommodate the system of five equations. For example,
\begin{equation*}
  \Dxilhat{1}^{\rmL}\equiv\barDxilhat{1}\otimes\Imat{5},\quad \barDxilhat{1}^{\rmL}\equiv\frac{2}{\Deltalk{1}{\rmL}}\DoneD_{\rmL}\otimes\Imat{\rmL}\otimes\Imat{\rmL}.
\end{equation*}
Similar to the linear stability, entropy stability necessitates that the last set of terms on the 
left-hand side of~\eqref{eq:NSCCS1disc} be zero and leads to the same set of discrete GCL conditions.

The semi-discrete entropy analysis follows the continuous analysis in a one-to-one fashion. 
In order to simplify the derivation, the following matrices are introduced: 
\begin{equation*}
  \begin{split}
  &\Dtildem{m}\equiv\sum\limits_{l=1}^{3}\Dtildelm{l}{m}
,\quad \Qtildem{m}\equiv\Mtilde\Dtildem{m},\quad
\Etildem{m}\equiv\Qtildem{m}+\Qtildem{m}\Tr.
  \end{split}
\end{equation*}
Assuming that the discrete GCL conditions are satisfied, \eqref{eq:NSCCS1disc} becomes
\begin{equation}\label{eq:NSCCS1disc2}
\begin{split}
&\Jtilde\frac{\mr{d}\qtilde}{\partial t}+
\sum\limits_{m=1}^{3}\Dtildem{m}\circ\matFxm{m}{\qtilde}{\qtilde}\tildeone=\sum\limits_{l,a=1}^{3}\Dtildehatl{l}\matChatla{l}{a}\Dtildehatl{a}\wtilde.
\end{split}
\end{equation}
Multiplying \Eq~\eqref{eq:NSCCS1disc2} by $\tilde{\bm{w}}\Tr\Mtilde$ (the discrete analogue of 
multiplying by $\bfnc{W}\Tr$ and integrating over the domain) yields
\begin{equation}\label{eq:NSCCS1disc3}
\begin{split}
&\Jtilde\wtilde\Tr\Mtilde\frac{\mr{d}\qtilde}{\partial t}+
\sum\limits_{m=1}^{3}\wtilde\Qtildem{m}\circ\matFxm{m}{\qtilde}{\qtilde}\tildeone=
\sum\limits_{l,a=1}^{3}\wtilde\Tr\Qtildehatl{l}\matChatla{l}{a}\Dtildehatl{a}\wtilde.
\end{split}
\end{equation}
Taking the transpose of one half of the volume term on the left-hand side of \Eq~\eqref{eq:NSCCS1disc3}, 
using the SBP property $\mat{Q}=-\mat{Q}\Tr+\mat{E}$, and the symmetry of the two-point flux function matrix, 
results in
\begin{equation}\label{eq:NSCCS1disc4}
\begin{split}
\Jtilde\wtilde\Tr\Mtilde\frac{\mr{d}\qtilde}{\partial t}+
\frac{1}{2}\sum\limits_{m=1}^{3}&\left(
  \wtilde\Qtildem{m}\circ\matFxm{m}{\qtilde}{\qtilde}\tildeone
  -\tildeone\Tr\Qtildem{m}\circ\matFxm{m}{\qtilde}{\qtilde}\Tr\wtilde\right.\\
&\left.+\tildeone\Tr\Etildem{m}\circ\matFxm{m}{\qtilde}{\qtilde}\Tr\wtilde
\right)=\\
&\sum\limits_{l,a=1}^{3}\wtilde\Tr\Etildehatl{l}\matChatla{l}{a}\Dtildehatl{a}\wtilde
-\sum\limits_{l,a=1}^{3}\wtilde\Tr\Dtildehatl{l}\Tr\Mtilde\matChatla{l}{a}\Dtildehatl{a}\wtilde.
\end{split}
\end{equation}
To further reduce the left-hand side terms requires the following theorem 
(this is Theorem $8$ in~\cite{Fernandez2018_TM} and the proof is given in Appendix D of that document):
\begin{thrm}\label{thrm:telescope}
Consider the matrix of $\overline{\mat{A}}$ of size $\Nl{\kappa}\times \Nl{r}$ with a tensor extension 
$\mat{A}\equiv\overline{\mat{A}}\otimes\Imat{5}$, and a two argument matrix flux 
function $\matFxm{m}{\qk{\kappa}}{\qk{r}}$ constructed from the two-point 
flux function $\fxmsc{m}{\qki{\kappa}{(i)}}{\qki{r}{(j)}}$ that satisfies the Tadmor's shuffle condition
\[
\left(\wki{\kappa}{(i)}-\wki{r}{(j)}\right)\Tr\fxmsc{m}{\qki{\kappa}{(i)}}{\qki{r}{(j)}}=
\psixmki{m}{\kappa}{i}-\psixmki{m}{r}{j}
\]
and is symmetric, \ie, $\fxmsc{m}{\qki{\kappa}{(i)}}{\qki{r}{(j)}}= \fxmsc{m}{\qki{r}{(j)}}{\qki{\kappa}{(i)}}$, then
\begin{equation*}
\wk\Tr\left(\mat{A}\right)\circ\matFxm{m}{\qk{\kappa}}{\qk{r}}\ones{r}-
\ones{\kappa}\Tr\mat{A}\circ\matFxm{m}{\qk{\kappa}}{\qk{r}}\wk{r} =
\left(\psixmk{m}{\kappa}\right)\Tr\overline{\mat{A}}\barones{r}-\barones{\kappa}\Tr\overline{\mat{A}}\psixmk{m}{r}.
\end{equation*}
\end{thrm}
Applying \Theorem~\ref{thrm:telescope} to the volume terms on the left-hand side of 
\Eq~\eqref{eq:NSCCS1disc4} yields
\begin{equation}\label{eq:NSCCS1disc5}
\begin{split}
\Jtilde\wtilde\Tr\Mtilde\frac{\mr{d}\qtilde}{\partial t}+&
\frac{1}{2}\sum\limits_{m=1}^{3}\left\{
  \left(\psitildem{m}\right)\Tr\barQtildem{m}\barones{}
  -\barones{}\Tr\barQtildem{m}\psitildem{m}
+\tildeone\Tr\Etildem{m}\circ\matFxm{m}{\qtilde}{\qtilde}\Tr\wtilde
\right\}=\\
&\sum\limits_{l,a=1}^{3}\wtilde\Tr\Etildehatl{l}\matChatla{l}{a}\Dtildehatl{a}\wtilde
-\sum\limits_{l,a=1}^{3}\wtilde\Tr\Dtildehatl{l}\Tr\Mtilde\matChatla{l}{a}\Dtildehatl{a}\wtilde.
\end{split}
\end{equation}
The term $\barQtildem{m}\barones{}$ is zero by the discrete GCL conditions and the consistency 
of the derivative operator ($\mat{D}\bm{1}=\bm{0}$) and for the same reasons 
$\barones{}\Tr\barQtildem{m}=\barones{}\Tr\barEtildem{m}$. Therefore, after 
some rearrangements 
\Eq~\eqref{eq:NSCCS1disc5} reduces to
\begin{equation}\label{eq:NSCCS1disc5}
\begin{split}
\Jtilde\wtilde\Tr\Mtilde\frac{\mr{d}\qtilde}{\partial t}=&
-\frac{1}{2}\sum\limits_{m=1}^{3}\left\{
  -\barones{}\Tr\barEtildem{m}\psitildem{m}
+\tildeone\Tr\Etildem{m}\circ\matFxm{m}{\qtilde}{\qtilde}\Tr\wtilde
\right\}\\
&+\sum\limits_{l,a=1}^{3}\wtilde\Tr\Etildehatl{l}\matChatla{l}{a}\Dtildehatl{a}\wtilde
-\sum\limits_{l,a=1}^{3}\wtilde\Tr\Dtildehatl{l}\Tr\Mtilde\matChatla{l}{a}\Dtildehatl{a}\wtilde.
\end{split}
\end{equation}
The right-hand side of ~\eqref{eq:NSCCS1disc5} contains surface terms 
(those constructed from the $\mat{E}$ matrices) and viscous dissipation terms (the last set of terms).
 The surface terms can be decomposed into the contributions of the separate surfaces 
 of the element (node-wise). The entropy conservation of the algorithm follows immediately for periodic problems 
 because these terms would cancel out with the contributions from the coupling SATs. 
 For general boundary conditions, appropriate SATs need to be constructed so that an 
 entropy inequality or equality is attained (see, for example, 
 \cite{Parsani2015,Svard2018,dalcin_2019_wall_bc}).
\section{Interface dissipation and boundary SATs}\label{sec:dissipation}
In order to render the entropy conservative scheme entropy stable, interface dissipation is added. 
The numerical dissipation added for the inviscid SATs 
(\ie, added to the right-hand side of the discretization) is motivated by a Roe approximate Riemann 
solver (for a detailed discussion see~\cite{Carpenter2014,Parsani2015b,Fernandez2019_p_euler,Fernandez2018_TM}). The 
inviscid dissipation for element $\rmL$ is given as
\begin{equation}\label{eq:dissinIRftoL}
\begin{split}
\dissL\equiv&-\left(\M^{\rmL}\right)^{-1}\RL\Tr\Porthohatl{1}^{\rmL}\left\{
\sum\limits_{f=1}^{4}\IRftoL{f}\left|\frac{\partial\bfnc{F}_{\xilhat{1}}^{I}}{\partial\bfnc{W}}\right|_{\Rf{f}}
\left(\ILtoRf{f}\RL\wL-\RRf{f}\wRf{f}\right)
\right\},
\end{split}
\end{equation}
where 
\begin{equation*}
  \begin{split}
  &\RL\equiv\left(\eNl{1}^{\rmL}\right)\Tr\otimes\Imat{\rmL}\otimes\Imat{\rmL}\otimes\Imat{5},\\
  &\Porthohatl{1}^{\rmL}\equiv\frac{\Deltalk{2}{\rmL}\Deltalk{3}{\rmL}}{4}\PoneD_{\rmL}\otimes\PoneD_{\rmL}\otimes\Imat{5},\\
  &\IRftoL{f}\equiv\IRftoLoneD{f}{2}\otimes\IRftoLoneD{f}{3}\otimes\Imat{5},\\
  &\ILtoRf{f}\equiv\ILtoRfoneD{f}{2}\otimes\ILtoRfoneD{f}{3}\otimes\Imat{5},\\
  &\RRf{f}\equiv\left(\eonel{1}^{\Rf{f}}\right)\Tr\otimes\Imat{\Rf{f}}\otimes\Imat{\Rf{f}}\otimes\Imat{5}.
  \end{split}
\end{equation*}
The inviscid dissipation term for the $\Rf{f}$ element is constructed as 
\begin{equation}\label{eq:dissinILtoRf}
\begin{split}
\dissRf{f}\equiv&-\left(\M^\Rf{f}\right)^{-1}\RRf{f}\Tr\Porthohatl{1}^{\Rf{f}}
\left|\frac{\partial\bfnc{F}_{\xilhat{1}}^{I}}{\partial\bfnc{W}}\right|_{\Rf{f}}
\left(\RRf{f}\wRf{f}-\ILtoRf{f}\RL\wL\right),
\end{split}
\end{equation}
where 
\begin{equation*}
\left|\frac{\partial\bfnc{F}^{I}}{\partial\bfnc{W}}\right|\equiv\mat{Y}\left|\mat{\Lambda}\right|\mat{Y}\Tr.
\end{equation*}
The matrices $\mat{Y}$ and $\mat{\Lambda}$ are block diagonal matrices constructed 
by assembling the point-wise $5\times5$ 
matrices obtained from the decomposition of the Jacobian 
matrix of $\bfnc{F}$ with respect to $\bfnc{W}$ and evaluated at the Roe-averaged of two states. In particular, 
$\left|\frac{\partial\bfnc{F}_{\xilhat{1}}^{I}}{\partial\bfnc{W}}\right|_{\Rf{f}}$ is constructed from the Roe 
averaged states of $\ILtoRf\qL$ and $\qRf{f}$.

Next, a theorem on the accuracy, stability, element-wise conservation, and free-stream preservation
of the inviscid dissipation term is 
presented. 
\begin{thrm}\label{thrm:dissin}
The dissipation terms~\eqref{eq:dissinIRftoL} and~\eqref{eq:dissinILtoRf} are of order 
$\min\left(\pL,\pRf{1},\dots,\pRf{4}\right)+3$, and lead to an entropy stable 
inviscid scheme and have no impact on element-wise conservation or free-stream preservation.
\end{thrm}
\begin{proof}
The proofs are similar to those in Ref.~\cite{Fernandez2019_p_euler,Fernandez2018_TM} and 
are omitted for brevity. 
\end{proof}
The viscous interface dissipation terms (interior penalty terms) take the following form:
\begin{equation}\label{eq:IPL}
\IPL\equiv-\left(\M^{\rmL}\right)^{-1}\RL\Tr\Porthohatl{1}^{\rmL}\IRftoL{f}
\matJRftilde{f}^{-1}\CtildeRfij{1}{1}{f}\left(\ILtoRf{f}\RL\wL-\Rf{f}\wRf{f}\right),
\end{equation}
\begin{equation}\label{eq:IPRf}
\IPRf{f}\equiv-\left(\M^{\Rf{f}}\right)^{-1}\RRf{f}\Tr\Porthohatl{1}^{\Rf{f}}
\matJRftilde{f}^{-1}\CtildeRfij{1}{1}{f}\left(\RRf{f}\wRf{f}-\ILtoRf{f}\RL\wL\right),
\end{equation}
where 
\begin{equation*}
\CtildeRfij{1}{1}{f}\equiv\frac{1}{2}\left\{\Chatij{1}{1}
\left(\ILtoRf{f}\RL\qL\right)+\Chatij{1}{1}\left(\RRf{f}\qRf{f}\right)\right\},
\end{equation*}
and the diagonal matrix $\matJRftilde{f}$ has the metric Jacobian associated with surface of element $\Rf{f}$ 
along its diagonal. The next theorem summarizes the properties of the viscous dissipation terms. 
\begin{thrm}
The dissipation terms~\eqref{eq:IPL} and~\eqref{eq:IPRf} are of order 
$\min\left(\pL,\pRf{1},\dots,\pRf{f}\right)+3$, lead to an entropy stable 
viscous scheme and have no impact on free-stream preservation.
\end{thrm}
\begin{proof}
The proofs are similar to those in Ref.~\cite{Fernandez2019_p_ns} and are omitted for brevity.
\end{proof}

In Section~\ref{sec:num}, four problems are used to characterize the $h/p$ nonconforming algorithm: 
1) the propagation of an isentropic vortex, 2) 
the propagation of a viscous shock, 3) the Taylor--Green vortex problem, and 4) the turbulent flow 
past a sphere. In all cases, the boundary conditions are weakly imposed by using the same type of mechanics 
as for the interface SATs discussed in this section (for details see~\cite{Parsani2015,dalcin_2019_wall_bc}).
\section{Numerical experiments}\label{sec:num}
In this section, we verify that the proposed $h/p$-algorithm retains the accuracy
and robustness of the conforming algorithm \cite{Carpenter2014,Parsani2015,Carpenter2016}.

The unstructured grid $h/p$-adaptive solver used herein has been developed at the Extreme Computing Research Center (ECRC) at KAUST on 
top of the Portable and Extensible Toolkit for Scientific computing (PETSc)~\cite{petsc-user-ref}, its mesh topology 
abstraction (DMPLEX)~\cite{KnepleyKarpeev09} and salable ordinary differential equation (ODE)/differential algebraic equations (DAE) solver library~\cite{abhyankar2018petsc}. The $p$-refinement algorithm is fully implemented 
in the unstructured solver whereas the $h$-refinement strategy leverages the 
capabilities of the p4est library \cite{BursteddeWilcoxGhattas11}.
Additionally, the conforming 
numerical solver is based on the algorithms proposed in \cite{Carpenter2014,Parsani2015,Carpenter2016}.
The systems of ODEs arising from the spatial
discretizations are integrated using the fourth-order
accurate Dormand--Prince method \cite{dormand_rk_1980} endowed with an adaptive time stepping technique based on digital signal processing \cite{Soderlind2003,Soderlind2006}. To make the temporal error negligible, a tolerance of $10^{-8}$ is always used for the time-step adaptivity.

The errors are computed using volume scaled (for the $L^{1}$ and $L^{2}$ norms) discrete norms as follows:
\begin{equation*}
\begin{split}
  &\|\bm{u}\|_{L^{1}}=\Omega_{c}^{-1}\sum\limits_{\kappa=1}^{K}\ones{\kappa}\Tr\M^{\kappa}\matJk{\kappa}\textrm{abs}\left(\bm{u}_{\kappa}\right),\\
  &\|\bm{u}\|_{L^{2}}^{2}=\Omega_{c}^{-1}\sum\limits_{\kappa=1}^{K}\bm{u}_{\kappa}\M^{\kappa}\matJk{\kappa}\bm{u}_{\kappa},\\
&\|\bm{u}\|_{L^{\infty}}=\max\limits_{\kappa=1\dots K}\textrm{abs}\left(\bm{u}_{\kappa}\right),
\end{split}
\end{equation*}
where $\Omega_{c}$ is the volume of $\Omega$ computed as 
$\Omega_{c}\equiv\sum\limits_{\kappa=1}^{K}\ones{\kappa}\Tr\M^{\kappa}\matJk{\kappa}\ones{\kappa}$.

\subsection{Isentropic Euler vortex propagation}\label{subsec:iv}
For verification and characterization of the inviscid components of the algorithm, 
the propagation of an isentropic vortex is used. This benchmark problem
has an analytical solution which is given by
\begin{equation*}
\begin{split}
& \fnc{G}\left(\xm{1},\xm{2},\xm{3},t\right) = 1
-\left\{
\left[
\left(\xm{1}-x_{1,0}\right)
-U_{\infty}\cos\left(\alpha\right)t
\right]^{2}
+
\left[
\left(\xm{2}-x_{2,0}\right)
-U_{\infty}\sin\left(\alpha\right)t
\right]^{2}
\right\},\\
&\rho = T^{\frac{1}{\gamma-1}},\\
  &\Um{1} = U_{\infty}\cos(\alpha)-\epsilon_{\nu}
\frac{\left(\xm{2}-x_{2,0}\right)-U_{\infty}\sin\left(\alpha\right)t}{2\pi}
\exp\left(\frac{\fnc{G}}{2}\right),\\
  &\Um{2} = U_{\infty}\sin(\alpha)-\epsilon_{\nu}
\frac{\left(\xm{1}-x_{1,0}\right)-U_{\infty}\cos\left(\alpha\right)t}{2\pi}
\exp\left(\frac{\fnc{G}}{2}\right),\\
  &\Um{3} = 0,\\
&T = \left[1-\epsilon_{\nu}^{2}M_{\infty}^{2}\frac{\gamma-1}{8\pi^{2}}\exp\left(\fnc{G}\right)\right],
\end{split}
\end{equation*}
where $U_{\infty}$, $M_{\infty}$, and $\left(x_{1,0},x_{2,0},x_{3,0}\right)$ are the modulus of the 
free-stream velocity, the free-stream Mach number, and the vortex center, respectively. In this paper, 
the following values are used: $U_{\infty}=M_{\infty} c_{\infty}$, $\epsilon_{\nu}=5$, $M_{\infty}=0.5$, 
$\gamma=1.4$, $\alpha=45^{\degree}$, and $\left(x_{1,0},x_{2,0},x_{3,0}\right)=\left(0,0,0\right)$.
The computational domain is
\[
  \xm{1}\in[-5,5], \qquad\xm{2}\in[-5,5],\qquad\xm{3}\in[-5,5],\qquad  t\in[0,2].
\]
The analytical solution is used to furnish data for the initial condition.

First, we report on the results aimed at validating the entropy conservation properties of 
the interior domain SBP-SAT algorithm. Thus, periodic boundary conditions are used on all six faces of the computational domain.
Furthermore, all the dissipation terms used for the 
interface coupling are turned off.
The discrete integral over the volume of the time rate of change of the entropy function,
$\displaystyle\int_{\Omega}\frac{\partial\fnc{S}}{\partial t}\mr{d}\Omega$, is monitored at every time step. This
means that at each time step the compressible Euler equations are multiplied by
the discrete entropy variables to construct the discrete analog of the right-hand side of
\eqref{eq:NSCCSW5}.

We subdivide the computational
domain using ten hexahedrons in each coordinate direction. Subsequently,
we split random cells in the mesh
using one or two levels of $h$-refinement. Then, we assign the solution polynomial degree in 
each element to a random integer chosen uniformly from the set $\{2,3,4,5\}$ (\ie, each member in the set has 
an equal probability of being chosen). To test the conservation of entropy and therefore the free-stream condition when
curved element interfaces are used we construct the the LGL collocation point coordinates 
at element interfaces\footnote{In a general setting, element interfaces can also be boundary element interfaces.} as follows:

\begin{itemize}
\item Construct a mesh using a $p_i$th-order polynomial
  approximation for the element interfaces.
\item Perturb the nodes that are used to define the $p_i$th-order polynomial
  approximation of the element interfaces as follows:
\[
\begin{split}
&x_1 = x_{1,*} + \frac{1}{15} L_1 \cos \left(   a \right) \cos \left( 3 b \right) \sin \left( 4 c \right),\\
&x_2 = x_{2,*} + \frac{1}{15} L_2 \sin \left( 4 a \right) \cos \left(   b \right) \cos \left( 3 c \right),\\
&x_3 = x_{3,*} + \frac{1}{15} L_3 \cos \left( 3 a \right) \sin \left( 4 b \right) \cos \left(   c \right),
\end{split}
\]
where,
\[
\begin{split}
&a = \frac{\pi}{L_1} \left(x_{1,*}-\frac{x_{1,H}+x_{1,L}}{2}\right), \quad
b = \frac{\pi}{L_2} \left(x_{2,*}-\frac{x_{2,H}+x_{2,L}}{2}\right), \\
&c = \frac{\pi}{L_3} \left(x_{3,*}-\frac{x_{3,H}+x_{3,L}}{2}\right).
\end{split}
\]
The symbols $L_1$, $L_2$ and $L_3$ represent the dimensions of the computational domain in the three
coordinate directions and the sub-script $*$ indicates the unperturbed coordinate of
the nodes. This step yields a perturbed $p_i$th-order polynomial.
\item Compute the coordinate of the LGL points at the element interface
by evaluating the perturbed $p_i$th-order polynomial at the LGL points used 
to define the cell solution polynomial of order $p_s$.
\end{itemize}
Herein, we use $p_i=2$.
Figure \ref{fig:iv_mesh_cut} shows a cut of the mesh
where each cell is colored according to the solution polynomial degree assigned to it.
Curved element interfaces are clearly visible. 

The propagation of the vortex is simulated for two time units.
\begin{figure}
     \centering
     \begin{subfigure}[b]{0.40\textwidth}
         \centering
         \includegraphics[width=\textwidth]{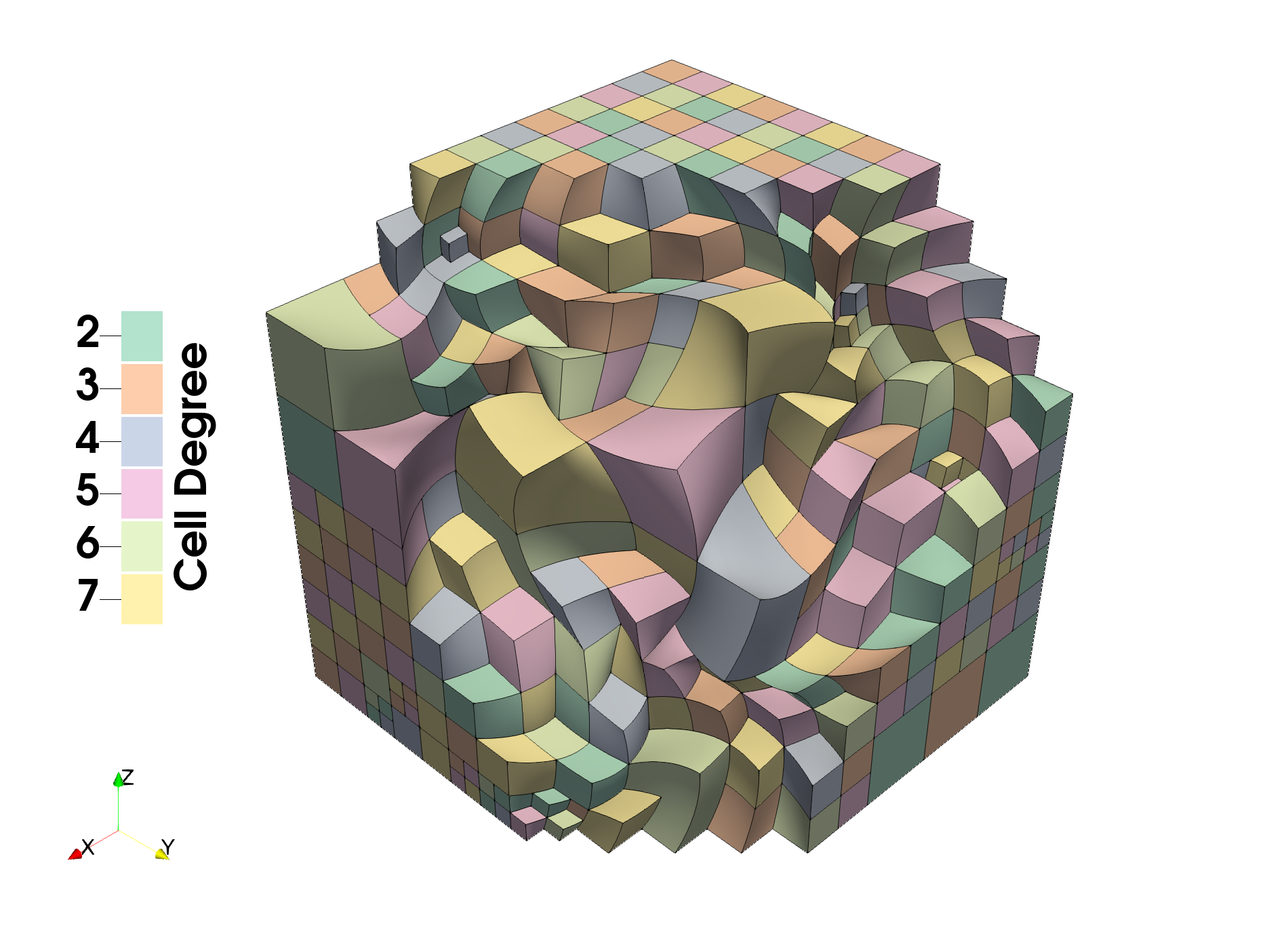}
         \caption{Polynomial degree distribution.}
         \label{fig:iv_mesh_cut}
     \end{subfigure}
     \hfill
     \begin{subfigure}[b]{0.55\textwidth}
         \centering
         \includegraphics[width=\textwidth]{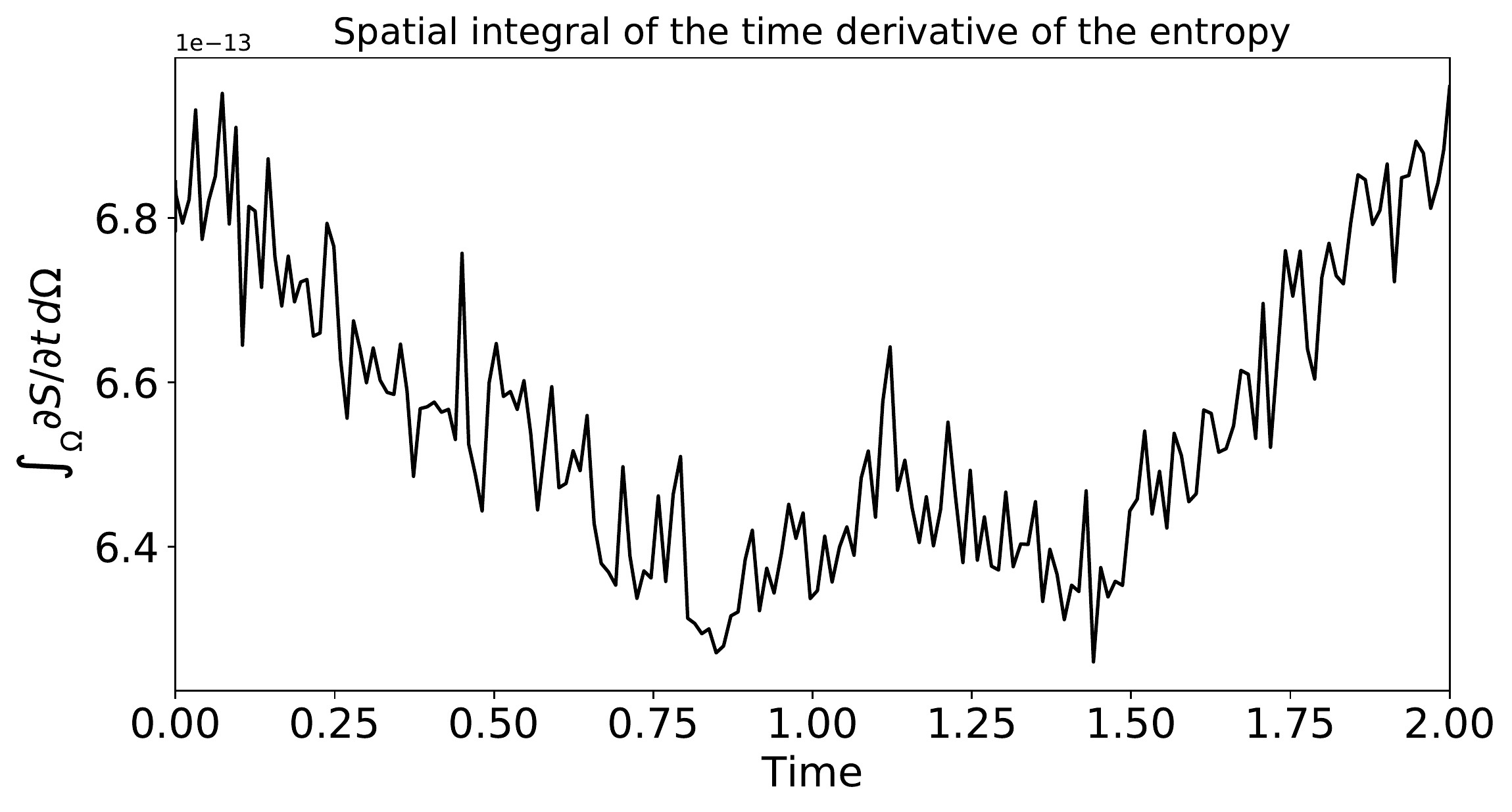}
         \caption{Time rate of change of the entropy function.}
         \label{fig:iv_dsdt}
     \end{subfigure}
        \caption{Isentropic Euler vortex.}
        \label{fig:iv_conservation}
\end{figure}
Figure \ref{fig:iv_dsdt} plots the integral over the volume of the time derivative of the entropy function. We can
see that the global variation of the discrete time rate of change of $\fnc{S}$ is practically
zero (i.e., machine double precision). This implies that the nonconforming algorithm is entropy conservative.

Second, we perform a grid convergence study to investigate the order of convergence of the $h/p$-adaptive
approach. The base grid (labeled with ``0'' in the first column of following tables) is constructed as follow:
\begin{itemize}
\item Divide the computational domain with four hexahedral elements in each coordinate direction.
\item Refine random elements by using one or two levels of $h$-refinement.
\item Assign the solution polynomial degree in each element to a random integer chosen uniformly from the set $\{p_s, p_s+1\}$.
\item Approximate the curved element interfaces with a $p_s$th-order accurate polynomial.
\item Construct the perturbed elements and their corresponding LGL points as described previously. 
\end{itemize}
From the base grid, which is similar to the one depicted in Figure \ref{fig:iv_mesh_cut}, 
a sequence of nested grids is then generated to perform the convergence study.
The results are reported in Tables
\ref{tab:iv_p1p2} through \ref{tab:iv_p4p5}
for the error on the density. The number listed in the first column denoted by 
``Levels'' indicates the number of uniform refinements in each coordinate direction.

\floatsetup[table]{font=footnotesize}
\begin{table}[htbp]
\vspace{0.5cm}
\begin{center}
\begin{tabular}{||l||c|c|c|c|c|c||c|c|c|c|c|c||}
\hline \hline
        & \multicolumn{6}{c||}{Conforming, $p=1$} & \multicolumn{6}{c||}{Nonconforming, $p=1$ and $p=2$}\\ \hline
 Levels & $L^1$    & Rate  & $L^2$    & Rate  & $L^{\infty}$ & Rate  & $L^1$    & Rate  & $L^2$    & Rate  & $L^{\infty}$ & Rate  \\ \hline
 0      & 2.74E-02 & -     & 1.32E-03 & -     &  1.55E-01    & -     & 1.02E-02 & -     & 5.80E-04 & -     & 1.43E-01     & -     \\ \hline
 1      & 1.14E-02 & -1.26 & 6.61E-04 & -1.00 &  1.12E-01    & -0.47 & 4.38E-03 & -1.22 & 2.82E-04 & -1.04 & 7.17E-02     & -1.00 \\ \hline
 2      & 5.13E-03 & -1.16 & 3.31E-04 & -1.00 &  7.29E-02    & -0.62 & 1.45E-03 & -1.60 & 9.99E-05 & -1.50 & 4.30E-02     & -0.74 \\ \hline
 3      & 1.70E-03 & -1.59 & 1.15E-04 & -1.52 &  3.01E-02    & -1.28 & 4.16E-04 & -1.80 & 3.02E-05 & -1.72 & 2.29E-02     & -0.91 \\ \hline
 4      & 4.76E-04 & -1.84 & 3.24E-05 & -1.83 &  8.53E-03    & -1.82 & 1.11E-04 & -1.91 & 8.76E-06 & -1.79 & 1.06E-02     & -1.10 \\ \hline
 5      & 1.23E-04 & -1.96 & 8.33E-06 & -1.96 &  2.13E-03    & -2.00 & 2.61E-05 & -2.09 & 2.28E-06 & -1.94 & 4.05E-03     & -1.40 \\ \hline
 6      & 3.08E-05 & -1.99 & 2.09E-06 & -1.99 &  5.24E-04    & -2.02 & 6.31E-06 & -2.05 & 5.75E-07 & -1.99 & 1.25E-03     & -1.70 \\ \hline
 7      & 7.68E-06 & -2.00 & 5.23E-07 & -2.00 &  1.30E-04    & -2.01 & 1.56E-06 & -2.02 & 1.50E-07 & -1.94 & 4.45E-04     & -1.48 \\ \hline
\end{tabular}
\caption{Convergence study of the isentropic vortex propagation: two levels of $h$-refinement, $p=1$ and $p=2$; density error.}.
\label{tab:iv_p1p2}
\end{center}
\end{table}

\floatsetup[table]{font=footnotesize}
\begin{table}[htbp]
\vspace{0.5cm}
\begin{center}
\begin{tabular}{||l||c|c|c|c|c|c||c|c|c|c|c|c||}
\hline \hline
      & \multicolumn{6}{c||}{Conforming, $p=2$} & \multicolumn{6}{c||}{Nonconforming, $p=2$ and $p=3$}\\ \hline
 Levels & $L^1$    & Rate  & $L^2$    & Rate  & $L^{\infty}$ & Rate  & $L^1$    & Rate  & $L^2$    & Rate  & $L^{\infty}$ & Rate  \\ \hline
 0      & 9.75E-03 & -     & 5.32E-04 & -     & 1.24E-01     & -     & 2.59E-03 & -     & 1.75E-04 & -     & 7.52E-02     & -     \\ \hline
 1      & 3.18E-03 & -1.61 & 2.09E-04 & -1.35 & 6.97E-02     & -0.83 & 4.11E-04 & -2.65 & 3.45E-05 & -2.34 & 3.14E-02     & -1.26 \\ \hline
 2      & 5.18E-04 & -2.62 & 3.88E-05 & -2.43 & 2.51E-02     & -1.47 & 4.91E-05 & -3.07 & 5.13E-06 & -2.75 & 8.16E-03     & -1.95 \\ \hline
 3      & 6.38E-05 & -3.02 & 5.50E-06 & -2.82 & 7.23E-03     & -1.79 & 5.86E-06 & -3.07 & 6.74E-07 & -2.93 & 2.09E-03     & -1.97 \\ \hline
 4      & 7.61E-06 & -3.07 & 7.23E-07 & -2.93 & 1.21E-03     & -2.58 & 7.05E-07 & -3.06 & 8.97E-08 & -2.91 & 4.32E-04     & -2.27 \\ \hline
 5      & 9.48E-07 & -3.00 & 9.95E-08 & -2.86 & 2.75E-04     & -2.14 & 8.54E-08 & -3.04 & 1.20E-08 & -2.91 & 1.00E-04     & -2.11 \\ \hline
 6      & 1.23E-07 & -2.94 & 1.43E-08 & -2.83 & 3.41E-05     & -3.01 & 9.98E-09 & -3.10 & 1.56E-09 & -2.94 & 2.52E-05     & -1.99 \\ \hline
\end{tabular}
  \caption{Convergence study of the isentropic vortex propagation: two levels of $h$-refinement, $p=2$ and $p=3$; density error.}.
\label{tab:iv_p2p3}
\end{center}
\end{table}

\floatsetup[table]{font=footnotesize}
\begin{table}[htbp]
\vspace{0.5cm}
\begin{center}
\begin{tabular}{||l||c|c|c|c|c|c||c|c|c|c|c|c||}
\hline \hline
      & \multicolumn{6}{c||}{Conforming, $p=3$} & \multicolumn{6}{c||}{Nonconforming, $p=3$ and $p=4$}\\ \hline
 Levels & $L^1$    & Rate  & $L^2$    & Rate  & $L^{\infty}$ & Rate  & $L^1$    & Rate  & $L^2$    & Rate  & $L^{\infty}$ & Rate  \\ \hline
 0      & 5.22E-03 & -     & 3.38E-04 & -     & 9.16E-02     & -     & 5.38E-04 & -     & 4.51E-05 & -     & 5.38E-02     & -     \\ \hline
 1      & 6.84E-04 & -2.93 & 5.30E-05 & -2.67 & 4.71E-02     & -0.96 & 4.18E-05 & -3.69 & 3.89E-06 & -3.54 & 5.42E-03     & -3.31 \\ \hline
 2      & 5.50E-05 & -3.64 & 4.54E-06 & -3.55 & 6.61E-03     & -2.83 & 2.61E-06 & -4.00 & 2.82E-07 & -3.78 & 6.47E-04     & -3.07 \\ \hline
 3      & 3.48E-06 & -3.98 & 3.33E-07 & -3.77 & 5.47E-04     & -3.59 & 1.79E-07 & -3.86 & 1.92E-08 & -3.88 & 7.36E-05     & -3.14 \\ \hline
 4      & 2.10E-07 & -4.05 & 2.45E-08 & -3.76 & 4.93E-05     & -3.47 & 1.09E-08 & -4.04 & 1.23E-09 & -3.96 & 6.93E-06     & -3.41 \\ \hline
 5      & 1.39E-08 & -3.92 & 1.87E-09 & -3.71 & 6.32E-06     & -2.96 & 7.05E-10 & -3.96 & 8.10E-11 & -3.93 & 8.10E-07     & -3.10 \\ \hline
\end{tabular}
  \caption{Convergence study of the isentropic vortex propagation: two levels of $h$-refinement, $p=3$ and $p=4$; density error.}.
\label{tab:iv_p3p4}
\end{center}
\end{table}

\floatsetup[table]{font=footnotesize}
\begin{table}[htbp]
\vspace{0.5cm}
\begin{center}
\begin{tabular}{||l||c|c|c|c|c|c||c|c|c|c|c|c||}
\hline \hline
      & \multicolumn{6}{c||}{Conforming, $p=4$} & \multicolumn{6}{c||}{Nonconforming, $p=4$ and $p=5$}\\ \hline
 Levels & $L^1$    & Rate  & $L^2$    & Rate  & $L^{\infty}$ & Rate  & $L^1$    & Rate  & $L^2$    & Rate  & $L^{\infty}$ & Rate  \\ \hline
 0      & 2.34E-03 & -     & 1.56E-04 & -     & 9.92E-02     & -     & 7.48E-05 & -     & 5.15E-06 & -     & 4.21E-03     & -     \\ \hline
 1      & 1.70E-04 & -3.78 & 1.43E-05 & -3.45 & 2.32E-02     & -2.09 & 2.20E-06 & -5.09 & 1.87E-07 & -4.79 & 2.45E-04     & -4.10 \\ \hline
 2      & 6.07E-06 & -4.81 & 6.41E-07 & -4.48 & 1.27E-03     & -4.19 & 6.76E-08 & -5.02 & 6.66E-09 & -4.81 & 1.81E-05     & -3.76 \\ \hline
 3      & 1.99E-07 & -4.93 & 2.25E-08 & -4.83 & 5.13E-05     & -4.64 & 2.08E-09 & -5.03 & 2.21E-10 & -4.91 & 1.26E-06     & -3.84 \\ \hline
 4      & 7.11E-09 & -4.81 & 8.60E-10 & -4.71 & 3.01E-06     & -4.09 & 6.61E-11 & -4.97 & 6.78E-12 & -5.03 & 9.13E-08     & -3.79 \\ \hline
\end{tabular}
  \caption{Convergence study of the isentropic vortex propagation: two levels of $h$-refinement, $p=4$ and $p=5$; density error.}.
\label{tab:iv_p4p5}
\end{center}
\end{table}

For all the degree tested (i.e. $p = 1$ to $p = 5$), the order of convergence of the conforming and 
nonconforming algorithms is very similar. However, note that in the $L^1$ and $L^2$ norms the nonconforming algorithm is
more accurate than the conforming one. In the discrete $L^{\infty}$ norm, the nonconforming scheme is sometimes slightly
worse than the conforming scheme; this most likely results from the fact that the interpolation 
matrices are sub-optimal.

\subsection{Viscous shock propagation}
Next we study the propagation of a viscous shock using the
compressible Navier--Stokes equations.  We assume a planar shock propagating along the $\xm{1}$
coordinate direction with a Prandtl number of $Pr=3/4$.
The exact solution of this problem is known;
the momentum $\fnc{V}(x_1)$ satisfies the ODE
\begin{equation}
\begin{split}
&\alpha\fnc{V}\frac{\partial\fnc{V}}{\partial\xm{1}}-(\fnc{V}-1)(\fnc{V}-\fnc{V}_{f})=0;\;-\infty\leq\xm{1}\leq+\infty,\; t\ge 0,\\
\end{split}
\end{equation}
whose solution can be written implicitly as
{\small
\begin{equation}\label{eq:implicit_sol_vs}
  \begin{split}
&\xm{1}-\frac{1}{2}\alpha\left(\log\left|(\fnc{V}(x_1)-1)(\fnc{V}(x_1)-\fnc{V}_{f})\right|\right.\\
&\left.+\frac{1+\fnc{V}_{f}}{1-\fnc{V}_{f}}\log\left|\frac{\fnc{V}(x_1)-1}{\fnc{V}(x_1)-\fnc{V}_{f}}\right|\right) = 0,
  \end{split}
\end{equation}
}
where
\begin{equation}
\fnc{V}_{f}\equiv\frac{\fnc{U}_{L}}{\fnc{U}_{R}},\qquad
\alpha\equiv\frac{2\gamma}{\gamma + 1}\frac{\,\mu}{Pr\dot{\fnc{M}}}.
\end{equation}
Here $\fnc{U}_{L/R}$ are known velocities to the left and right of the shock at
$-\infty$ and $+\infty$, respectively, $\dot{\fnc{M}}$ is the constant mass
flow across the shock, $Pr$ is the Prandtl number, and $\mu$ is the dynamic
viscosity.
The mass and total enthalpy are constant across the shock. Moreover,
the momentum and energy equations become redundant.

For our tests, $\fnc{V}$ is computed from Equation \eqref{eq:implicit_sol_vs}
to machine precision using bisection.
The moving shock solution is obtained by applying a uniform translation to the above solution.
The shock is located at the center of the domain at $t=0$ and the following
values are used: $M_{\infty}=2.5$, $Re_{\infty}=10$, and $\gamma=1.4$.
The domain is given by
\[
	\xm{1}\in[-0.5,0.5], \quad \xm{2}\in[-0.5,0.5], \quad \xm{3}\in[-0.5,0.5], \quad t\in[0,0.5].
\]
The boundary conditions are prescribed by penalizing the numerical solution
against the exact solution. The analytical solution is also used to furnish data
for the initial condition.

The base grid (labeled with ``0'' in the first column of Tables
\ref{tab:vs_p1p2} through \ref{tab:vs_p4p5}) is constructed as 
as described in Section \ref{subsec:iv}. From the base grid, which is similar to the one depicted in Figure \ref{fig:iv_mesh_cut}, 
a sequence of nested grids is then generated to perform the convergence study.
The results are reported in Tables
\ref{tab:vs_p1p2} through \ref{tab:vs_p4p5}
for the error on the density. Again, the number listed in the first column denoted by 
``Levels'' indicates the number of uniform refinement in each coordinate direction.

Similar to the propagation of the inviscid vortex, for all the degree tested (i.e., $p = 1$ to $p = 5$), the order of
convergence of the conforming and nonconforming algorithms is similar. However, note that,
the nonconforming algorithm is more accurate than the conforming one, for all the three
norms reported. 

\floatsetup[table]{font=footnotesize}
\begin{table}[htbp]
\vspace{0.5cm}
\begin{center}
\begin{tabular}{||l||c|c|c|c|c|c||c|c|c|c|c|c||}
\hline \hline
        & \multicolumn{6}{c||}{Conforming, $p=1$} & \multicolumn{6}{c||}{Nonconforming, $p=1$ and $p=2$}\\ \hline
 Levels & $L^1$    & Rate  & $L^2$    & Rate  & $L^{\infty}$ & Rate  & $L^1$    & Rate  & $L^2$    & Rate  & $L^{\infty}$ & Rate  \\ \hline
 0      & 5.43E-02 & -     & 6.54E-02 & -     & 1.40E-01   & -       & 1.31E-02 & -     & 2.11E-02 & -     & 7.32E-02     & -     \\ \hline
 1      & 2.04E-02 & -1.41 & 2.92E-02 & -1.16 & 8.42E-02   & -0.73   & 3.29E-03 & -1.99 & 5.76E-03 & -1.87 & 2.94E-02     & -1.32 \\ \hline
 2      & 5.56E-03 & -1.87 & 8.45E-03 & -1.79 & 2.85E-02   & -1.57   & 8.39E-04 & -1.97 & 1.46E-03 & -1.98 & 9.16E-03     & -1.68 \\ \hline
 3      & 1.44E-03 & -1.94 & 2.23E-03 & -1.92 & 8.12E-03   & -1.81   & 2.11E-04 & -1.99 & 3.76E-04 & -1.96 & 2.36E-03     & -1.96 \\ \hline
 4      & 3.68E-04 & -1.97 & 5.66E-04 & -1.98 & 2.26E-03   & -1.84   & 5.03E-05 & -2.07 & 8.97E-05 & -2.07 & 5.97E-04     & -1.98 \\ \hline
 5      & 9.28E-05 & -1.99 & 1.43E-04 & -1.99 & 6.05E-04   & -1.90   & 1.24E-05 & -2.02 & 2.10E-05 & -2.09 & 1.48E-04     & -2.01 \\ \hline
\end{tabular}
\caption{Convergence study of the viscous shock propagation: two levels of $h$-refinement, $p=1$ and $p=2$; density error.}.
\label{tab:vs_p1p2}
\end{center}
\end{table}

\floatsetup[table]{font=footnotesize}
\begin{table}[htbp]
\vspace{0.5cm}
\begin{center}
\begin{tabular}{||l||c|c|c|c|c|c||c|c|c|c|c|c||}
\hline \hline
      & \multicolumn{6}{c||}{Conforming, $p=2$} & \multicolumn{6}{c||}{Nonconforming, $p=2$ and $p=3$}\\ \hline
 Levels & $L^1$    & Rate  & $L^2$    & Rate  & $L^{\infty}$ & Rate  & $L^1$    & Rate  & $L^2$    & Rate  & $L^{\infty}$ & Rate  \\ \hline
 0      & 1.78E-02 & -     & 2.68E-02 & -     & 1.36E-01     & -     & 1.85E-03 & -     & 3.84E-03 & -     & 4.86E-02    & -      \\ \hline
 1      & 2.93E-03 & -2.60 & 5.05E-03 & -2.41 & 5.98E-02     & -1.19 & 2.75E-04 & -2.75 & 5.85E-04 & -2.72 & 1.11E-02    & -2.14  \\ \hline
 2      & 3.86E-04 & -2.92 & 6.93E-04 & -2.87 & 1.09E-02     & -2.45 & 4.01E-05 & -2.78 & 8.74E-05 & -2.74 & 2.03E-03    & -2.45  \\ \hline
 3      & 5.55E-05 & -2.80 & 1.03E-04 & -2.74 & 2.23E-03     & -2.29 & 5.00E-06 & -3.00 & 1.01E-05 & -3.11 & 3.18E-04    & -2.67  \\ \hline
 4      & 8.96E-06 & -2.63 & 1.79E-05 & -2.53 & 4.96E-04     & -2.17 & 6.10E-07 & -3.04 & 1.23E-06 & -3.04 & 4.20E-05    & -2.92  \\ \hline
 5      & 1.46E-06 & -2.66 & 2.99E-06 & -2.58 & 8.96E-05     & -2.47 & 7.00E-08 & -3.12 & 1.51E-07 & -3.03 & 5.50E-06    & -2.93  \\ \hline
\end{tabular}
  \caption{Convergence study of the viscous shock propagation: two levels of $h$-refinement, $p=2$ and $p=3$; density error.}.
\label{tab:vs_p2p3}
\end{center}
\end{table}

\floatsetup[table]{font=footnotesize}
\begin{table}[htbp]
\vspace{0.5cm}
\begin{center}
\begin{tabular}{||l||c|c|c|c|c|c||c|c|c|c|c|c||}
\hline \hline
      & \multicolumn{6}{c||}{Conforming, $p=3$} & \multicolumn{6}{c||}{Nonconforming, $p=3$ and $p=4$}\\ \hline
 Levels & $L^1$    & Rate  & $L^2$    & Rate  & $L^{\infty}$ & Rate  & $L^1$    & Rate  & $L^2$    & Rate  & $L^{\infty}$ & Rate  \\ \hline
 0      & 4.45E-03 & -     & 7.52E-03 & -     & 7.51E-02     & -     & 2.41E-04 & -     & 5.24E-04 & -     & 1.12E-02     & -     \\ \hline
 1      & 3.40E-04 & -3.71 & 6.50E-04 & -3.53 & 1.19E-02     & -2.66 & 1.80E-05 & -3.74 & 4.11E-05 & -3.67 & 1.01E-03     & -3.47 \\ \hline
 2      & 2.67E-05 & -3.67 & 5.36E-05 & -3.60 & 1.20E-03     & -3.30 & 1.21E-06 & -3.90 & 3.00E-06 & -3.78 & 8.17E-05     & -3.63 \\ \hline
 3      & 1.95E-06 & -3.77 & 4.25E-06 & -3.66 & 1.25E-04     & -3.26 & 7.30E-08 & -4.05 & 1.90E-07 & -3.98 & 5.82E-06     & -3.81 \\ \hline
 4      & 1.48E-07 & -3.72 & 3.67E-07 & -3.53 & 1.12E-05     & -3.48 & 4.23E-09 & -4.11 & 1.09E-08 & -4.13 & 3.60E-07     & -4.02 \\ \hline
\end{tabular}
  \caption{Convergence study of the viscous shock propagation: two levels of $h$-refinement, $p=3$ and $p=4$; density error.}.
\label{tab:vs_p3p4}
\end{center}
\end{table}

\floatsetup[table]{font=footnotesize}
\begin{table}[htbp]
\vspace{0.5cm}
\begin{center}
\begin{tabular}{||l||c|c|c|c|c|c||c|c|c|c|c|c||}
\hline \hline
      & \multicolumn{6}{c||}{Conforming, $p=4$} & \multicolumn{6}{c||}{Nonconforming, $p=4$ and $p=5$}\\ \hline
 Levels & $L^1$    & Rate  & $L^2$    & Rate  & $L^{\infty}$ & Rate  & $L^1$     & Rate  & $L^2$    & Rate  & $L^{\infty}$ & Rate   \\ \hline
 0      & 1.21E-03 & -     & 2.28E-03 & -     & 2.50E-02     & -     & 3.47E-05  & -     & 8.15E-05 & -     & 1.90E-03     & -      \\ \hline
 1      & 8.50E-05 & -3.83 & 1.54E-04 & -3.88 & 3.04E-03     & -3.04 & 1.37E-06  & -4.66 & 3.13E-06 & -4.70 & 8.02E-05     & -4.57  \\ \hline
 2      & 2.75E-05 & -4.95 & 5.66E-06 & -4.77 & 1.52E-04     & -4.32 & 4.73E-08  & -4.85 & 1.10E-07 & -4.83 & 3.04E-06     & -4.72  \\ \hline
 3      & 1.16E-07 & -4.57 & 2.54E-07 & -4.48 & 7.54E-06     & -4.33 & 1.62E-09  & -4.87 & 3.62E-09 & -4.93 & 1.44E-07     & -4.40  \\ \hline
 4      & 5.21E-09 & -4.48 & 1.11E-08 & -4.52 & 4.01E-07     & -4.23 & 6.01E-11  & -4.75 & 1.30E-10 & -4.80 & 6.96E-09     & -4.37  \\ \hline
\end{tabular}
  \caption{Convergence study of the viscous shock propagation: two levels of $h$-refinement, $p=4$ and $p=5$; density error.}.
\label{tab:vs_p4p5}
\end{center}
\end{table}


\subsection{Taylor--Green vortex at $Re=1,600$}
The purpose of this section is to demonstrate that the nonconforming algorithm has the same stability properties as the 
conforming algorithm. To do so, the Taylor--Green vortex problem on a very coarse grid is solved.

The numerical solution is computed in a periodic cube $[-\pi L\leq x,y,z\leq \pi L]$ and the initial condition 
is given by
\begin{equation}\label{eq:TG}
\begin{split}
&\fnc{U}_{1} = \fnc{V}_{0}\sin\left(\frac{x_{1}}{L}\right)\cos\left(\frac{x_{2}}{L}\right)\cos\left(\frac{x_{3}}{L}\right),\\
&\fnc{U}_{2} = -\fnc{V}_{0}\cos\left(\frac{x_{1}}{L}\right)\sin\left(\frac{x_{2}}{L}\right)\cos\left(\frac{x_{3}}{L}\right),\\
&\fnc{U}_{3} = 0,\\
&\fnc{P} = \fnc{P}_{0}+\frac{\rho_{0}\fnc{V}_{0}^{2}}{16}\left[\cos\left(\frac{2x_{1}}{L}+\cos\left(\frac{2x_{2}}{L}\right)\right)\right]
\left[\cos\left(\frac{2x_{3}}{L}+2\right)\right].
\end{split}
\end{equation}
The flow is initialized to be isothermal, \ie, $\fnc{P}/\rho=\fnc{P}_{0}/\rho_{0}=R\fnc{T}_{0}$, and $\fnc{P}_{0}=1$, $\fnc{T}_{0}=1$, $L=1$, and $\fnc{V}_{0} = 1$. 
Finally, the Reynolds number is defined by $Re=(\rho_{0}\fnc{V}_{0})/\mu$, where $\mu$ is the dynamic viscosity.  

In order to obtain results that are reasonably close to those found for the incompressible 
equations, a Mach number of $M = 0.05$ is used. Furthermore, the Reynolds number, the Prandtl number, and the initial density distribution are set to $Re=1,600$, $Pr=0.71$, and $\rho_{0}=\gamma M^{2}$, respectively, where $\gamma=1.4$. 

For this test case, a grid is constructed as follow:
\begin{itemize}
\item Divide the computational domain with $N_{1h}$ hexahedral elements in each coordinate direction.
\item Refine random elements by using randomly one or two levels of $h$-refinement.
\item Assign the solution polynomial degree in each element to a random integer chosen uniformly from a set (see the legend in Figure \ref{fig:dkedt_tg}).
\item Construct the perturbed elements and their corresponding LGL points as described previously (the element interfaces are approximated using a polynomial degree which is the minimum solution polynomial degree used in the simulation).
\end{itemize}
Herein, two grids with $N_{1h}=4$ and $N_{1h}=8$ are considered. Their total number of hexahedrons is 869 and 7547, respectively.
The simulations are run without additional stabilization mechanisms (dissipation model, de-aliasing, filtering, \etc), where the only numerical dissipation originates from the upwind inter-element coupling procedure.

Figure \ref{fig:dkedt_tg} shows the time rate of change of the kinetic energy, $dke/dt$, for
the nonconforming algorithm using a random distribution of solution polynomial
order between i) $p=2$ and $p=8$ and ii) $p=7$ and $p=13$. The reference
DNS solution reported in \cite{de_wiart_tgv} is also plotted. We note that by increasing the 
order of accuracy of the solution polynomial in each cell and the grid density the
solution get closer to the DNS solution. The main 
take-away from the figure is that all simulations are stable, 
which is numerical evidence that the $h/p$ nonconforming 
scheme inherits the stability characteristics of the conforming 
and fully-staggered algorithms \cite{Carpenter2014,Parsani2015,Carpenter2016b,Carpenter2016,Parsani2016}.

\begin{figure}
   \centering
   \includegraphics[width=0.95\textwidth]{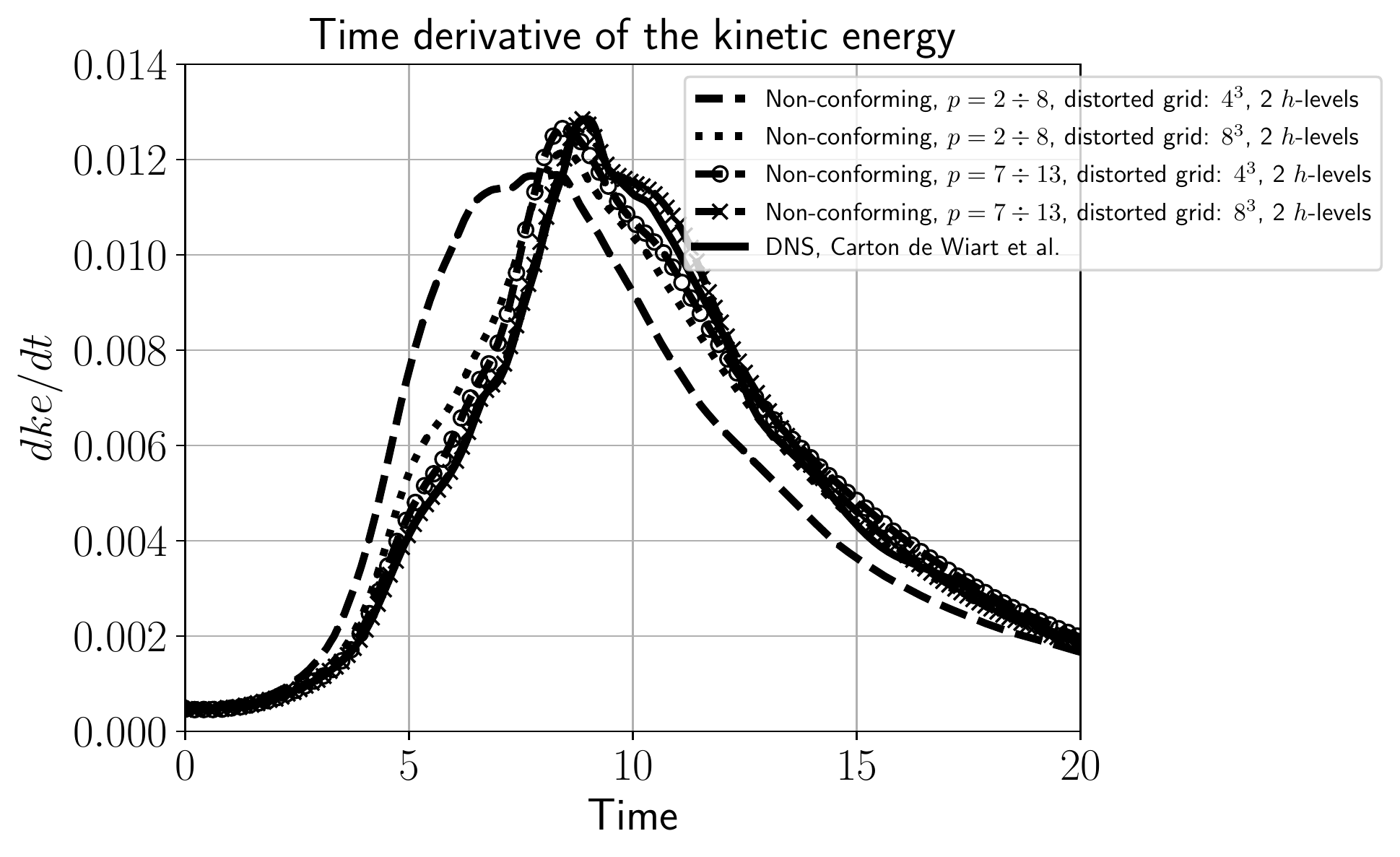}
   \caption{Evolution of the time derivative of the kinetic energy 
   for the Taylor--Green vortex at $Re = 1,600$, $M = 0.05$.}
   \label{fig:dkedt_tg}
\end{figure}


\subsection{Flow around a sphere at $Re=2,000$}

In this section, we test our implementation within a more complex setting represented by the flow around a sphere at $Re=2,000$ and $M=0.05$. With this value of the Reynolds number the flow is fully turbulent.
In this case, a sphere of diameter $d$ is centered at the origin of the axes, and a box is respectively extended $20d$ and $60d$ upstream
and downstream the direction of the flow; the box size is $30d$ in both the $x_2$ and $x_3$ directions. As boundary conditions, 
we consider adiabatic solid walls at the surface of the sphere \cite{dalcin_2019_wall_bc} and far field on all faces of the box.
We use a grid with 24,704 hexahedral elements.
Figure \ref{fig:mesh-sphere} shows the mesh near the sphere. The colors indicates the solution polynomial order used in each cell.
The quality of the elements is good in the boundary layer region whereas in the other portion of the domain is fairly poor. This
choice is intentional and is for the purpose of demonstrating the performance of the algorithm on non-ideal grids.

\begin{figure}
   \centering
   \includegraphics[width=0.65\textwidth]{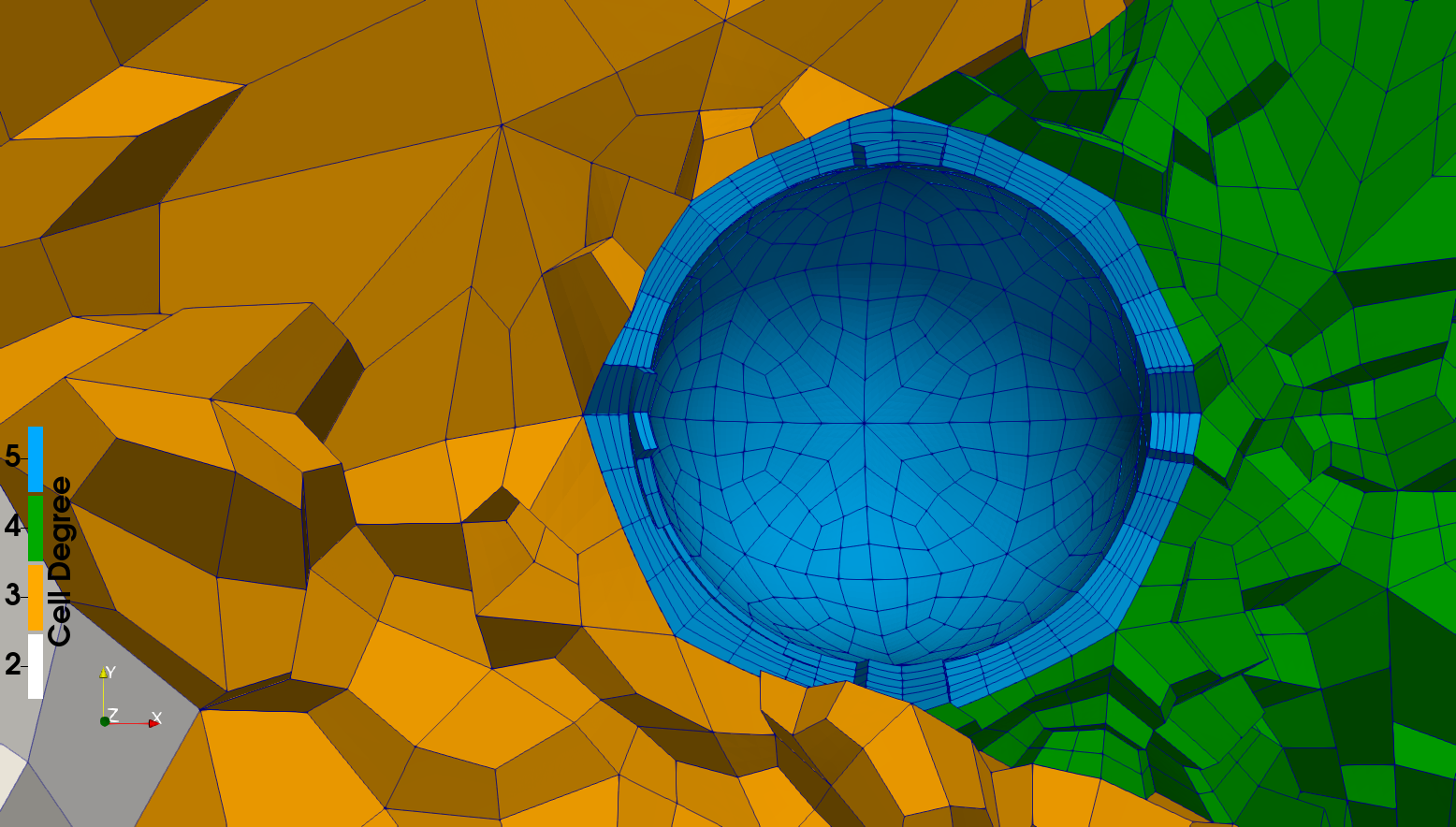}
   \caption{Polynomial degree distribution for the mesh around a sphere.}
   \label{fig:mesh-sphere}
\end{figure}

We compute the time-average value of the drag coefficient, $\langle C_D \rangle$, 
and we compare it with the value reported in in literature \cite{Munson_1990}.
From Table \ref{tab:sphere_cd}, it can be seen that the computed time-average 
drag coefficient matches very well the value reported in literature. 

\begin{table}[]
\begin{tabular}{||c|c||}
\hline \hline
                                 & $\langle C_D \rangle$ \\ \hline 
Munson et al. \cite{Munson_1990} & 0.412 \\ \hline
Present                          & 0.416 \\ \hline
\end{tabular}
\caption{Time-average drag coefficient of a sphere at $Re = 2,000$, $M = 0.05$.}.
\label{tab:sphere_cd}
\end{table}


\subsection{Entropy conservation of the fully-discrete explicit discretization}
To conclude the numerical results section, we demonstrate the entropy conservation 
of the fully-discrete explicit discretization of the compressible Navier--Stokes equations by integrating 
in time the system of ODEs which arise from the spatial discretization with an explicit relaxation 
Runge--Kutta scheme \cite{ranocha2019relaxation}. As shown in \cite{ketcheson2019relaxation,ranocha2019relaxation}, 
the term ``relaxation'' represents a general approach which
allows any Runge--Kutta method to preserve the correct time evolution of an
arbitrary functional, without sacrificing the linear covariance, accuracy, or
stability properties of the original method. In the context of the compressible
Euler and Navier--Stokes equations, the relaxation Runge--Kutta 
scheme is constructed to preserve the discrete entropy function obtained
from the spatial discretization. This leads to a fully discrete algorithm which is 
entropy conservative or entropy stable if the spatial discretization
is entropy conservative or entropy stable, respectively.

As a model problem, we again use the propagation of an isentropic vortex 
and we analyze the time evolution of the entropy function, which for the current
setting must be zero. The same grid and solution polynomial distribution shown in Figure \ref{fig:iv_mesh_cut} 
is used for this test case.
To achieve entropy conservation at the spatial level, all the dissipation terms 
used for the interface coupling are turned off, including upwind and interior penalty SATs.
\begin{figure}
   \centering
   \includegraphics[width=0.7\textwidth]{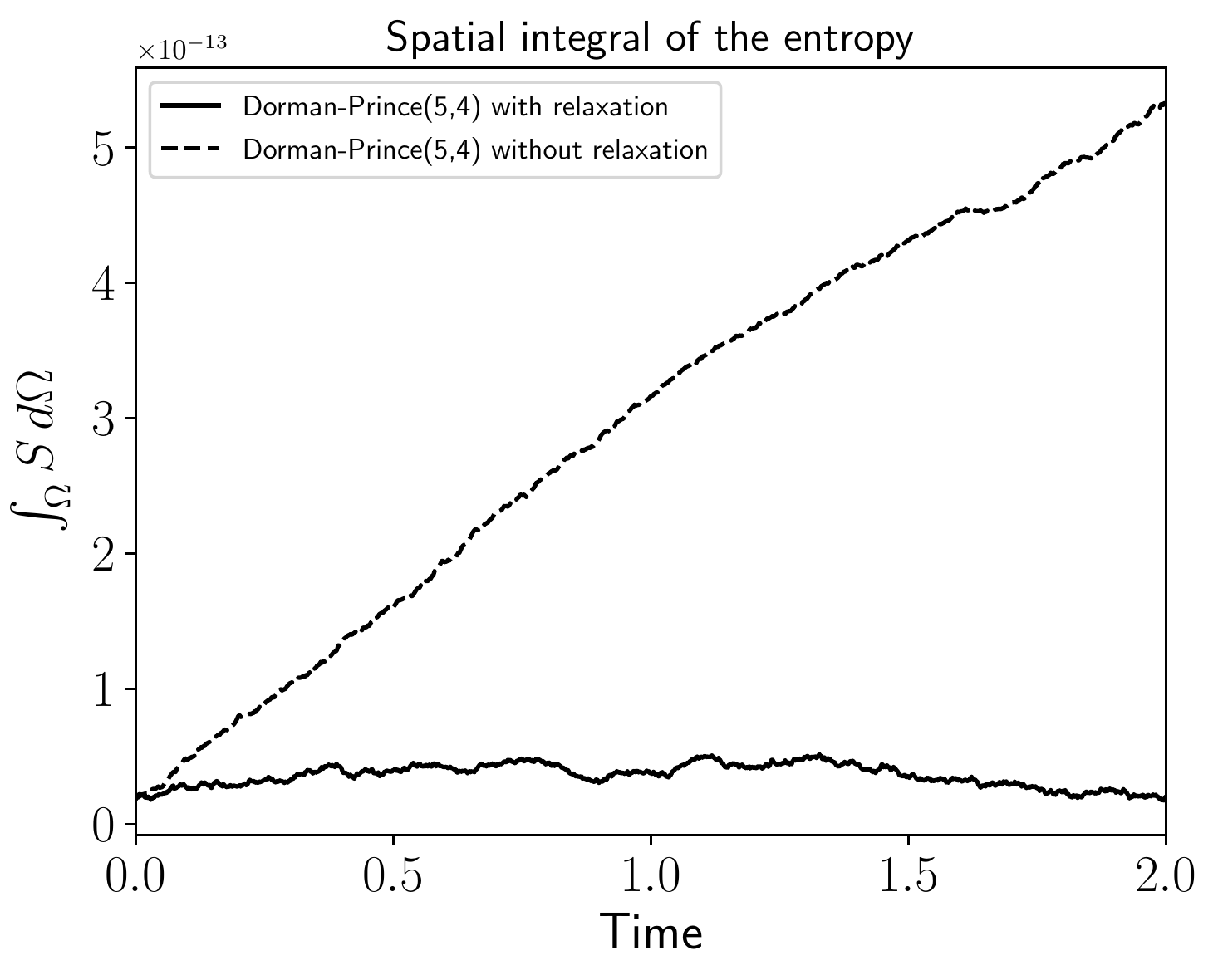}
   \caption{Evolution of the discrete spatial integral of the entropy function.}
   \label{fig:iv_hp_ref_rrk}
\end{figure}
The fourth-order accurate Dormand--Prince method \cite{dormand_rk_1980} with and 
without relaxation algorithm are used. We show the entropy variation in
Figure \ref{fig:iv_hp_ref_rrk}. The entropy is conserved up to machine (double)
precision using relaxation, whereas, without relaxation, the solution shows
significant essentially monotone changes in total entropy function. 

%
%
\section{Conclusions}\label{sec:conclusions}
In this paper, the $p$-refinement/coarsening algorithms in~\cite{Fernandez2019_p_ns,Fernandez2019_p_euler,Fernandez2018_TM} 
are extended to arbitrary $h/p$-refinement/coarsening. In order to obtain an algorithm for which 
the discrete GCL conditions are solved for element by element, the surface metric terms 
need to be localized to the small elements on an $h$-refined face. 
The discrete GCL conditions are then solved using the procedure in Crean~\etal~\cite{Crean2018}. 
The resulting algorithm is entropy conservative/stable, element-wise conservative, and free-stream 
preserving. 
Finally, the algorithm is shown to retain the accuracy and stability characteristics of the original 
conforming scheme on a set of test problems and, when coupled with relaxation Runge--Kutta schemes
\cite{ranocha2019relaxation}, yields a fully discrete entropy conservative/stable scheme. 

\begin{acknowledgements}
The research reported in this publication was
supported by funding from King Abdullah University of Science and Technology (KAUST).
We are thankful for the computing resources of the Supercomputing Laboratory and the Extreme 
Computing Research Center at KAUST.
Special thanks are extended to Dr. Mujeeb R. Malik for supporting this work as part of
NASA's ``Transformational Tools and Technologies'' ($T^3$) project.
\end{acknowledgements}
\bibliographystyle{spmpsci}
\bibliography{Bib}
\end{document}

%% file: def_v2.tex
\usepackage{xspace}
\usepackage{amsmath}
\usepackage{amsfonts}
\usepackage{amssymb}
\usepackage{bm}
\usepackage{xspace}
\usepackage{afterpage}
\usepackage{color}
\usepackage{soul}
\usepackage{cancel}
\usepackage{gensymb}
\usepackage{tikz}
\usetikzlibrary{matrix,backgrounds,shapes}
\usetikzlibrary{positioning}
\usetikzlibrary{fit}
\usetikzlibrary{plotmarks}
\usetikzlibrary{decorations.pathreplacing}
\usepackage{pgfplots}
\usetikzlibrary{shapes}
\usetikzlibrary{positioning,arrows}
\usetikzlibrary{fadings,shapes.arrows,shadows}  
\usetikzlibrary{arrows.meta}
\definecolor{darkorange}{rgb}{1.0, 0.55, 0.0}
\definecolor{do}{rgb}{1.0, 0.55, 0.0}
\definecolor{Royalblue}{rgb}{0.254,0.41,0.88}
\usepackage{xcolor}

\definecolor{darkorange}{rgb}{1.0, 0.55, 0.0}
\definecolor{royalblue}{rgb}{0.254,0.41,0.88}

\newcommand{\Tr}{\ensuremath{^{\mr{T}}}}

\newcommand{\mr}[1]{\ensuremath{\mathrm{#1}}}
\newcommand{\fnc}[1]{\ensuremath{\mathcal{#1}}}
\newcommand{\bfnc}[1]{\ensuremath{\bm{\mathcal{#1}}}}
\newcommand{\mat}[1]{\ensuremath{\mathsf{#1}}}
\newcommand{\etal}[0]{{\em et~al.\@}\xspace}
\newcommand{\eg}[0]{{e.g.\@}\xspace}
\newcommand{\ie}[0]{{i.e.\@}\xspace}
\newcommand{\etc}[0]{{etc.\@}\xspace}
\newcommand{\Theorem}[0]{Theorem}
\newcommand{\Eq}[0]{Eq.}
\newcommand{\Th}[0]{\ensuremath{^{\mathrm{th}}}}
\newtheorem{assume}{Assumption}
\newtheorem{thrm}{Theorem}
\DeclareMathOperator{\diag}{diag}

\newcommand{\pL}[0]{\ensuremath{p_{\mathrm{L}}}}

\newcommand{\qL}[0]{\ensuremath{\bm{q}_{\mathrm{L}}}}

\newcommand{\rmL}[0]{\mathrm{L}}

\newcommand{\xm}[1]{\ensuremath{x_{#1}}}
\newcommand{\xil}[1]{\ensuremath{\xi_{#1}}}
\newcommand{\xilhat}[1]{\ensuremath{\hat{\xi}_{#1}}}
\newcommand{\alphal}[1]{\ensuremath{\alpha_{#1}}}
\newcommand{\betal}[1]{\ensuremath{\beta_{#1}}}
\newcommand{\Nl}[1]{\ensuremath{N_{#1}}}
\newcommand{\bxil}[1]{\ensuremath{\bm{\xi}_{#1}}}

\newcommand{\bxili}[2]{\ensuremath{\bm{\xi}_{#1}^{(#2)}}}

\newcommand{\bmui}[1]{\ensuremath{\bm{u}^{(#1)}}}

\newcommand{\Q}[0]{\ensuremath{\bm{\fnc{Q}}}}
\newcommand{\Jk}[0]{\ensuremath{\fnc{J}_{\kappa}}}
\newcommand{\Jkhat}[0]{\ensuremath{\hat{\fnc{J}}_{\kappa}}}
\newcommand{\Qk}[0]{\ensuremath{\bm{\fnc{Q}_{\kappa}}}}
\newcommand{\Sk}[0]{\ensuremath{\fnc{S}_{\kappa}}}
\newcommand{\Jdxildxm}[2]{\ensuremath{\Jk\frac{\partial\xil{#1}}{\partial\xm{#2}}}}
\newcommand{\Jdxildxmhat}[2]{\ensuremath{\Jkhat\frac{\partial\xilhat{#1}}{\partial\xm{#2}}}}
\newcommand{\Fxm}[1]{\ensuremath{\bm{\fnc{F}}_{\xm{#1}}}}
\newcommand{\FxmI}[1]{\ensuremath{\bm{\fnc{F}}_{\xm{#1}}^{I}}}
\newcommand{\FxmV}[1]{\ensuremath{\bm{\fnc{F}}_{\xm{#1}}^{V}}}
\newcommand{\fxm}[1]{\ensuremath{\fnc{F}_{\xm{#1}}}}
\newcommand{\Um}[1]{\ensuremath{\fnc{U}_{#1}}}
\newcommand{\E}[0]{\ensuremath{\fnc{E}}}
\newcommand{\nxm}[1]{\ensuremath{n_{\xm{#1}}}}
\newcommand{\nxil}[1]{\ensuremath{n_{\xil{#1}}}}
\newcommand{\nxilhat}[1]{\ensuremath{n_{\xilhat{#1}}}}

\newcommand{\DxiloneD}[1]{\ensuremath{\mat{D}_{\xil{#1}}^{(1D)}}}
\newcommand{\PxiloneD}[1]{\ensuremath{\mat{P}_{\xil{#1}}^{(1D)}}}
\newcommand{\QxiloneD}[1]{\ensuremath{\mat{Q}_{\xil{#1}}^{(1D)}}}
\newcommand{\SxiloneD}[1]{\ensuremath{\mat{S}_{\xil{#1}}^{(1D)}}}
\newcommand{\ExiloneD}[1]{\ensuremath{\mat{E}_{\xil{#1}}^{(1D)}}}
\newcommand{\txilalpha}[1]{\ensuremath{\bm{t}_{\alphal}}}
\newcommand{\txilbeta}[1]{\ensuremath{\bm{t}_{\betal}}}
\newcommand{\DoneD}[0]{\mat{D}^{(1D)}}
\newcommand{\PoneD}[0]{\mat{P}^{(1D)}}
\newcommand{\QoneD}[0]{\mat{Q}^{(1D)}}

\newcommand{\Imat}[1]{\ensuremath{\mat{I}_{#1}}}
\newcommand{\M}[0]{\ensuremath{\mat{P}}}

\newcommand{\Dxm}[1]{\ensuremath{\mat{D}_{\xm{#1}}}}
\newcommand{\Exm}[1]{\ensuremath{\mat{E}_{\xm{#1}}}}
\newcommand{\Dxil}[1]{\ensuremath{\mat{D}_{\xil{#1}}}}
\newcommand{\Qxil}[1]{\ensuremath{\mat{Q}_{\xil{#1}}}}

\newcommand{\Exil}[1]{\ensuremath{\mat{E}_{\xil{#1}}}}

\newcommand{\Porthohatl}[1]{\ensuremath{\mat{P}_{\perp\xilhat{#1}}}}

\newcommand{\barDxilhat}[1]{\ensuremath{\overline{\mat{D}}_{\xilhat{#1}}}}

\newcommand{\matFxm}[3]{\ensuremath{\mat{F}_{\xm{#1}}\left(#2,#3\right)}}

\newcommand{\qk}[1]{\ensuremath{\bm{q}_{#1}}}
\newcommand{\qki}[2]{\ensuremath{\bm{q}_{#1}^{#2}}}
\newcommand{\tildeqi}[1]{\ensuremath{\tilde{\bm{q}}^{#1}}}
\newcommand{\wki}[2]{\ensuremath{\bm{w}_{#1}^{#2}}}
\newcommand{\tildewi}[1]{\ensuremath{\tilde{\bm{w}}^{#1}}}

\newcommand{\fxmsc}[3]{\ensuremath{\bm{f}_{\xm{#1}}^{sc}\left(#2,#3\right)}}

\newcommand{\ones}[1]{\ensuremath{\bm{1}_{#1}}}
\newcommand{\barones}[1]{\ensuremath{\overline{\bm{1}}_{#1}}}
\newcommand{\wk}[1]{\ensuremath{\bm{w}_{#1}}}
\newcommand{\psixmk}[2]{\ensuremath{\bm{\psi}_{\xm{#1}}^{#2}}}
\newcommand{\psixmki}[3]{\ensuremath{\left(\bm{\psi}_{\xm{#1}}^{#2}\right)^{(#3)}}}

\newcommand{\matJk}[1]{\ensuremath{{\color{orange}\mat{J}}_{#1}}}

\newcommand{\matAlmk}[3]{\ensuremath{\left[{\color{blue}\fnc{J}\frac{\partial\xil{#1}}{\partial\xm{#2}}}\right]_{#3}}}

\newcommand{\RL}[0]{\ensuremath{\mat{R}_{\rmL}}}

\newcommand{\Ok}[0]{\ensuremath{\Omega_{\kappa}}}
\newcommand{\pOk}[0]{\ensuremath{\partial\Omega_{\kappa}}}
\newcommand{\Ohat}[0]{\ensuremath{\hat{\Omega}}}
\newcommand{\Ohatk}[0]{\ensuremath{\hat{\Omega}_{\kappa}}}
\newcommand{\Ghatk}[0]{\ensuremath{\hat{\Gamma}_{\kappa}}}
\newcommand{\Ghat}[0]{\ensuremath{\hat{\Gamma}}}
\newcommand{\pOhatk}[0]{\ensuremath{\partial\hat{\Omega}_{\kappa}}}

\newcommand{\dissL}[0]{\ensuremath{\bm{diss}_{\mathrm{L}}}}
\newcommand{\dissRf}[1]{\ensuremath{\bm{diss}_{\Rf{#1}}}}

\newcommand{\uk}[0]{\ensuremath{\bm{u}_{\kappa}}}
\newcommand{\ur}[0]{\ensuremath{\bm{u}_{r}}}
\newcommand{\uki}[1]{\ensuremath{\bm{u}_{\kappa}^{#1}}}
\newcommand{\uri}[1]{\ensuremath{\bm{u}_{r}^{#1}}}
\newcommand{\uL}[0]{\ensuremath{\bm{u}_{\rmL}}}
\newcommand{\uRf}[1]{\ensuremath{\bm{u}_{\Rf{#1}}}}

\newcommand{\Cij}[2]{\ensuremath{\mat{C}_{#1,#2}}}
\newcommand{\Chatij}[2]{\ensuremath{\hat{\mat{C}}_{#1,#2}}}

\newcommand{\Jtilde}[0]{\ensuremath{{\color{orange}\hat{\mat{J}}}}}
\newcommand{\utilde}[0]{\ensuremath{\widetilde{\bm{u}}}}
\newcommand{\qtilde}[0]{\ensuremath{\widetilde{\bm{q}}}}
\newcommand{\wtilde}[0]{\ensuremath{\widetilde{\bm{w}}}}

\newcommand{\psitildem}[1]{\ensuremath{\widetilde{\bm{\psi}_{\xm{#1}}}}}
\newcommand{\Mtilde}[0]{\ensuremath{\widetilde{\mat{P}}}}

\newcommand{\Dtildel}[1]{\ensuremath{\widetilde{\mat{D}}_{\xilhat{#1}}}}

\newcommand{\Stildel}[1]{\ensuremath{\widetilde{\mat{S}}_{\xilhat{#1}}}}
\newcommand{\Etildel}[1]{\ensuremath{\widetilde{\mat{E}}_{\xilhat{#1}}}}

\newcommand{\tildeone}[0]{\ensuremath{\tilde{\bm{1}}}}

\newcommand{\eonel}[1]{\ensuremath{\bm{e}_{1_{#1}}}}
\newcommand{\eNl}[1]{\ensuremath{\bm{e}_{\Nl{#1}}}}
\newcommand{\bmxi}[1]{\ensuremath{\bm{\xi}^{(#1)}}}
\newcommand{\U}[0]{\ensuremath{\fnc{U}}}

\newcommand{\OhatL}[0]{\ensuremath{\hat{\Omega}_{\rmL}}}

\newcommand{\GammahatL}[0]{\ensuremath{\hat{\Gamma}_{\rmL}}}
\newcommand{\Gammahat}[0]{\ensuremath{\hat{\Gamma}}}

\newcommand{\Dtildem}[1]{\tilde{\mat{D}}_{#1}}
\newcommand{\Qtildem}[1]{\tilde{\mat{Q}}_{#1}}
\newcommand{\Etildem}[1]{\tilde{\mat{E}}_{#1}}

\newcommand{\barQtildem}[1]{\overline{\tilde{\mat{Q}}}_{#1}}
\newcommand{\barEtildem}[1]{\overline{\tilde{\mat{E}}}_{#1}}

\newcommand{\am}[1]{\ensuremath{a_{#1}}}
\newcommand{\Bm}[1]{\ensuremath{b_{#1}}}
\newcommand{\Chatla}[2]{\ensuremath{\hat{\mat{C}}_{#1,#2}}}

\newcommand{\matChatla}[2]{\ensuremath{\left[{\color{blue}\Chatla{#1}{#2}}\right]}}
\newcommand{\matCtildela}[2]{\ensuremath{\left[{\color{blue}\Chatla{#1}{#2}}\right]}}

\newcommand{\Deltalk}[2]{\ensuremath{\Delta_{#1}^{#2}}}
\newcommand{\xilHkhat}[2]{\ensuremath{\hat{\mathrm{H}}_{#2}^{#1}}}
\newcommand{\xilLkhat}[2]{\ensuremath{\hat{\mathrm{L}}_{#2}^{#1}}}
\newcommand{\matJkhat}[1]{\ensuremath{{\color{orange}\mat{\hat{J}}}_{#1}}}
\newcommand{\matAlmkhat}[3]{\ensuremath{\left[{\color{blue}\Jdxildxmhat{#1}{#2}}\right]_{#3}}}
\newcommand{\matAlmkhatgamma}[5]{\ensuremath{\left[{\color{#5}\Jdxildxmhat{#1}{#2}}\right]_{#3}^{#4}}}
\newcommand{\pRf}[1]{\ensuremath{p_{\Rf{#1}}}}
\newcommand{\Rf}[1]{\ensuremath{\mathrm{R}_{#1}}}
\newcommand{\qRf}[1]{\ensuremath{\bm{q}_{\Rf{#1}}}}
\newcommand{\IRftoLoneD}[2]{\ensuremath{\mat{I}_{\Rf{#1}\rm{to}\rmL}^{#2}}}
\newcommand{\ILtoRfoneD}[2]{\ensuremath{\mat{I}_{\rmL\rm{to}\Rf{#1}}^{#2}}}
\newcommand{\IRftoL}[1]{\ensuremath{\mat{I}_{\Rf{#1}\rm{to}\rmL}}}
\newcommand{\ILtoRf}[1]{\ensuremath{\mat{I}_{\rmL\rm{to}\Rf{#1}}}}
\newcommand{\DtildehatH}[1]{\ensuremath{\widetilde{\mat{D}}_{\xilhat{#1}}}}
\newcommand{\Dtildehatl}[1]{\ensuremath{\widetilde{\mat{D}}_{\xilhat{#1}}}}
\newcommand{\Qtildehatl}[1]{\ensuremath{\widetilde{\mat{Q}}_{\xilhat{#1}}}}
\newcommand{\Stildehatl}[1]{\ensuremath{\widetilde{\mat{S}}_{\xilhat{#1}}}}
\newcommand{\Etildehatl}[1]{\ensuremath{\widetilde{\mat{E}}_{\xilhat{#1}}}}
\newcommand{\Dxilhat}[1]{\ensuremath{\mat{D}_{\xilhat{#1}}}}
\newcommand{\Qxilhat}[1]{\ensuremath{\mat{Q}_{\xilhat{#1}}}}
\newcommand{\Sxilhat}[1]{\ensuremath{\mat{S}_{\xilhat{#1}}}}
\newcommand{\Exilhat}[1]{\ensuremath{\mat{E}_{\xilhat{#1}}}}

\newcommand{\amk}[2]{\ensuremath{\bm{a}_{#1}^{#2}}}
\newcommand{\cmk}[2]{\ensuremath{\bm{c}_{#1}^{#2}}}
\newcommand{\Mkopt}[1]{\ensuremath{\mat{M}^{#1}}}
\newcommand{\Uk}[1]{\ensuremath{\mat{U}^{#1}}}
\newcommand{\Vk}[1]{\ensuremath{\mat{V}^{#1}}}
\newcommand{\Sigmak}[1]{\ensuremath{\mat{\Sigma}^{#1}}}

\newcommand{\wRf}[1]{\ensuremath{\bm{w}_{\Rf{#1}}}}
\newcommand{\wL}[0]{\ensuremath{\bm{w}_{\mathrm{L}}}}
\newcommand{\RRf}[1]{\ensuremath{\mat{R}}_{\Rf{#1}}}
\newcommand{\IPL}[0]{\ensuremath{\bm{I}_{\mathrm{P}}^{\mathrm{L}}}}
\newcommand{\IPRf}[1]{\ensuremath{\bm{I}_{\mathrm{P}}^{\Rf{#1}}}}
\newcommand{\matJRftilde}[1]{\ensuremath{\tilde{\mat{J}}_{\Rf{f}}}}
\newcommand{\CtildeRfij}[3]{\ensuremath{\tilde{\mat{C}}_{#1,#2}^{\Rf{#3}}}}

\newcommand{\Dtildelm}[2]{\ensuremath{\widetilde{\mat{D}}_{#1,#2}}}

%% file: non_conforming_face.tex
%
%
\begin{tikzpicture}[scale=0.35]
\node(left) at (8,7) {Left element:};
\node(right) at (33.5,7) {Right elements:};
\node(opL1) at (8,5.5) {$\pL$,};
\node(opL2) at (8,4) {$\Dxil{l}^{\rmL},\Qxil{l}^{\rmL},\Exil{l}^{\rmL}$};
\node(opL3) at (14.5,4) {$\mathrm{L}$};
\node(opR1) at (33.5,5.5) {$\pRf{f}$,};
\node(opR2) at (33.5,4) {$\Dxil{l}^{\Rf{f}},\Qxil{l}^{\Rf{f}},\Exil{l}^{\Rf{f}}$,};
\node(opR3) at (33.5,2.5) {$f=1,2,3,4$};
\node(opR4) at (25,2.5) {$\Rf{1}$};
\node(opR5) at (25,5) {$\Rf{2}$};
\node(opR5) at (28,2.5) {$\Rf{3}$};
\node(opR5) at (28,5) {$\Rf{4}$};
\draw[-latex, line width = 0.5mm,] (19.5,2,0) -- (21.5,2,0);
\draw[-latex, line width = 0.5mm,] (19.5,2,0) -- (19.5,4,0);
\draw (21.5,2,0) node[anchor=west] {\small$\xil{2}$};
\draw (19.5,4,0) node[anchor=south] {\small$\xil{3}$};

\begin{axis}[%
width=12.115in,
height=2.548in,
at={(2.032in,0.344in)},
scale only axis,
xmin=-4,
xmax=4,
xtick={\empty},
ymin=-1,
ymax=1,
ytick={\empty},
axis line style={draw=none},
ticks=none
]
\addplot [color=darkorange, line width=2.0pt, forget plot]
  table[row sep=crcr]{%
-2.40521804601421	-0.923586335148351\\
-2.38797489170308	-0.92746114879694\\
-2.37072671589107	-0.931134492676316\\
-2.35347275708605	-0.934579417158512\\
-2.33621225851262	-0.93777064842321\\
-2.31894446872928	-0.940684773880094\\
-2.30166864224136	-0.943300413936212\\
-2.28438404010905	-0.945598378848181\\
-2.26708993055006	-0.94756180950847\\
-2.24978558953651	-0.949176301132889\\
-2.23247030138543	-0.950430008941843\\
-2.2151433593425	-0.951313735060006\\
-2.19780406615848	-0.951820995996888\\
-2.18045173465789	-0.95194807021321\\
-2.16308568829951	-0.95169402542412\\
-2.14570526172813	-0.951060725438937\\
-2.12830980131726	-0.950052816487241\\
-2.11089866570222	-0.948677693131635\\
-2.09347122630326	-0.946945444017248\\
-2.07602686783821	-0.94486877785601\\
-2.05856498882437	-0.942462930188695\\
-2.04108500206901	-0.939745551608799\\
-2.02358633514835	-0.93673657826828\\
-2.00606843087444	-0.933458085615216\\
-1.98853074774955	-0.929934126436441\\
-1.97097276040791	-0.926190554393348\\
-1.95339396004413	-0.922254834345515\\
-1.93579385482823	-0.91815584085371\\
-1.91817197030677	-0.913923646340575\\
-1.90052784978983	-0.909589300463139\\
-1.88286105472347	-0.905184602315829\\
-1.86517116504745	-0.900741867135191\\
-1.84745777953784	-0.89629368921792\\
-1.82972051613422	-0.891872702791559\\
-1.81195901225135	-0.887511342592238\\
-1.79417292507484	-0.883241605905987\\
-1.77636193184082	-0.879094817819428\\
-1.75852573009914	-0.875101401402074\\
-1.74066403796007	-0.871290654506319\\
-1.72277659432423	-0.867690534822617\\
-1.7048631590955	-0.864327454766822\\
-1.68692351337691	-0.861226087704482\\
-1.66895745964911	-0.858409186933756\\
-1.65096482193153	-0.855897418754962\\
-1.63294544592587	-0.853709210851482\\
-1.61489919914195	-0.851860617094352\\
-1.5968259710058	-0.850365199762426\\
-1.57872567294982	-0.849233930042196\\
-1.56059823848502	-0.848475107537264\\
-1.54244362325528	-0.848094299377979\\
-1.52426180507346	-0.848094299377979\\
-1.50605278393957	-0.848475107537264\\
-1.48781658204073	-0.849233930042196\\
-1.46955324373307	-0.850365199762426\\
-1.45126283550558	-0.851860617094352\\
-1.43294544592587	-0.853709210851482\\
-1.4146011855679	-0.855897418754962\\
-1.39623018692184	-0.858409186933756\\
-1.377832604286	-0.861226087704482\\
-1.35940861364096	-0.864327454766822\\
-1.34095841250605	-0.867690534822617\\
-1.32248221977826	-0.871290654506319\\
-1.30398027555369	-0.875101401402074\\
-1.28545284093173	-0.879094817819428\\
-1.26690019780211	-0.883241605905987\\
-1.24832264861498	-0.887511342592238\\
-1.22972051613422	-0.891872702791559\\
-1.2110941431742	-0.89629368921792\\
-1.19244389232018	-0.900741867135191\\
-1.17377014563256	-0.905184602315829\\
-1.15507330433528	-0.909589300463139\\
-1.13635378848859	-0.913923646340575\\
-1.11761203664642	-0.91815584085371\\
-1.09884850549868	-0.922254834345515\\
-1.08006366949882	-0.926190554393348\\
-1.06125802047683	-0.929934126436441\\
-1.04243206723807	-0.933458085615216\\
-1.02358633514835	-0.93673657826828\\
-1.00472136570537	-0.939745551608799\\
-0.985837716097094	-0.942462930188695\\
-0.966935958747304	-0.94486877785601\\
-0.948016680848712	-0.946945444017248\\
-0.92908048388404	-0.948677693131635\\
-0.910127983135443	-0.950052816487241\\
-0.891159807182678	-0.951060725438937\\
-0.872176597390422	-0.95169402542412\\
-0.853179007385167	-0.95194807021321\\
-0.834167702522113	-0.951820995996888\\
-0.815143359342497	-0.951313735060006\\
-0.79610666502179	-0.950430008941843\\
-0.777058316809233	-0.949176301132889\\
-0.757999021459149	-0.94756180950847\\
-0.738929494654503	-0.945598378848181\\
-0.71985046042318	-0.943300413936212\\
-0.700762650547459	-0.940684773880094\\
-0.681666803967161	-0.93777064842321\\
-0.662563666176959	-0.934579417158512\\
-0.643453988618341	-0.931134492676316\\
-0.624338528066722	-0.92746114879694\\
-0.605218046014211	-0.923586335148351\\
};
\addplot [color=darkorange, line width=2.0pt, forget plot]
  table[row sep=crcr]{%
-2.40521804601421	0.87641366485165\\
-2.38797489170308	0.87253885120306\\
-2.37072671589107	0.868865507323684\\
-2.35347275708605	0.865420582841488\\
-2.33621225851262	0.86222935157679\\
-2.31894446872928	0.859315226119906\\
-2.30166864224136	0.856699586063788\\
-2.28438404010905	0.854401621151819\\
-2.26708993055006	0.85243819049153\\
-2.24978558953651	0.850823698867111\\
-2.23247030138543	0.849569991058157\\
-2.2151433593425	0.848686264939994\\
-2.19780406615848	0.848179004003112\\
-2.18045173465789	0.84805192978679\\
-2.16308568829951	0.84830597457588\\
-2.14570526172813	0.848939274561063\\
-2.12830980131726	0.849947183512759\\
-2.11089866570222	0.851322306868365\\
-2.09347122630326	0.853054555982752\\
-2.07602686783821	0.85513122214399\\
-2.05856498882437	0.857537069811305\\
-2.04108500206901	0.860254448391201\\
-2.02358633514835	0.86326342173172\\
-2.00606843087444	0.866541914384784\\
-1.98853074774955	0.870065873563559\\
-1.97097276040791	0.873809445606652\\
-1.95339396004413	0.877745165654485\\
-1.93579385482823	0.88184415914629\\
-1.91817197030677	0.886076353659425\\
-1.90052784978983	0.890410699536861\\
-1.88286105472347	0.894815397684171\\
-1.86517116504745	0.899258132864809\\
-1.84745777953784	0.90370631078208\\
-1.82972051613422	0.908127297208441\\
-1.81195901225135	0.912488657407762\\
-1.79417292507484	0.916758394094013\\
-1.77636193184082	0.920905182180572\\
-1.75852573009914	0.924898598597926\\
-1.74066403796007	0.928709345493681\\
-1.72277659432423	0.932309465177383\\
-1.7048631590955	0.935672545233178\\
-1.68692351337691	0.938773912295518\\
-1.66895745964911	0.941590813066244\\
-1.65096482193153	0.944102581245038\\
-1.63294544592587	0.946290789148518\\
-1.61489919914195	0.948139382905648\\
-1.5968259710058	0.949634800237574\\
-1.57872567294982	0.950766069957804\\
-1.56059823848502	0.951524892462736\\
-1.54244362325528	0.951905700622021\\
-1.52426180507346	0.951905700622021\\
-1.50605278393957	0.951524892462736\\
-1.48781658204073	0.950766069957804\\
-1.46955324373307	0.949634800237574\\
-1.45126283550558	0.948139382905648\\
-1.43294544592587	0.946290789148518\\
-1.4146011855679	0.944102581245038\\
-1.39623018692184	0.941590813066244\\
-1.377832604286	0.938773912295518\\
-1.35940861364096	0.935672545233178\\
-1.34095841250605	0.932309465177383\\
-1.32248221977826	0.928709345493681\\
-1.30398027555369	0.924898598597926\\
-1.28545284093173	0.920905182180572\\
-1.26690019780211	0.916758394094013\\
-1.24832264861498	0.912488657407762\\
-1.22972051613422	0.908127297208441\\
-1.2110941431742	0.90370631078208\\
-1.19244389232018	0.899258132864809\\
-1.17377014563256	0.894815397684171\\
-1.15507330433528	0.890410699536861\\
-1.13635378848859	0.886076353659425\\
-1.11761203664642	0.88184415914629\\
-1.09884850549868	0.877745165654485\\
-1.08006366949882	0.873809445606652\\
-1.06125802047683	0.870065873563559\\
-1.04243206723807	0.866541914384784\\
-1.02358633514835	0.86326342173172\\
-1.00472136570537	0.860254448391201\\
-0.985837716097094	0.857537069811305\\
-0.966935958747304	0.85513122214399\\
-0.948016680848712	0.853054555982752\\
-0.92908048388404	0.851322306868365\\
-0.910127983135443	0.849947183512759\\
-0.891159807182678	0.848939274561063\\
-0.872176597390422	0.84830597457588\\
-0.853179007385167	0.84805192978679\\
-0.834167702522113	0.848179004003112\\
-0.815143359342497	0.848686264939994\\
-0.79610666502179	0.849569991058157\\
-0.777058316809233	0.850823698867111\\
-0.757999021459149	0.85243819049153\\
-0.738929494654503	0.854401621151819\\
-0.71985046042318	0.856699586063788\\
-0.700762650547459	0.859315226119906\\
-0.681666803967161	0.86222935157679\\
-0.662563666176959	0.865420582841488\\
-0.643453988618341	0.868865507323684\\
-0.624338528066722	0.87253885120306\\
-0.605218046014211	0.87641366485165\\
};
\addplot [color=darkorange, line width=2.0pt, forget plot]
  table[row sep=crcr]{%
-2.40521804601421	-0.923586335148351\\
-2.40607527778793	-0.908521127555506\\
-2.40688793805364	-0.892912324992386\\
-2.40765006469817	-0.876707548280959\\
-2.40835606635025	-0.859866550558203\\
-2.40900076340172	-0.84236203658618\\
-2.409579426008	-0.824180218404362\\
-2.41008780878872	-0.805321096012749\\
-2.41052218197415	-0.785798457371868\\
-2.4108793587688	-0.765639597719659\\
-2.41115671873144	-0.744884763919142\\
-2.41135222700008	-0.723586335148351\\
-2.41146444922084	-0.701807756611312\\
-2.41149256207109	-0.679622247976659\\
-2.41143635929983	-0.657111312836321\\
-2.41129625324085	-0.634363079527432\\
-2.41107327178761	-0.611470507093524\\
-2.41076905085207	-0.588529492906476\\
-2.41038582236277	-0.565636920472569\\
-2.4099263978902	-0.542888687163679\\
-2.40939414801957	-0.520377752023341\\
-2.40879297762244	-0.498192243388688\\
-2.40812729720844	-0.47641366485165\\
-2.4074019905674	-0.455115236080858\\
-2.40662237893925	-0.434360402280341\\
-2.40579418197455	-0.414201542628132\\
-2.40492347577202	-0.394678903987251\\
-2.40401664830104	-0.375819781595638\\
-2.40308035253607	-0.35763796341382\\
-2.40212145764681	-0.340133449441797\\
-2.40114699860234	-0.323292451719041\\
-2.40016412455871	-0.307087675007614\\
-2.39918004640898	-0.291478872444494\\
-2.39820198388015	-0.27641366485165\\
-2.39723711256541	-0.261828607025052\\
-2.39629251128008	-0.247650479296068\\
-2.39537511012764	-0.23379777807277\\
-2.39449163965677	-0.220182375018023\\
-2.39364858148239	-0.206711311088295\\
-2.39285212073305	-0.193288688911705\\
-2.39210810067355	-0.179817624981977\\
-2.39142197983551	-0.16620222192723\\
-2.39079879197073	-0.152349520703932\\
-2.39024310912083	-0.138171392974948\\
-2.3897590080743	-0.123586335148351\\
-2.38935004045701	-0.108521127555506\\
-2.38901920667553	-0.0929123249923863\\
-2.38876893390455	-0.0767075482809588\\
-2.38860105827986	-0.0598665505582033\\
-2.38851681142738	-0.0423620365861802\\
-2.38851681142738	-0.024180218404362\\
-2.38860105827986	-0.00532109601274876\\
-2.38876893390455	0.0142015426281321\\
-2.38901920667553	0.034360402280341\\
-2.38935004045701	0.0551152360808576\\
-2.3897590080743	0.0764136648516495\\
-2.39024310912083	0.0981922433886878\\
-2.39079879197073	0.120377752023341\\
-2.39142197983551	0.142888687163679\\
-2.39210810067355	0.165636920472569\\
-2.39285212073305	0.188529492906476\\
-2.39364858148239	0.211470507093524\\
-2.39449163965677	0.234363079527431\\
-2.39537511012764	0.257111312836321\\
-2.39629251128008	0.279622247976659\\
-2.39723711256541	0.301807756611312\\
-2.39820198388015	0.32358633514835\\
-2.39918004640898	0.344884763919142\\
-2.40016412455871	0.365639597719659\\
-2.40114699860234	0.385798457371868\\
-2.40212145764681	0.405321096012749\\
-2.40308035253607	0.424180218404362\\
-2.40401664830104	0.44236203658618\\
-2.40492347577202	0.459866550558203\\
-2.40579418197455	0.476707548280959\\
-2.40662237893925	0.492912324992386\\
-2.4074019905674	0.508521127555506\\
-2.40812729720844	0.52358633514835\\
-2.40879297762244	0.538171392974948\\
-2.40939414801957	0.552349520703932\\
-2.4099263978902	0.56620222192723\\
-2.41038582236277	0.579817624981977\\
-2.41076905085207	0.593288688911705\\
-2.41107327178761	0.606711311088295\\
-2.41129625324085	0.620182375018023\\
-2.41143635929983	0.63379777807277\\
-2.41149256207109	0.647650479296068\\
-2.41146444922084	0.661828607025052\\
-2.41135222700008	0.676413664851649\\
-2.41115671873144	0.691478872444494\\
-2.4108793587688	0.707087675007614\\
-2.41052218197415	0.723292451719041\\
-2.41008780878872	0.740133449441797\\
-2.409579426008	0.75763796341382\\
-2.40900076340172	0.775819781595638\\
-2.40835606635025	0.794678903987251\\
-2.40765006469817	0.814201542628132\\
-2.40688793805364	0.834360402280341\\
-2.40607527778793	0.855115236080857\\
-2.40521804601421	0.87641366485165\\
};
\addplot [color=darkorange, line width=2.0pt, forget plot]
  table[row sep=crcr]{%
-0.605218046014211	-0.923586335148351\\
-0.606075277787933	-0.908521127555506\\
-0.606887938053635	-0.892912324992386\\
-0.607650064698174	-0.876707548280959\\
-0.608356066350251	-0.859866550558203\\
-0.609000763401725	-0.84236203658618\\
-0.609579426008	-0.824180218404362\\
-0.61008780878872	-0.805321096012749\\
-0.610522181974154	-0.785798457371868\\
-0.610879358768801	-0.765639597719659\\
-0.611156718731439	-0.744884763919142\\
-0.611352227000084	-0.723586335148351\\
-0.611464449220841	-0.701807756611312\\
-0.611492562071091	-0.679622247976659\\
-0.611436359299834	-0.657111312836321\\
-0.611296253240851	-0.634363079527432\\
-0.611073271787607	-0.611470507093524\\
-0.610769050852068	-0.588529492906476\\
-0.610385822362771	-0.565636920472569\\
-0.609926397890197	-0.542888687163679\\
-0.609394148019571	-0.520377752023341\\
-0.608792977622444	-0.498192243388688\\
-0.608127297208441	-0.47641366485165\\
-0.607401990567399	-0.455115236080858\\
-0.606622378939251	-0.434360402280341\\
-0.605794181974546	-0.414201542628132\\
-0.604923475772015	-0.394678903987251\\
-0.604016648301039	-0.375819781595638\\
-0.603080352536066	-0.35763796341382\\
-0.602121457646813	-0.340133449441797\\
-0.601146998602336	-0.323292451719041\\
-0.600164124558712	-0.307087675007614\\
-0.599180046408981	-0.291478872444494\\
-0.598201983880155	-0.27641366485165\\
-0.59723711256541	-0.261828607025052\\
-0.596292511280077	-0.247650479296068\\
-0.595375110127641	-0.23379777807277\\
-0.594491639656769	-0.220182375018023\\
-0.593648581482385	-0.206711311088295\\
-0.592852120733054	-0.193288688911705\\
-0.592108100673548	-0.179817624981977\\
-0.591421979835509	-0.16620222192723\\
-0.590798791970728	-0.152349520703932\\
-0.590243109120826	-0.138171392974948\\
-0.589759008074301	-0.123586335148351\\
-0.589350040457008	-0.108521127555506\\
-0.589019206675526	-0.0929123249923863\\
-0.588768933904554	-0.0767075482809588\\
-0.588601058279857	-0.0598665505582033\\
-0.588516811427379	-0.0423620365861802\\
-0.588516811427379	-0.024180218404362\\
-0.588601058279857	-0.00532109601274876\\
-0.588768933904554	0.0142015426281321\\
-0.589019206675526	0.034360402280341\\
-0.589350040457008	0.0551152360808576\\
-0.589759008074301	0.0764136648516495\\
-0.590243109120826	0.0981922433886878\\
-0.590798791970728	0.120377752023341\\
-0.591421979835509	0.142888687163679\\
-0.592108100673548	0.165636920472569\\
-0.592852120733054	0.188529492906476\\
-0.593648581482385	0.211470507093524\\
-0.594491639656769	0.234363079527431\\
-0.595375110127641	0.257111312836321\\
-0.596292511280077	0.279622247976659\\
-0.59723711256541	0.301807756611312\\
-0.598201983880155	0.32358633514835\\
-0.599180046408981	0.344884763919142\\
-0.600164124558712	0.365639597719659\\
-0.601146998602336	0.385798457371868\\
-0.602121457646813	0.405321096012749\\
-0.603080352536066	0.424180218404362\\
-0.604016648301039	0.44236203658618\\
-0.604923475772015	0.459866550558203\\
-0.605794181974546	0.476707548280959\\
-0.606622378939251	0.492912324992386\\
-0.607401990567399	0.508521127555506\\
-0.608127297208441	0.52358633514835\\
-0.608792977622444	0.538171392974948\\
-0.609394148019571	0.552349520703932\\
-0.609926397890197	0.56620222192723\\
-0.610385822362771	0.579817624981977\\
-0.610769050852068	0.593288688911705\\
-0.611073271787607	0.606711311088295\\
-0.611296253240851	0.620182375018023\\
-0.611436359299834	0.63379777807277\\
-0.611492562071091	0.647650479296068\\
-0.611464449220841	0.661828607025052\\
-0.611352227000084	0.676413664851649\\
-0.611156718731439	0.691478872444494\\
-0.610879358768801	0.707087675007614\\
-0.610522181974154	0.723292451719041\\
-0.61008780878872	0.740133449441797\\
-0.609579426008	0.75763796341382\\
-0.609000763401725	0.775819781595638\\
-0.608356066350251	0.794678903987251\\
-0.607650064698174	0.814201542628132\\
-0.606887938053636	0.834360402280341\\
-0.606075277787933	0.855115236080857\\
-0.605218046014211	0.87641366485165\\
};
\addplot [color=darkorange, line width=2.0pt, forget plot]
  table[row sep=crcr]{%
0.594781953985789	-0.923586335148351\\
0.612025108296915	-0.92746114879694\\
0.629273284108932	-0.931134492676316\\
0.64652724291395	-0.934579417158512\\
0.663787741487384	-0.93777064842321\\
0.681055531270723	-0.940684773880094\\
0.698331357758639	-0.943300413936212\\
0.715615959890952	-0.945598378848181\\
0.732910069449942	-0.94756180950847\\
0.750214410463494	-0.949176301132889\\
0.767529698614574	-0.950430008941843\\
0.784856640657503	-0.951313735060006\\
0.802195933841523	-0.951820995996888\\
0.819548265342106	-0.95194807021321\\
0.836914311700487	-0.95169402542412\\
0.854294738271868	-0.951060725438937\\
0.871690198682739	-0.950052816487241\\
0.889101334297779	-0.948677693131635\\
0.906528773696743	-0.946945444017248\\
0.923973132161787	-0.94486877785601\\
0.941435011175633	-0.942462930188695\\
0.958914997930995	-0.939745551608799\\
0.976413664851649	-0.93673657826828\\
0.993931569125564	-0.933458085615216\\
1.01146925225045	-0.929934126436441\\
1.02902723959209	-0.926190554393348\\
1.04660603995587	-0.922254834345515\\
1.06420614517177	-0.91815584085371\\
1.08182802969323	-0.913923646340575\\
1.09947215021017	-0.909589300463139\\
1.11713894527653	-0.905184602315829\\
1.13482883495255	-0.900741867135191\\
1.15254222046216	-0.89629368921792\\
1.17027948386578	-0.891872702791559\\
1.18804098774865	-0.887511342592238\\
1.20582707492516	-0.883241605905987\\
1.22363806815918	-0.879094817819428\\
1.24147426990086	-0.875101401402074\\
1.25933596203993	-0.871290654506319\\
1.27722340567577	-0.867690534822617\\
1.2951368409045	-0.864327454766822\\
1.31307648662309	-0.861226087704482\\
1.33104254035089	-0.858409186933756\\
1.34903517806847	-0.855897418754962\\
1.36705455407413	-0.853709210851482\\
1.38510080085805	-0.851860617094352\\
1.4031740289942	-0.850365199762426\\
1.42127432705018	-0.849233930042196\\
1.43940176151498	-0.848475107537264\\
1.45755637674472	-0.848094299377979\\
1.47573819492654	-0.848094299377979\\
1.49394721606043	-0.848475107537264\\
1.51218341795927	-0.849233930042196\\
1.53044675626693	-0.850365199762426\\
1.54873716449442	-0.851860617094352\\
1.56705455407413	-0.853709210851482\\
1.5853988144321	-0.855897418754962\\
1.60376981307816	-0.858409186933756\\
1.622167395714	-0.861226087704482\\
1.64059138635904	-0.864327454766822\\
1.65904158749395	-0.867690534822617\\
1.67751778022174	-0.871290654506319\\
1.69601972444631	-0.875101401402074\\
1.71454715906827	-0.879094817819428\\
1.73309980219789	-0.883241605905987\\
1.75167735138502	-0.887511342592238\\
1.77027948386578	-0.891872702791559\\
1.7889058568258	-0.89629368921792\\
1.80755610767982	-0.900741867135191\\
1.82622985436744	-0.905184602315829\\
1.84492669566472	-0.909589300463139\\
1.86364621151141	-0.913923646340575\\
1.88238796335358	-0.91815584085371\\
1.90115149450132	-0.922254834345515\\
1.91993633050118	-0.926190554393348\\
1.93874197952317	-0.929934126436441\\
1.95756793276193	-0.933458085615216\\
1.97641366485165	-0.93673657826828\\
1.99527863429463	-0.939745551608799\\
2.01416228390291	-0.942462930188695\\
2.0330640412527	-0.94486877785601\\
2.05198331915129	-0.946945444017248\\
2.07091951611596	-0.948677693131635\\
2.08987201686456	-0.950052816487241\\
2.10884019281732	-0.951060725438937\\
2.12782340260958	-0.95169402542412\\
2.14682099261483	-0.95194807021321\\
2.16583229747789	-0.951820995996888\\
2.1848566406575	-0.951313735060006\\
2.20389333497821	-0.950430008941843\\
2.22294168319077	-0.949176301132889\\
2.24200097854085	-0.94756180950847\\
2.2610705053455	-0.945598378848181\\
2.28014953957682	-0.943300413936212\\
2.29923734945254	-0.940684773880094\\
2.31833319603284	-0.93777064842321\\
2.33743633382304	-0.934579417158512\\
2.35654601138166	-0.931134492676316\\
2.37566147193328	-0.92746114879694\\
2.39478195398579	-0.923586335148351\\
};
\addplot [color=darkorange, line width=2.0pt, forget plot]
  table[row sep=crcr]{%
0.594781953985789	0.87641366485165\\
0.612025108296915	0.87253885120306\\
0.629273284108932	0.868865507323684\\
0.64652724291395	0.865420582841488\\
0.663787741487384	0.86222935157679\\
0.681055531270723	0.859315226119906\\
0.698331357758639	0.856699586063788\\
0.715615959890952	0.854401621151819\\
0.732910069449942	0.85243819049153\\
0.750214410463494	0.850823698867111\\
0.767529698614574	0.849569991058157\\
0.784856640657503	0.848686264939994\\
0.802195933841523	0.848179004003112\\
0.819548265342106	0.84805192978679\\
0.836914311700487	0.84830597457588\\
0.854294738271868	0.848939274561063\\
0.871690198682739	0.849947183512759\\
0.889101334297779	0.851322306868365\\
0.906528773696743	0.853054555982752\\
0.923973132161787	0.85513122214399\\
0.941435011175633	0.857537069811305\\
0.958914997930995	0.860254448391201\\
0.976413664851649	0.86326342173172\\
0.993931569125564	0.866541914384784\\
1.01146925225045	0.870065873563559\\
1.02902723959209	0.873809445606652\\
1.04660603995587	0.877745165654485\\
1.06420614517177	0.88184415914629\\
1.08182802969323	0.886076353659425\\
1.09947215021017	0.890410699536861\\
1.11713894527653	0.894815397684171\\
1.13482883495255	0.899258132864809\\
1.15254222046216	0.90370631078208\\
1.17027948386578	0.908127297208441\\
1.18804098774865	0.912488657407762\\
1.20582707492516	0.916758394094013\\
1.22363806815918	0.920905182180572\\
1.24147426990086	0.924898598597926\\
1.25933596203993	0.928709345493681\\
1.27722340567577	0.932309465177383\\
1.2951368409045	0.935672545233178\\
1.31307648662309	0.938773912295518\\
1.33104254035089	0.941590813066244\\
1.34903517806847	0.944102581245038\\
1.36705455407413	0.946290789148518\\
1.38510080085805	0.948139382905648\\
1.4031740289942	0.949634800237574\\
1.42127432705018	0.950766069957804\\
1.43940176151498	0.951524892462736\\
1.45755637674472	0.951905700622021\\
1.47573819492654	0.951905700622021\\
1.49394721606043	0.951524892462736\\
1.51218341795927	0.950766069957804\\
1.53044675626693	0.949634800237574\\
1.54873716449442	0.948139382905648\\
1.56705455407413	0.946290789148518\\
1.5853988144321	0.944102581245038\\
1.60376981307816	0.941590813066244\\
1.622167395714	0.938773912295518\\
1.64059138635904	0.935672545233178\\
1.65904158749395	0.932309465177383\\
1.67751778022174	0.928709345493681\\
1.69601972444631	0.924898598597926\\
1.71454715906827	0.920905182180572\\
1.73309980219789	0.916758394094013\\
1.75167735138502	0.912488657407762\\
1.77027948386578	0.908127297208441\\
1.7889058568258	0.90370631078208\\
1.80755610767982	0.899258132864809\\
1.82622985436744	0.894815397684171\\
1.84492669566472	0.890410699536861\\
1.86364621151141	0.886076353659425\\
1.88238796335358	0.88184415914629\\
1.90115149450132	0.877745165654485\\
1.91993633050118	0.873809445606652\\
1.93874197952317	0.870065873563559\\
1.95756793276193	0.866541914384784\\
1.97641366485165	0.86326342173172\\
1.99527863429463	0.860254448391201\\
2.01416228390291	0.857537069811305\\
2.0330640412527	0.85513122214399\\
2.05198331915129	0.853054555982752\\
2.07091951611596	0.851322306868365\\
2.08987201686456	0.849947183512759\\
2.10884019281732	0.848939274561063\\
2.12782340260958	0.84830597457588\\
2.14682099261483	0.84805192978679\\
2.16583229747789	0.848179004003112\\
2.1848566406575	0.848686264939994\\
2.20389333497821	0.849569991058157\\
2.22294168319077	0.850823698867111\\
2.24200097854085	0.85243819049153\\
2.2610705053455	0.854401621151819\\
2.28014953957682	0.856699586063788\\
2.29923734945254	0.859315226119906\\
2.31833319603284	0.86222935157679\\
2.33743633382304	0.865420582841488\\
2.35654601138166	0.868865507323684\\
2.37566147193328	0.87253885120306\\
2.39478195398579	0.87641366485165\\
};
\addplot [color=darkorange, line width=2.0pt, forget plot]
  table[row sep=crcr]{%
0.611493733937614	-0.0333561150534745\\
0.631743136950461	-0.0388359290669784\\
0.651981479157201	-0.0440308218004521\\
0.672207083227343	-0.048902680724524\\
0.6924182822198	-0.053415763259729\\
0.712613420942475	-0.0575369590033111\\
0.732790857302256	-0.0612360326449603\\
0.752948963644332	-0.0644858457893241\\
0.773086128079734	-0.0672625560578834\\
0.793200755800009	-0.0695457920094755\\
0.813291270377975	-0.0713188025961514\\
0.833356115053474	-0.0725685800578809\\
0.853393754003088	-0.0732859553544811\\
0.87340267359278	-0.0734656654346313\\
0.893381383612433	-0.07310639184845\\
0.913328418491291	-0.0722107704203535\\
0.933242338493297	-0.0707853719112284\\
0.953121730891361	-0.0688406538117936\\
0.972965211119603	-0.0663908836208188\\
0.992771423902624	-0.0634540341710743\\
1.01253904436089	-0.0600516517709548\\
1.03226677909131	-0.0562086981291638\\
1.05195336722218	-0.0519533672221817\\
1.07159758144155	-0.0473168784480789\\
1.09119822899822	-0.0423332475842062\\
1.11075415267462	-0.0370390372291433\\
1.13026423173058	-0.0314730885597937\\
1.1497273828175	-0.0256762363716048\\
1.16914256086186	-0.0196910094925275\\
1.18850875991769	-0.0135613187686424\\
1.20782501398694	-0.0073321349105569\\
1.22709039780744	-0.00104915856406543\\
1.24630402760751	0.00524151497438745\\
1.26546506182679	0.0114937339376138\\
1.28457270180264	0.0176616286818885\\
1.30362619242152	0.0236999482113465\\
1.32262482273489	0.0295643921636455\\
1.34156792653907	0.0352119358212701\\
1.36045488291857	0.0406011457640186\\
1.37928511675257	0.0456924838468761\\
1.39805809918396	0.0504485972731286\\
1.4167733480507	0.0548345926345861\\
1.43543042827907	0.0588182919084068\\
1.45402895223854	0.062370468532394\\
1.47256858005788	0.0654650618267874\\
1.49104901990244	0.0680793681894385\\
1.50947002821209	0.0701942076616559\\
1.52783140989994	0.0717940646427072\\
1.54613301851141	0.0728672017206169\\
1.56437475634372	0.0734057457841391\\
1.58255657452554	0.0734057457841391\\
1.60067847305687	0.0728672017206169\\
1.61874050080903	0.0717940646427072\\
1.63674275548481	0.0701942076616559\\
1.6546853835388	0.0680793681894385\\
1.67256858005788	0.0654650618267874\\
1.69039258860217	0.062370468532394\\
1.70815770100634	0.0588182919084069\\
1.72586425714161	0.0548345926345861\\
1.7435126446385	0.0504485972731286\\
1.76110329857075	0.0456924838468761\\
1.77863670110038	0.0406011457640186\\
1.79611338108452	0.0352119358212701\\
1.81353391364398	0.0295643921636455\\
1.83089891969425	0.0236999482113465\\
1.848209065439	0.0176616286818886\\
1.86546506182679	0.0114937339376138\\
1.88266766397114	0.00524151497438754\\
1.89981767053471	-0.00104915856406538\\
1.91691592307785	-0.00733213491055682\\
1.93396330537223	-0.0135613187686424\\
1.95096074268005	-0.0196910094925274\\
1.96790920099932	-0.0256762363716048\\
1.98480968627604	-0.0314730885597936\\
2.00166324358371	-0.0370390372291432\\
2.01847095627095	-0.0423332475842062\\
2.03523394507791	-0.0473168784480788\\
2.05195336722218	-0.0519533672221817\\
2.06863041545495	-0.0562086981291638\\
2.08526631708816	-0.0600516517709548\\
2.10186233299353	-0.0634540341710743\\
2.11841975657415	-0.0663908836208188\\
2.13493991270954	-0.0688406538117935\\
2.15142415667512	-0.0707853719112284\\
2.16787387303675	-0.0722107704203535\\
2.18429047452152	-0.07310639184845\\
2.20067540086551	-0.0734656654346313\\
2.21703011763945	-0.0732859553544811\\
2.23335611505347	-0.0725685800578809\\
2.24965490674161	-0.0713188025961515\\
2.26592802852728	-0.0695457920094755\\
2.28217703717064	-0.0672625560578834\\
2.29840350909888	-0.0644858457893241\\
2.31460903912044	-0.0612360326449603\\
2.33079523912429	-0.0575369590033111\\
2.34696373676525	-0.053415763259729\\
2.36311617413643	-0.048902680724524\\
2.37925420642993	-0.0440308218004522\\
2.39537950058682	-0.0388359290669784\\
2.41149373393761	-0.0333561150534745\\
};
\addplot [color=darkorange, line width=2.0pt, forget plot]
  table[row sep=crcr]{%
0.594781953985789	-0.923586335148351\\
0.593924722212067	-0.908521127555506\\
0.593112061946365	-0.892912324992386\\
0.592349935301826	-0.876707548280959\\
0.591643933649749	-0.859866550558203\\
0.590999236598275	-0.84236203658618\\
0.590420573992	-0.824180218404362\\
0.58991219121128	-0.805321096012749\\
0.589477818025846	-0.785798457371868\\
0.589120641231199	-0.765639597719659\\
0.588843281268561	-0.744884763919142\\
0.588647772999916	-0.723586335148351\\
0.588535550779159	-0.701807756611312\\
0.588507437928909	-0.679622247976659\\
0.588563640700166	-0.657111312836321\\
0.588703746759149	-0.634363079527432\\
0.588926728212393	-0.611470507093524\\
0.589230949147932	-0.588529492906476\\
0.589614177637229	-0.565636920472569\\
0.590073602109803	-0.542888687163679\\
0.590605851980428	-0.520377752023341\\
0.591207022377556	-0.498192243388688\\
0.591872702791559	-0.47641366485165\\
0.592598009432601	-0.455115236080858\\
0.593377621060749	-0.434360402280341\\
0.594205818025454	-0.414201542628132\\
0.595076524227985	-0.394678903987251\\
0.595983351698961	-0.375819781595638\\
0.596919647463934	-0.35763796341382\\
0.597878542353187	-0.340133449441797\\
0.598853001397664	-0.323292451719041\\
0.599835875441288	-0.307087675007614\\
0.600819953591019	-0.291478872444494\\
0.601798016119845	-0.27641366485165\\
0.60276288743459	-0.261828607025052\\
0.603707488719923	-0.247650479296068\\
0.604624889872359	-0.23379777807277\\
0.605508360343231	-0.220182375018023\\
0.606351418517615	-0.206711311088295\\
0.607147879266945	-0.193288688911705\\
0.607891899326452	-0.179817624981977\\
0.60857802016449	-0.16620222192723\\
0.609201208029272	-0.152349520703932\\
0.609756890879174	-0.138171392974948\\
0.610240991925699	-0.123586335148351\\
0.610649959542992	-0.108521127555506\\
0.610980793324474	-0.0929123249923863\\
0.611231066095446	-0.0767075482809588\\
0.611398941720143	-0.0598665505582033\\
0.611483188572621	-0.0423620365861802\\
0.611483188572621	-0.024180218404362\\
0.611398941720143	-0.00532109601274876\\
0.611231066095446	0.0142015426281321\\
0.610980793324474	0.034360402280341\\
0.610649959542992	0.0551152360808576\\
0.610240991925699	0.0764136648516495\\
0.609756890879174	0.0981922433886878\\
0.609201208029272	0.120377752023341\\
0.60857802016449	0.142888687163679\\
0.607891899326452	0.165636920472569\\
0.607147879266945	0.188529492906476\\
0.606351418517615	0.211470507093524\\
0.605508360343231	0.234363079527431\\
0.604624889872359	0.257111312836321\\
0.603707488719923	0.279622247976659\\
0.60276288743459	0.301807756611312\\
0.601798016119845	0.32358633514835\\
0.600819953591019	0.344884763919142\\
0.599835875441288	0.365639597719659\\
0.598853001397664	0.385798457371868\\
0.597878542353187	0.405321096012749\\
0.596919647463934	0.424180218404362\\
0.595983351698961	0.44236203658618\\
0.595076524227985	0.459866550558203\\
0.594205818025454	0.476707548280959\\
0.593377621060749	0.492912324992386\\
0.592598009432601	0.508521127555506\\
0.591872702791559	0.52358633514835\\
0.591207022377556	0.538171392974948\\
0.590605851980428	0.552349520703932\\
0.590073602109803	0.56620222192723\\
0.589614177637229	0.579817624981977\\
0.589230949147932	0.593288688911705\\
0.588926728212393	0.606711311088295\\
0.588703746759149	0.620182375018023\\
0.588563640700166	0.63379777807277\\
0.588507437928909	0.647650479296068\\
0.588535550779159	0.661828607025052\\
0.588647772999916	0.676413664851649\\
0.588843281268561	0.691478872444494\\
0.589120641231199	0.707087675007614\\
0.589477818025846	0.723292451719041\\
0.58991219121128	0.740133449441797\\
0.590420573992	0.75763796341382\\
0.590999236598275	0.775819781595638\\
0.591643933649749	0.794678903987251\\
0.592349935301826	0.814201542628132\\
0.593112061946364	0.834360402280341\\
0.593924722212067	0.855115236080857\\
0.594781953985789	0.87641366485165\\
};
\addplot [color=darkorange, line width=2.0pt, forget plot]
  table[row sep=crcr]{%
2.39478195398579	-0.923586335148351\\
2.39392472221207	-0.908521127555506\\
2.39311206194636	-0.892912324992386\\
2.39234993530183	-0.876707548280959\\
2.39164393364975	-0.859866550558203\\
2.39099923659828	-0.84236203658618\\
2.390420573992	-0.824180218404362\\
2.38991219121128	-0.805321096012749\\
2.38947781802585	-0.785798457371868\\
2.3891206412312	-0.765639597719659\\
2.38884328126856	-0.744884763919142\\
2.38864777299992	-0.723586335148351\\
2.38853555077916	-0.701807756611312\\
2.38850743792891	-0.679622247976659\\
2.38856364070017	-0.657111312836321\\
2.38870374675915	-0.634363079527432\\
2.38892672821239	-0.611470507093524\\
2.38923094914793	-0.588529492906476\\
2.38961417763723	-0.565636920472569\\
2.3900736021098	-0.542888687163679\\
2.39060585198043	-0.520377752023341\\
2.39120702237756	-0.498192243388688\\
2.39187270279156	-0.47641366485165\\
2.3925980094326	-0.455115236080858\\
2.39337762106075	-0.434360402280341\\
2.39420581802545	-0.414201542628132\\
2.39507652422799	-0.394678903987251\\
2.39598335169896	-0.375819781595638\\
2.39691964746393	-0.35763796341382\\
2.39787854235319	-0.340133449441797\\
2.39885300139766	-0.323292451719041\\
2.39983587544129	-0.307087675007614\\
2.40081995359102	-0.291478872444494\\
2.40179801611985	-0.27641366485165\\
2.40276288743459	-0.261828607025052\\
2.40370748871992	-0.247650479296068\\
2.40462488987236	-0.23379777807277\\
2.40550836034323	-0.220182375018023\\
2.40635141851761	-0.206711311088295\\
2.40714787926695	-0.193288688911705\\
2.40789189932645	-0.179817624981977\\
2.40857802016449	-0.16620222192723\\
2.40920120802927	-0.152349520703932\\
2.40975689087917	-0.138171392974948\\
2.4102409919257	-0.123586335148351\\
2.41064995954299	-0.108521127555506\\
2.41098079332447	-0.0929123249923863\\
2.41123106609545	-0.0767075482809588\\
2.41139894172014	-0.0598665505582033\\
2.41148318857262	-0.0423620365861802\\
2.41148318857262	-0.024180218404362\\
2.41139894172014	-0.00532109601274876\\
2.41123106609545	0.0142015426281321\\
2.41098079332447	0.034360402280341\\
2.41064995954299	0.0551152360808576\\
2.4102409919257	0.0764136648516495\\
2.40975689087917	0.0981922433886878\\
2.40920120802927	0.120377752023341\\
2.40857802016449	0.142888687163679\\
2.40789189932645	0.165636920472569\\
2.40714787926695	0.188529492906476\\
2.40635141851761	0.211470507093524\\
2.40550836034323	0.234363079527431\\
2.40462488987236	0.257111312836321\\
2.40370748871992	0.279622247976659\\
2.40276288743459	0.301807756611312\\
2.40179801611985	0.32358633514835\\
2.40081995359102	0.344884763919142\\
2.39983587544129	0.365639597719659\\
2.39885300139766	0.385798457371868\\
2.39787854235319	0.405321096012749\\
2.39691964746393	0.424180218404362\\
2.39598335169896	0.44236203658618\\
2.39507652422799	0.459866550558203\\
2.39420581802545	0.476707548280959\\
2.39337762106075	0.492912324992386\\
2.3925980094326	0.508521127555506\\
2.39187270279156	0.52358633514835\\
2.39120702237756	0.538171392974948\\
2.39060585198043	0.552349520703932\\
2.3900736021098	0.56620222192723\\
2.38961417763723	0.579817624981977\\
2.38923094914793	0.593288688911705\\
2.38892672821239	0.606711311088295\\
2.38870374675915	0.620182375018023\\
2.38856364070017	0.63379777807277\\
2.38850743792891	0.647650479296068\\
2.38853555077916	0.661828607025052\\
2.38864777299992	0.676413664851649\\
2.38884328126856	0.691478872444494\\
2.3891206412312	0.707087675007614\\
2.38947781802585	0.723292451719041\\
2.38991219121128	0.740133449441797\\
2.390420573992	0.75763796341382\\
2.39099923659828	0.775819781595638\\
2.39164393364975	0.794678903987251\\
2.39234993530183	0.814201542628132\\
2.39311206194636	0.834360402280341\\
2.39392472221207	0.855115236080857\\
2.39478195398579	0.87641366485165\\
};
\addplot [color=darkorange, line width=2.0pt, forget plot]
  table[row sep=crcr]{%
1.46664388494653	-0.848046632777818\\
1.46116407093302	-0.822999889909775\\
1.45596917819955	-0.79915051784704\\
1.45109731927548	-0.776613891642752\\
1.44658423674027	-0.75547866263002\\
1.44246304099669	-0.735804953736428\\
1.43876396735504	-0.71762313555461\\
1.43551415421068	-0.700933208084565\\
1.43273744394212	-0.685704800733661\\
1.43045420799052	-0.671877790574312\\
1.42868119740385	-0.659363526273411\\
1.42743141994212	-0.648046632777818\\
1.42671404464552	-0.63778736001773\\
1.42653433456537	-0.628424427815094\\
1.42689360815155	-0.619778309082941\\
1.42778922957965	-0.611654884479162\\
1.42921462808877	-0.603849394116522\\
1.43115934618821	-0.596150605883478\\
1.43360911637918	-0.588345115520838\\
1.43654596582893	-0.580221690917059\\
1.43994834822905	-0.571575572184907\\
1.44379130187084	-0.56221263998227\\
1.44804663277782	-0.551953367222182\\
1.45268312155192	-0.540636473726589\\
1.45766675241579	-0.528122209425688\\
1.46296096277086	-0.514295199266339\\
1.46852691144021	-0.499066791915435\\
1.4743237636284	-0.48237686444539\\
1.48030899050747	-0.464195046263572\\
1.48643868123136	-0.44452133736998\\
1.49266786508944	-0.423386108357248\\
1.49895084143593	-0.40084948215296\\
1.50524151497439	-0.377000110090225\\
1.51149373393761	-0.351953367222182\\
1.51766162868189	-0.325849003618634\\
1.52369994821135	-0.298848299457634\\
1.52956439216365	-0.27113078182615\\
1.53521193582127	-0.242890570066292\\
1.54060114576402	-0.214332424065297\\
1.54569248384688	-0.185667575934703\\
1.55044859727313	-0.157109429933708\\
1.55483459263459	-0.12886921817385\\
1.55881829190841	-0.101151700542366\\
1.56237046853239	-0.0741509963813662\\
1.56546506182679	-0.0480466327778184\\
1.56807936818944	-0.0229998899097751\\
1.57019420766166	0.000849482152960324\\
1.57179406464271	0.023386108357248\\
1.57286720172062	0.0445213373699799\\
1.57340574578414	0.064195046263572\\
1.57340574578414	0.0823768644453902\\
1.57286720172062	0.0990667919154344\\
1.57179406464271	0.114295199266339\\
1.57019420766166	0.128122209425688\\
1.56807936818944	0.140636473726589\\
1.56546506182679	0.151953367222182\\
1.56237046853239	0.16221263998227\\
1.55881829190841	0.171575572184906\\
1.55483459263459	0.180221690917059\\
1.55044859727313	0.188345115520838\\
1.54569248384688	0.196150605883478\\
1.54060114576402	0.203849394116522\\
1.53521193582127	0.211654884479162\\
1.52956439216365	0.219778309082941\\
1.52369994821135	0.228424427815094\\
1.51766162868189	0.23778736001773\\
1.51149373393761	0.248046632777818\\
1.50524151497439	0.259363526273411\\
1.49895084143593	0.271877790574312\\
1.49266786508944	0.285704800733661\\
1.48643868123136	0.300933208084565\\
1.48030899050747	0.31762313555461\\
1.4743237636284	0.335804953736428\\
1.46852691144021	0.35547866263002\\
1.46296096277086	0.376613891642752\\
1.45766675241579	0.399150517847039\\
1.45268312155192	0.422999889909775\\
1.44804663277782	0.448046632777818\\
1.44379130187084	0.474150996381366\\
1.43994834822905	0.501151700542366\\
1.43654596582893	0.52886921817385\\
1.43360911637918	0.557109429933708\\
1.43115934618821	0.585667575934703\\
1.42921462808877	0.614332424065296\\
1.42778922957965	0.642890570066292\\
1.42689360815155	0.67113078182615\\
1.42653433456537	0.698848299457633\\
1.42671404464552	0.725849003618634\\
1.42743141994212	0.751953367222182\\
1.42868119740385	0.777000110090225\\
1.43045420799052	0.80084948215296\\
1.43273744394212	0.823386108357248\\
1.43551415421068	0.84452133736998\\
1.43876396735504	0.864195046263572\\
1.44246304099669	0.88237686444539\\
1.44658423674027	0.899066791915434\\
1.45109731927548	0.914295199266339\\
1.45596917819955	0.928122209425688\\
1.46116407093302	0.940636473726589\\
1.46664388494653	0.951953367222182\\
};
\end{axis}
\end{tikzpicture}%